\documentclass[12pt]{amsart}

\usepackage{graphicx}
\usepackage{ulem}
\usepackage{amssymb}
\usepackage{setspace} 
\usepackage{color, overpic}
\usepackage{hyperref}

\newtheorem{lemma}{Lemma}[section]
\newtheorem{proposition}[lemma]{Proposition}

\newtheorem{theorem}[lemma]{Theorem}
\newtheorem{claim}[lemma]{Claim}

\newtheorem{conjecture}[lemma]{Conjecture}
\newtheorem{corollary}[lemma]{Corollary}
\newtheorem{definition}[lemma]{Definition}

\newcommand{\Label}[1]{\label{#1}\note{#1}}
\newcommand{\Gm}{\Gamma}
\newcommand{\gm}{\gamma}

\newcommand{\Og}{\Omega}
\newcommand{\og}{\omega}
\newcommand{\Ld}{\Lambda}
\newcommand{\ld}{\lambda}
\newcommand{\ldp}{{\lambda'}}
\newcommand{\ap}{{\alpha}}
\newcommand{\kp}{\kappa}
\newcommand{\e}{\epsilon}

\newcommand{\dl}{\delta}
\newcommand{\bt}{\beta}
\newcommand{\Dl}{\Delta}
\newcommand{\zt}{\zeta}
\newcommand{\mc}{\mathcal}
\newcommand{\rs}{\hat{\mathbb{C}}}
\newcommand{\R}{\mathbb{R}}
\newcommand{\C}{\mathbb{C}}
\newcommand{\N}{\mathbb{N}}
\newcommand{\h}{\mathbb{H}}
\newcommand{\Z}{\mathbb{Z}}
\newcommand{\s}{\mathbb{S}}
\newcommand{\hb}{\overline{\mathbb{H}^3}}
\newcommand{\sm}{\setminus}
\newcommand{\psl}{ {\rm PSL}({\rm 2},\mathbb{C})}
\newcommand{\sq}{\sqcup}
\newcommand{\st}{\subset}
\newcommand{\lf}{\lfloor}
\newcommand{\rf}{\rfloor}
\newcommand{\td}{\tilde}
\newcommand{\dt}{\ldots}
\newcommand{\pt}{\partial}
\newcommand{\D}{\mathbb{D}}
\newcommand{\po}{\pi_1}
\newcommand{\im}{{\rm Im}}

\newcommand{\cn}{\colon}
 
\newcommand{\nin}{\noindent}

\newcommand{\ol}{\overline}

\newcommand{\iv}{^{-1}}

\newcommand{\pg}{\Psi_{\Gamma}}

\newcommand{\sr}{^\star}
\newcommand{\id}{{id}}
\newcommand{\mR}{\mathcal{R}}
\newcommand{\mC}{\mathcal{C}}

\newcommand{\ps}{projective structure}

\newcommand{\fb}{\bar{f}}

\newcommand{\Fb}{{\bar{F}}}

\newcommand{\Lt}{{\tilde{L}}}
\newcommand{\Nt}{\tilde{N}}
\newcommand{\Mt}{\tilde{M}}

\newcommand{\et}{\tilde{\epsilon}}
\newcommand{\elt}{\tilde{\ell}}
\newcommand{\St}{{\tilde{S}}}
\newcommand{\Ft}{{\tilde{F}}}
\newcommand{\lt}{\tilde{\ell}}
\newcommand{\Ht}{\tilde{H}}

\newcommand{\Gt}{\tilde{G}}
\newcommand{\kpt}{\tilde{\kappa}}
\newcommand{\at}{^\ast}
\newcommand{\cc}{\circ}
\newcommand{\mt}{{\tilde{\mu}}}
\newcommand{\Dlt}{\tilde{\Delta}}
\newcommand{\ett}{\tilde{\eta}}
\newcommand{\Ct}{\tilde{C}}
\newcommand{\Tt}{\tilde{T}}
\newcommand{\Qt}{\tilde{Q}}
\newcommand{\Pt}{\tilde{P}}

\newcommand{\apt}{\tilde{\alpha}}
\newcommand{\rt}{\tilde{\rho}}
\newcommand{\Fh}{{\hat{F}}}

\newcommand{\Fc}{{\check{F}}}

\newcommand{\np}{N^+}

\definecolor{darkblue}{cmyk}{1,0, 0,.7}
\newcommand\note[1]{\mbox{}\marginpar{\scriptsize\hspace{0pt}\color{darkblue}{#1}}}

\newcommand{\bb}{\begin}
\newcommand{\ee}{\end}

\usepackage{overpic}
 \title[\today]{Complex Projective Structures with Schottky holonomy}

\author{Shinpei Baba}

\address{Department of Mathematics, California Institute of Technology, Pasadena, CA 91125, USA}
\email{shinpei@caltech.edu}
\date{\today}

\begin{document}

\maketitle

\begin{abstract}
Let $S$ be a closed orientable surface of genus at least two.
Let $\Gm$ be a Schottky group whose rank is equal to the genus of $S$, and $\Og$ be the domain of discontinuity of $\Gm$. 
Pick an arbitrary epimorphism  $\rho\colon \po(S) \to \Gm$.
Then $\Og/\Gm$ is a surface homeomorphic to $S$ carrying  a (complex) projective structure  with holonomy $\rho$.
We show that every projective structure with holonomy $\rho$ is obtained by (2$\pi$-)grafting  $\Og/\Gm$ along a multiloop on $S$.
\end{abstract}

\section{Introduction}
Throughout this paper, let $S$ denote a closed orientable surface of genus $g > 1$. 
Then Gallo, Kapovich and Marden (\cite{Gallo-Kapovich-Marden}) proved that a homomorphism  $\rho\colon \po(S) \to \psl$ is a holonomy representation of some projective structure on $S$ if and only if $\rho$ satisfies: 

\begin{itemize}
\item[(i)] the image of $\rho$ is non-elementary and
\item[(ii)] $\rho$  lifts to a homomorphism from $\po(S)$ to ${\rm SL}(2, \C)$. 
\end{itemize}
Their proof of this theorem implies that there are infinitely many distinct projective structures with fixed holonomy $\rho$ satisfying (i) and (ii) (see also \cite{Baba-10}). 
Then it is a natural question to ask for a characterization of  projective structures with fixed holonomy $\rho$.
In literature,  this question goes back to a paper by Hubbard (\cite{Hubbard-81}; see also \cite[\S 12]{Gallo-Kapovich-Marden}).
In this paper, we give a characterization of projective structures with a certain type of discrete  nonfaithful holonomy representation, using a certain surgery operation.

Basic examples of projective structures arise from Kleinian groups with nonempty domain of discontinuity.
Let $\Gm$ be a  (not necessarily classical) {\it Schottky group} of rank $g > 1$, which can be defined as a subgroup of $\psl$ isomorphic to a free group of rank $g$ consisting only of loxodromic elements (see \cite[pp 75]{Kapovich-01}, \cite{Matsuzaki-Taniguchi-98} about Schottky groups).
Let $\Og \st \rs$ denote the domain of discontinuity of $\Gm$. 
Let $\rho: \po(S) \to \Gm$ be an epimorphism, which we call a  {\it Schottky holonomy (representation)}. 
Then $\Og/\Gm$ enjoys a projective structure on $S$ with holonomy $\rho$  ({\it uniformizable structure}), when  $\Og/\Gm$ and $S$ are  homeomorphically identified in an appropriate way.

A ($2\pi$-)grafting is an operation that transforms a projective structure into another projective structure on the same surface without changing holonomy representation (\S  \ref{graft}). 
It is a surgery operation that inserts a projective cylinder  along a certain type of loop on a projective surface, called an admissible loop.
If there is a multiloop consisting of disjoint admissible loops on a  projective surface, a grafting operation can be done simultaneously along all loops. 

A {\it Schottky structure} is a projective structure on $S$ with Schottky holonomy. 
The goal of this paper is to show:
\vskip 2mm
\nin{\bf Theorem \ref{main}.}
{\it Every projective structure on $S$ with Schottky holonomy $\rho$ is obtained by grafting $\Og/\Gm$ along a multiloop on $S$. }
\vskip 2mm

{\it Remark:} 
By a quasiconformal map of $\rs$, the proof of Theorem \ref{main} is easily reduced to the case when $\Gm$ is a real Schottky group, i.e. the limit set of $\Gm$ lies in the equator $\R \cup \{\infty\}$ of $\rs$ (as in  the proof of Theorem \ref{fuchsian} below from \cite{Goldman-87}). 

A projective structure is called {\it minimal} if it can $not$ be obtained by grafting another projective structure.
If we graft an arbitrary projective structure on $S$, then the developing map of the new projective structure is necessarily surjective onto $\rs$\,; thus $\Og/\Gm$ is a minimal structure.
By Theorem \ref{main}, moreover,  $\Og/\Gm$ is the unique  minimal structure among all projective structures on $S$ with holonomy $\rho$, up to a change of marking on $S$.
There is an incorrect theorem in the literature, implying that there are other minimal structures with holonomy $\rho$, i.e. they are not $\Og/\Gm$ (Theorem 3.7.3, Example 3.7.6 in \cite{Tan-88})

Recently Thompson \cite{Thompson_10} showed that 
 if two uniformizable projective structures with holonomy $\rho$ differ by a Dehn twist, then they can be grafted to a  common projective structure.
Thus, since such Dehn twists generate the subgroup of the mapping class group of $S$ that preserves the holonomy $\rho$ and the orientation of the surface (\cite{Luft-78}), combining with Theorem \ref{main},  the set of projective structures with Schottky holonomy $\rho$ is generated by graftings.
This answers a special case of Problem 12.1.2 in \cite{Gallo-Kapovich-Marden} in the affirmative.

Theorem \ref{main} is an analog to the case of a quasifuchsian holonomy:
Let $\Gm'$ be a quasifuchsian group and let $\Og^+ $ and $\Og^-$ be the connected components of the domain of discontinuity of $\Gm'$.
Then $\Og^+/\Gm'$ and $\Og^-/\Gm'$ are  projective surfaces on $S$ with the same holonomy  $\rho'$, where $\rho'$ is an isomorphism from $\po(S)$ onto $\Gm'$.

\begin{theorem}[Goldman \cite{Goldman-87}]\label{fuchsian}
Every projective structure with quasifuchsian holonomy $\rho'$ is obtained by grafting either $\Og^+/\Gm'$ or $\Og^-/\Gm'$ along a  multiloop. 
\end{theorem}


In Theorem \ref{fuchsian}, for a given projective structure, the choice of the multiloop, up to an isotopy, and the choice of  $\Og^+/\Gm'$ or $\Og^-/\Gm'$ are unique. 
On the other hand, in Theorem \ref{main},  there are typically infinitely many pairs of a multiloop and a marking on $\Og/\Gm$  for a single Schottky structure (see the remark preceding Lemma \ref{lemma}; see also \cite{Thompson_10}). 

Thus we discuss an approach to describe Schottky structures in a unique manner.
Fix an appropriate marking (and, therefore, an orientation) on $\Og/\Gm$ so that it has the Schottky holonomy $\rho$.
Let $\mc{P}_\rho$ be the collection of all  projective structures on $S$ of the fixed orientation with holonomy $\rho$.
Let $\phi\colon S \to S$ be a homeomorphism.
Then the {\it support} of $\phi$ is the minimal $\po$-injective subsurface $R$ of $S$, defined up to an isotopy, such that the restriction of $\phi$ to $S \sm R$  is isotopic to the identity. 
The homeomorphism $\phi$ induces an isomorphism $\phi\at\colon \po(S) \to \po(S)$. 
Let $Stab_{\rho}$ denote the subgroup of the mapping class group of $S$ consisting of orientation-preserving mapping classes $[\phi]$ such that $\rho \cc \phi\at = \rho$.
As mentioned above,  $Stab_{\rho}$ is  generated by Dehn twists along simple  loops on $S$ representing elements in $\ker(\rho)$. 
Let $\mc{AML}_\rho(S)$ denote the set of isotopy classes of multiloops on $S$ consisting of disjoint admissible loops on $\Og/\Gm$.

\begin{conjecture} Every $C \in \mc{P}_\rho$ can be obtained by changing the marking of $\Og/\Gm$ by a unique $\phi \in Stab_\rho$ and grafting $\Og/\Gm$ along a unique $L \in \mc{AML}_\rho(S)$ such that the multiloop $L$ and the support of $\phi$ are disjoint:
$$\mc{P}_{\rho}(S) \cong \{(\phi, L) \in Stab_{\rho} \times \mc{AML}_{\rho}(S)~ |~ {\it Supp}(\phi) \cap L = \emptyset \}.$$
\end{conjecture}

Theorem \ref{main} ensures that, for every $C \in \mc{P}_{\rho}$, there is a corresponding pair  $(\phi, L) \in Stab_{\rho} \times \mc{AML}_{\rho}(S)$; however ${\it Supp}(\phi) \cap L$ may be nonempty. 
The conjecture asserts that this intersection can be uniquely ``resolved''. 

{\it Outline of the proof of Theorem \ref{main}.} 
Let $C$ be  a projective structure on $S$ with Schottky holonomy $\rho$.
The main step is to reduce the theorem to the case of ``genus zero'': 
We decompose $C$, by cutting it along a multiloop $M$, into certain (very simple) projective structures on holed-spheres, called  {\it good holed spheres}  (\S \ref{AlmostGood}):  
Namely each component $F$ of $C \sm M$ has trivial holonomy and the boundary components of $F$, via its developing map,  cover disjoint loops that bound disjoint disks on $\rs$.
Then we show that the the good holed sphere $F$ is obtained  by grafting a holed sphere embedded in $\rs$ along a multiarc  (Proposition \ref{GoodHoledSphere}).

The multiloop on $S$ in Theorem \ref{main} is realized as the union of such multiarcs over all components of $C \sm M$, and the uniformizable structure $\Og/\Gm$ the union of the corresponding embedded holes spheres. 

It is straightforward to find a multiloop realizing  the above decomposition on $C$ except that different boundary components  of $F$ may cover the same loops on $\rs$ (see \S \ref{IntoAlmostGood}).
Then the main work is  to isotope this imperfect multiloop on $C$ to a desired one.
To realize this isotopy, we construct a certain embedding a handlebody bounded by $S$ to itself associated with the projective structure $C$ (\S \ref{CellularStructures}). 

{\it Acknowledgements:}  I am grateful to Ken Bromberg, Ulrich Oertel, and Brian Osserman for helpful discussions. 
 I  thank Misha Kapovich for his valuable comments.  
In particular, the proof of Proposition \ref{grafting} in this paper is developed based on his suggestion.  
I also thank the referee of the paper.



\section{conventions and terminology}

In this paper, unless otherwise stated, 

\begin{itemize}
\item a surface is connected and orientable;
\item a loop and arc are simple; 
\item a component is a connected component.
\end{itemize}

As usual,
\begin{itemize}
\item a loop on a surface $F$ may be regarded as an element of $\po(F)$. 
\end{itemize}

Here is some terminology often used in this paper: 
\begin{itemize}
\item 
A {\it multiloop} is the union of isolated and disjoint simple closed curves embedded in a surface; similarly a {\it multiarc} is the union of isolated and disjoint simple curves property embedded in a surface. 
\item Let $F$ be a surface and let $F_1$ and $F_2$ be subsurfaces of $F$.
Then we say that $F_1$ and $F_2$ are {\it adjacent} if $F_1$ and $F_2$ have disjoint interiors and share a boundary component. 
\end{itemize}
 
\section{Preliminaries}  



\subsection{Projective structures}
Let $F$ be a connected orientable surface possibly with boundary.
A {\it (complex) projective structure} is a $(\rs, \psl)$-structure, i.e.\  an atlas modeled on  the Riemann sphere $\rs$ with its transition maps lying in $\psl$.
Let $\Fb$ denote the universal cover of $F$.
It is well-known that a {\rm projective structure} is equivalently defined as a pair $(f, \rho)$ consisting of
\begin{itemize}
\item  a topological immersion $\fb\colon \Fb \to \rs$ (i.e.\ a locally injective continuous map) and
\item a homomorphism $\rho\colon \po(F) \to \psl$ 
\end{itemize}
such that $\fb$ is $\rho$-equivariant, i.e.\  $\fb \cc \ap = \rho(\ap) \cc \fb$ for all $\ap \in \po(F)$ (see \cite[\S 3.4]{Thurston-97}).  
The immersion  $\fb$ is called the {\it (maximal) developing map} and the homomorphism $\rho$ is called the {\it holonomy (representation)} of the projective structure.
A projective structure is defined up to an isotopy of $F$ and an element of $\psl$, that is,  $(\fb, \rho) \sim (\gm \cc \fb, ~\gm \cc \rho \cc \gm\iv)$ for all $\gm \in \psl$.  

If $C$ is a projective structure on $F$, the pair $(F, C)$ is called a {\it projective surface}.
As usual, we will often conflate the projective structure $C$ and the projective surface $(F, C)$.

\subsubsection{Minimal developing maps}\label{MinDev}
Let $C = (\bar{f}, \rho)$ be a \ps\ on $F$. 
Then the short exact sequence
$$ 1 \to \ker(\rho) \to \pi_1(F) \xrightarrow{\rho} {\rm Im}(\rho) \to 1$$  induces an isomorphism $\td{\rho}\colon \pi_1(F)/\ker(\rho) \to {\rm Im}(\rho)$.
Let $\td{F} = \bar{F} / \ker(\rho)$ and let $\phi: \Fb \to \Ft$ denote the corresponding covering map;  
then $\td{F}$ is the cover of $F$ whose deck transformation group is $\pi_1(F)/\ker(\rho)$.
Then the maximal developing map $\fb\cn \Fb \to \rs$ canonically descends  to a locally injective map $f\colon \td{F} \to \rs$ that is $\td{\rho}$-equivariant, so that $\fb = \phi \cc f$. 
 We call this map $f$ the {\it minimal developing map} of $C$ and the cover $\td{F}$  the {\it developable  cover} of $F$ associated with $\rho$ (c.f. \cite{Kullkani-Pinkall-98}).

Conversely the pair $(f, \rho)$ of the minimal developing map and the holonomy representation of $C$, recovers the maximal developing map $\fb$. 

Since Schottky representations have a large kernel, it is much easier to work with minimal developing maps. 
Thus,  for the remainder of the 
paper, unless otherwise stated, developing maps are always minimal developing maps; in particular,  we use a pair $(f, \rho)$ of a minimal developing map and a holonomy representation to represent  a projective structure. 
We denote the minimal developing map of a projective structure $C$ by $dev(C)$.

\subsubsection{Restriction of projective structures to subsurfaces} 
Let $C = (f, \rho)$ be a projective structure on a surface $F$.
Let $E$ be  a subsurface of $F$.
The {\it restriction} of $C$ to $E$ is the projective structure on $E$ given by restricting the atlas of $C$ on $F$ to $E$, which we denote by $C|E$.
We can equivalently define the restriction $C|E$ as a pair of the (minimal) developing map and a holonomy representation as follows: 
The inclusion $E \st F$ induces a homomorphism $i\at\colon \po(E) \to \po(F)$.
Let $\td{F}$ be the developable cover of $F$ associated with $\rho$ and let $\td{E}$ be a lift of $E$ to $\td{F}$ invariant under $\po(E)$. 
Then the restriction of $C$ to $E$ is  the pair of $f|_{\td{E}}\cn \td{E} \to \rs$ and  $\rho \cc i\at\cn \po(E) \to \psl$.

\subsubsection{Embedded projective structures}\label{basic}

A projective surface $C$ is {\it uniformizable} (by $\Gm$) if $C$ is isomorphic to a component of the ideal boundary of a hyperbolic 3-manifold $\h^3/\Gm$, where $\Gm$ is a Kleinian group. 
More generally a projective surface $C$ is {\it embedded} if its (minimal) developing map  is an embedding onto a subset of $\rs$.
Then it is easy to show: 
\begin{lemma}\label{basicSchottky}
Let $C$ be a projective structure on $S$ with Schottky holonomy $\rho\colon \po(S) \to \Gm \st \psl$, where $\Gm$ is a Schottky group.
Then $C$ is embedded  if and only if it is uniformizable by $\Gm$, i.e. $dev(C)$ is an embedding onto  the domain of discontinuity of $\Gm$.
\end{lemma}
\begin{proof}
Suppose that $C  = (f, \rho)$ is embedded. 
Then we first show that the developing map $f\cn \St \to \rs$ is an embedding into the complement of the limit set $\Ld$ of $\Gm$.
Suppose that there is  $x \in \St$ such that $f(x) \in \Ld$.
Then there is a compact neighborhood of $x$ in $\St$ that $f$ takes homeomorphically onto its image. 
On the other hand, since $f$ is equivariant under the isomorphism $\td{\rho}\cn \po(S)/ \po(\St) \to \Gm$, for any $y \in \St$, then $f(\Gm \cdot y)$ accumulates to the limit point $f(x)$.
Then $f$ is not injective, which is a contradiction. 

Then $f$ is injective, and thus $C$ embeds into $\Og/\Gm$. 
Since $C$ and $\Og/\Gm$ are homeomorphic to $S$,  thus $f$ must be an embedding onto $\Og$.
\end{proof}

\subsection{Good and almost good projective structures}\label{AlmostGood}

Let $F$ be a surface obtained from $\s^2$ by removing disjoint points and closed (topological) disks, i.e $F$ is a genus-zero surface.
In addition we assume that the Euler characteristic of  $F$ is non-positive.
Then let $M$ be the union of the boundary circles of  the removed disks, and let $P$ be the union of the removed points. 

A projective structure $C = (f, \rho)$ on the surface $F$ is \textit{almost good} if: 
\bb{itemize}
\item[(i)] $\rho\colon \po(F) \to \psl$ is the trivial representation (so that the domain of $f$ is $F$),
\ee{itemize}
and there is a genus-zero surface $R$ embedded in $\rs$ such that 
\begin{itemize}
\item[(ii)]   
$f$ continuously extends to $P$, taking $P$ into the set of punctures of $R$
, and
\item[(iii)] each component of $M$, via $f$,  covers a boundary component of $R$. 
\end{itemize}

The surface $R$ is called a {\it support} of the almost good projective structure $C$. 
Then there is a unique component of   $\rs \sm f(P \cup M)$ in $\rs$ that is also a support of $C$, and it is contained in $R$.
Thus we call the component the {\it full support} of $C$ and denote it by ${\it Supp}(C)$ or, alternatively,  ${\it Supp}_f(F)$ to indicate the associated developing map. 
Clearly the developing map $f\cn F \to \rs$ has the lifting property along every path  disjoint from the boundary of ${\it Supp}(C)$.

Moreover a projective structure $C = (f, \rho)$ on $F$ is ${\it good}$ if it is almost good and in addition\\

\begin{itemize}
\item[(vi)] $f$ yields bijections between the punctures of $F$ and of ${\it Supp}(C)$ 
and the boundary components of $F$ and  of ${\it Supp}(C)$.
\end{itemize}

Then, if $C$ is a good structure on $F$, there is a homeomorphism between $F$ to ${\it Supp}(C)$ realizing the bijections in (vi).
Obiously the full  support enjoys a good embedded projective structure on $F$ whose developing map is such a homeomorphism; this is called  an {\it embedded structure associated with} the good structure $C$. 
Then this structure associated structure is unique up to an element of the pure mapping class group of $F$.

Let us return to the case that $C = (f, \rho)$ is an almost good projective structure on $F$ supported on $R$. 
Then we can naturally extend  $C$ to an almost good structure on a punctured sphere homotopy equivalent to $F$, so that the  developing map of this extension is a branched covering of $\s^2$ to $\rs$ as follows:  
Along each boundary component $l$ of $F$,  attach a disk $D$ minus a point $p$.
The loop $f(l)$ bounds a unique disk $D'$ in $\rs$ whose interior is disjoint from the support $R$.
Pick a point $p'$ in the interior $D'$.
Then continuously extend the developing map $f$ to $D$ so that $f| D$ is a branched covering map of $D$ onto $D'$ such that $p$ is the unique ramification point and $p'$ is the unique branched point.

\subsection{Grafting}\label{graft}
Grafting was initially developed as an operation that transforms a hyperbolic strucuture to a projective structure on a surface by inserting a flat affine cylinder along a circular loop (Maskit \cite{Maskit-69}, Hejhal \cite{Hejhal-75}, Sullivan-Thurston \cite{Sullivan-Thurston-83}, Kamishima-Tan \cite{Kamishima-Tan-92}). 
Goldman used a variation of this grafting operation, which is often called {\it $2\pi$-grafting}: 
His grafting is, to an extent, more special  since it preserves the holonomy representation of a projective structure, but, to another extent, more general since it can be done along a non-circular loop (\cite{Goldman-87}). 

In this paper, we basically follow the definition given by Goldman.
Yet, following out convention,  we define grafting with respect to minimal developing maps.
In addition, we define grafting  also along an  arc properly embedded in a certain kind of projective surface so that grafting is compatible with identifying boundary components of projective surfaces. 
For simplicity, we graft  only embedded projective structures.

Let $R$ be the complement of a union of (possibly infinitely many) disjoint points and closed (topological) disks in $\rs$.
Then $R$ enjoys an embedded  projective structure with trivial holonomy. 
Let $F$ be the topological surface homeomorphic to $R$.
Let $\ap$ be a simple arc property embedded in $R$ that does not bound a disk in $C$\,; 
then $\ap$ connects different components $X, Y$ of $\rs \sm R$, which are a disk or a point. 
Since $X$ and  $Y$ are disjoint, $\rs \sm (X \cup Y) =: A$ is homeomorphic to a cylinder. 
Then $A$ carries an embedded  projective structure with trivial holonomy and $\ap$ is also canonically embedded in $A$.   
 Since $\ap$ is property embedded in both projective surfaces $R$ and $A$, we can cut and past them along $\ap$ to make a new  projective structure on $F$. Namely we can pair up  the boundary segments of  $R \sm \ap$ and $A \sm \ap$ coming from $\ap$ in an alternating fashion  (as in  Figure \ref{DefGrafting}).  
 
 \begin{figure}[h]
\begin{overpic}[scale=1
]{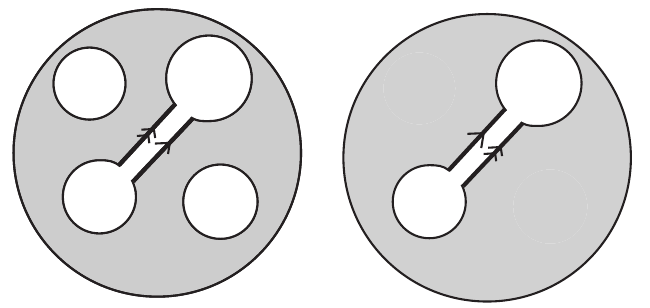}
      \put(70,5){$A \sm \ap$}
      \put(20,5){$R \sm \ap$}
      \end{overpic}
\caption{Cut-and-past for grafting of a four-holed sphere.}\label{DefGrafting}
\end{figure}

This operation is called the {\it grafting (operation)} on $R$ along $\ap$, and the resulting structure is denoted by $Gr_\ap(R)$.
Then $Gr_\ap(R)$ is also a good structure fully  supported $R$: 
If an endpoint of $\ap$ is a puncture, near this point, the new developing map is a branched covering of degree 2; similarly
if an endpoint of $\ap$ is on a boundary component of $R$,  this component is covered by the corresponding boundary component of $F$ via the new developing map and its covering degree is 2. 
Clearly $Gr_\ap(R)$ is non-embedded.
In addition when we isotopy $\ap$ on $R$ properly, the resulting structure $Gr_\ap(R)$ does $not$ change. 

If there is a multiarc $M$ on $R$ consisting of arcs that, as above, connect different punctures or boundary components of $R$, we  can simultaneously graft along all arcs of $M$ and obtain a good projective structure with full  support $R$. 
Similarly  denote this resulting structure by $Gr_M(R)$.
 
Next we define grafting along a loop on an embedded  projective structure $C = (f, \rho)$  on a surface $F$.
We set $f\cn \Ft \to \rs$, where $\Ft$ is the developable  cover of $F$ associated with holonomy $\rho$.
Accordingly, lifting $C$, we obtain a projective structure $\Ct$  on $\Ft$.
Then the holonomy of $\Ct$ is trivial and  $dev(\Ct)$ is an embedding of $\Ft$ into $\rs$ (see \S \ref{MinDev}). 
A loop $\ell $ on the embedded  projective surface $(F, C)$ is called {\it admissible} if $\rho(\ell)$ is loxodromic.
If $\ell$ is admissible, letting $\lt$ be a lift of $\ell$ to $\Ft$ invariant under $\ell \in \po(F)$, then $f(\lt)$ is a simple arc on $\rs$ invariant by $\rho(\ell)$.
Thus the end points of $f(\lt)$ are the fixed  points of the loxodromic element $\rho(\ell)$\,;  denote this fixed point set by  $Fix(\rho(\ell))$.
If $\ell$ is an admissible loop,  $\rs \sm Fix(\rho(\ell))$ is a twice-punctured sphere and $f(\lt)$ is a properly embedded arc in $\rs \sm Fix(\rho(\ell))$ connecting the punctures.  The infinite cyclic group $\langle \rho(\ell) \rangle$ acts on $\rs \sm Fix(\rho(\ell))$, preserving $f(\lt)$, and its quotient is thus a (Hopf) torus $T$ containing a copy of $\ell$.
Since $\ell$ is embedded in both $C$ and $T$, we can similarly obtain a new projective structure on $F$ by a canonical cut-and-past operation along $\ell$: Insert the cylinder $T \sm \ell$ between the boundary components of $C \sm \ell$.
This operation is called a {\it grafting} on $C$ along $\ell$, and we denote the resulting structure by $Gr_\ell(C)$. 

The {\it total lift} ${\bf l}$ of $\ell$ on $F$ to $\Ft$ is the union of all lifts of $\ell $ to $\Ft$.
Then the grafting of $C$ along $\ell$ lifts the grafting of $\Ct$ along ${\bf l}$.
Thus we see that $\Ct \sm {\bf l}$ is isomorphically embedded in $Gr_{\bf l}(\Ct)$ and, with respect to this embedding, the developing map of $\Ct \sm {\bf l}$ does not change under the grafting. 
Thus the maximal developing map of $Gr_\ell(C)$ is also $\rho$-equivariant, and the grafting preserves the holonomy $\rho$. 

If there is a multiloop on $(F,C)$ consisting of disjoint admissible loops, similarly we can graft $C$ simultaneously along all loops of the multiloop.

\subsection{Hurwitz spaces}\label{Hurwitz}

Let $d_i ~(i= 1,2, \dt, k)$ be integers greater than $1$.
Consider a pair consisting of a set of $k$ distinct ordered points $P_i \,(i =1,2, \dt, k)$ on $\rs$ and a rational function $\tau\colon \rs \to \rs$ such that $P_i$ are the ramification points of $\tau$ with ramification index $d_i$ (i.e.\ $\tau$ has zero of order $d_i$ at $P_i$) and that $\tau(P_i)$ are distinct points on $\rs$. 
Let $\mR = \mR(d_1, d_2, \dt, d_k)$ denote the space of all such pairs $(P_i, \tau)$.
Then $\psl$ acts on $\mR$ by postcomposition. 
The quotient space  $\mR / \psl$ is called a {\it Hurwitz space}, which we denote by $\mc{H} = \mc{H}(d_1, d_2, \dt, d_k)$.
It is well known that the map $\tau \mapsto \linebreak[4] (\tau(P_1), \tau(P_2), \dt, \tau(P_k))$ is a covering map from $\mc{R}$ onto $\rs^k \sm (diagonals)$, and therefore $\mR$ is  a complex manifold. 
We see that $\psl$ acts on $\mR$ freely and holomorphically.
Therefore  $\mc{H}$ is also a complex manifold and  $\mR$ is a $\psl$-bundle over $\mc{H}$.  (See \cite{Fulton-69, Volklein-96}.)
Moreover, 
\begin{theorem}[Liu-Osserman \cite{Liu-Osserman}]\label{connected}
$\mc{H}$ is a connected manifold; hence, $\mR$ is also a connected manifold.
\end{theorem}

\section{Decomposition of a Schottky structure into almost good holed spheres}\label{IntoAlmostGood}

\begin{definition}
Let $C = (f, \rho)$ be a \ps\ on a surface $F$.
A loop $\ell $ on $F$ is a {\it meridian} if $\ell $ is an essential loop on $F$ and $\rho(\ell) = id$.
\end{definition}

\begin{definition}
Let $F$  be a surface and $M$ be a multiloop on $F$.
The {\it essential part} of $M$ is the union of all essential loops of $M$,
which we denote  by $\lf M \rf$. 
\end{definition}
 

\begin{lemma}\label{support}

Let $F$ be a surface of genus zero and $C = (f, \rho_{id})$ be an almost good projective structure on $F$ supported on a surface $R \st \rs$.
Let $\ell $ be  an inessential loop in the complement of $\pt R$ in $\rs$.
Then $\lf f\iv(\ell) \rf = \emptyset$.
\end{lemma}

\begin{proof}
Suppose that $\ell$ is an inessential loop in $\rs \sm \pt R$. 
Then $\ell $ bounds a disk $E$ contained in $\rs \sm \pt R$.
Since $C$ is almost good,  $f$ has the lifting property along every path on $\rs \sm \pt R$.
Thus each component of $f\iv(\ell)$ is a loop bounding a disk in $F$ that $f$ homeomorphically maps onto $E$.
Therefore $\lf f\iv(\ell) \rf = \emptyset$.
\end{proof}


\subsection{Pullback of a multiloop via a developing map}\label{pullback}
Let $\Gm$ be a real Schottky group of rank $g$. (However this assumption that $\Gm$ is real is not essential since our arguments will be all topological.)
Let $\Og$ denote the domain of discontinuity of $\Gm$.
Let $C' = \Og / \Gm$, which is a uniformizable projective structure with  Schottky holonomy $\rho\cn \po(S) \to \Gm$.
Let $H'$ denote the handlebody $(\Og \cup \h^3) / \Gm$ bounded by $C'$.
Let $\St$ be the developable  cover $S$ with respect to $\rho$, so that $\Gm$ is its covering transformation group (see \S \ref{MinDev}), and let $\pg\cn\St \to S$ denote the covering map.
Then we see that $\St$ is naturally homeomorphic to $\rs \sm \Ld$, where $\Ld$ is the limit set of $\Gm$, and in particular $\St$ is  planar.  

Let $C$ be a projective structure on $S$ with holonomy $\rho$.
Set $C = (f , \rho)$, where $f\colon \St \to \rs$ is the developing map of $C$.
Let $\rt\cn \po(S)/ \ker(\rho) \to  \Gm$ be the isomorphism induced from $\rho$.
Then $f$ is $\rt$-equivariant.

Let $\ell$ be a meridian loop on $C'$. 
Let $\elt$ be a lift of $\ell$ to $\Og$, which is also a loop.
Then 
\begin{lemma}\label{multiloop}
Then $f\iv(\elt)$ is a multiloop\footnote{There may be locally infinitely many loops.} on $\St$.
\end{lemma}

\begin{proof}
Since $f$ has the path lifting property in the domain of discontinuity $\Og$  (see \cite{Kuiper-50}),  $~f\iv(\elt)$ is a one-manifold properly embedded in $\St$.
Moreover,   since $f$ is $\td{\rho}$-equivariant, each  $x \in \St$ has a neighborhood $U$, such that, if $\gm \in \Gm$, then $\gm U \cap f\iv(\elt)$ is either a  single arc property embedded in $U$ or the empty set.
Therefore $f\iv(\elt)/\Gm$ is a multiloop on $(S, C)$.

Suppose that $f\iv(\elt)$ contains a biinfinite simple curve $\ap$.
Then $\ap$ covers the loop $\elt$ via $f$.
On the other hand, since $\ap/\Gm$ is a component of  the multiloop $f\iv(\elt)/\Gm$, there is a non-identity element of $\Gm$ that preserves $\ap$.
Then, since $f$ is $\td{\rho}$-equivariant,  $f(\ap)$ is an unbounded curve in $\Og$. 
This is  a contradiction, since $f(\ap)$ is contained in the loop $\elt$. 
\end{proof}

A loop on the boundary of a handlebody is called {\it meridian} if it bounds a disk properly embedded in the handlebody. 
Let $N'$ be a multiloop on $C' = \pt H'$ consisting of meridian loops. 
Let $\Nt'$ be the  total lift of $N'$ to $\Og$.
Then $\Nt'$ is a $\Gm$-?nvariant multiloop on $\Og$.
Since $f$ is $\td{\rho}$-equivariant, by Lemma \ref{multiloop}, $f\iv(\Nt')$ is a  $\Gm$-invariant multiloop on $\St$.
 (A compact subset of $\St$ contain infinitely many components of $f\iv(\Nt')$.)  
Let $\Nt = \lf f\iv(\Nt') \rf$.
Let $N$ denote the multiloop on $(S, C)$ obtained by quotienting $\Nt$ by $\Gm$, which we call the {\it pullback} of $N'$ (via $f$).

\begin{proposition}\label{N}
Assume that $N'$ is a multiloop on $C'$ such that
\begin{itemize}
\item[(I)] $N'$ consists of finitely many meridian loops and
\item[(II)] each component of $C' \sm N'$ is a holed sphere. 
\end{itemize}
Then the pullback $N$ of $N'$ via $f$ satisfies
\begin{itemize}
\item[(i)] $N$ is nonempty,
 \item[(ii)] $N$ consists of finitely many meridian loops on $S$ (with respect to $\rho$), and 
 \item[(iii)]  if $Q$ is a component of $S \sm N$, then the restriction $C|Q'$ is an almost good holed sphere supported on a component of $\Og \sm \td{N}'$, and
 the full  support of $C|Q'$ has at least two boundary components. 
\end{itemize}
\end{proposition}
Let $\Ct = (f, \rho_\id)$ denote the projective structure on $\St$ obtained by lifting $C$, where $\rho_{id}\cn \po(\St) \to \psl$ is the trivial holonomy.
Then Proposition \ref{N} immediately implies
\begin{corollary}\label{SuppAreAdjacent}
(i) Let  $\td{Q}$ be a component of $\St \sm \Nt$.
Then there is a unique component $R$ of  $\Og \sm \Nt'$  such that $\Ct|\td{Q}$ is an almost good holed sphere supported on $R$ (via $f$).
(ii) Let $\ell$ be a loop of $\Nt$ and let $\td{Q}_1$ and $\td{Q}_2$ be the adjacent components of $\St \sm \Nt$ along $\ell$.
Then the support of the almost good structures  $\Ct|\td{Q}_1$ and $\Ct|\td{Q}_2$ are the adjacent components of $\Og \sm \Nt'$ along the loop $f(\ell)$.
\end{corollary}
The rest of the section is the proof of the proposition. 

\begin{proof}[Proof of Proposition \ref{N} (ii)]
First we show that,  for each $x \in \St$, there is a neighborhood $U (= U_x)$ of $x$ such that $U \cap \Nt$ is either the empty set or a single arc properly embedded in $U$. 
Since $f$ is a local homeomorphism, $x$ has a neighborhood $U$ that $f$ homeomorphically takes onto its image  $f(U) =:  U'$.
The codomain $\rs$ of $f$ is decomposed into   $\Ld$,  $\Nt'$ and  $\Og \sm \Nt'$. \\
{\it Case 1.} Suppose that $f(x) \in \Ld$.
Then, by Assumption (II), we can take $U$ so that $U'$ is a disk bounded by a loop of $\Nt'$.    
Then, since $f| U$ is a homeomorphism onto its image,  $\Nt' \cap U'$ is a union of infinitely many disjoint loops.
Since $U'$ is a disk,  those loops are  all inessential; thus $\Nt \cap U = \emptyset$.\\
{\it Case 2.} 
Suppose that $f(x) \in \Nt'$. 
Then, since each loop of $\Nt$ is isolated, we can take $U$ so that $U' \cap \Nt'$ is a single arc properly embedded in $U'$.
Accordingly $f\iv(\Nt') \cap U$ is a single arc properly embedded in $U$. 
Therefore $\lf f\iv(\Nt')\rf \cap U$ is either the arc or the empty set.\\
{\it Case 3.} 
Last suppose that $f(x) \in \Og \sm \Nt'$.
Then we can take $U$ so that $f(U)$ is disjoint from $\Nt'$.
Then $f\iv(\Nt') \cap U = \emptyset$ and therefore  $\lf f\iv(\Nt')\rf \cap U = \emptyset$.

The set of all $U_x$ for all $x \in \St$  is a cover of $\St$.
Via the covering map, this cover descends to a  cover of $S$.
Since $S$ is compact, there is  a finite subcover of $S$. 
Since the $\Gm$-action preserves $\Nt$ and each $U_x$ contains at most an single arc, $N$ consists of at most finitely many loops. 

Since $\Nt$ is a union of essential loops and the holonomy of $\Ct$ is trivial, each loop of $\Nt$ is a meridian loop on $\Ct$. 
Since $\Nt$ overs $N$, each loop of $N$ is a meridian loop on $C$.
\end{proof}

\begin{lemma}\label{pi(Q)}
Under the assumptions in Proposition \ref{N}, for every component $Q$ of $S \sm N$, the restriction of $\rho$ to $\pi_1(Q)$ is the trivial representation.
\end{lemma}

\begin{proof}
For $\ap \in \pi_1(Q)$, let $\gm = \rho(\ap) \in \Gm$.
We regard $\ap$ also as an oriented closed curve on $Q$ representing it.
Suppose that $\gm \neq id$.
Then $\gm$ is a loxodromic element. 
Therefore $\ap$ lifts a bi-infinite simple curve $\td{\ap}$ on $\St$ invariant under the infinite cyclic group $\langle \ap \rangle$ generated by $\ap$.
Let $p_1, p_2 \in \rs$ be the fixed points of the loxodromic element $\gm$.
Then $f|_{\td{\ap}} \,$ is a $\,\rho|_{\langle \ap \rangle}$-equivariant curve connecting $p_1$ and $p_2$.
By Assumption (II), there is a loop $\mu'$ of $\Nt'$ separating $p_1$ and $p_2$. 
By a small homotopy of $\ap$ on $Q$ if necessary, we can assume that the curve $f|_{\td{\ap}}$ is disjoint from the points $p_1$ and $p_2$ and it transversally intersects the loop $\mu'$.  
Since $f|_{\td{\ap}}$ is $\rho|_{\langle \ap \rangle}$-equivariant and $p_1, p_2$ are contained in the different components of $\rs \sm \mu'$,  $~f|_{\td{\ap}}$ intersects $\mu'$ an odd number of times.  
Therefore $\td{\ap}$ transversally intersects $f\iv(\mu')$ an odd number of times. 

By Lemma \ref{multiloop},  $~f\iv(\mu')$ is a multiloop on $\St$.  
Furthermore, by Proposition \ref{N} (ii), each component of $f\iv(\mu')$ is either  a loop of $\Nt$ or an inessential loop on $\St$. 
The curve $\td{\ap}$ is  $\langle \ap \rangle$-invariant and properly immersed in $\St$.
Thus $\td{\ap}$  is unbounded. 
Therefore,  each inessential loop of $f\iv(\mu')$ intersects $\td{\ap}$ an even number of times. 
Since $\td{\ap}$ intersects $f\iv(\mu')$ an odd number of times, $\td{\ap}$ must intersect at least one essential loop of $f\iv(\mu')$.
Thus $\td{\ap}$ intersects $\Nt$, and therefore $\ap$ transversally intersects $N$.
This contradicts the assumption that $\ap$ is contained in $Q$.
Hence $\rho(\ap) = id$ for all $\ap \in \pi_1(Q)$.
\end{proof}

Proposition \ref{N} (i) immediately follows from

\begin{lemma}\label{g=0}
Every component of $S \sm N$ has genus zero.
\end{lemma}

\begin{proof}
Let $Q$ be a component of $S \sm N$.
By Lemma \ref{pi(Q)}, $~\pi_1(Q) \st \ker \rho$.
Thus, by the definition of $\St$, $~Q$ homeomorphically lifts to a component $\td{Q}$ of $\St \sm \Nt$.
Since $\St$ is planar, $Q$ has genus zero. 
\end{proof}

What is left is: 

\begin{proof}[Proof  Proposition \ref{N} (iii)] 
Let $Q$ be a component of $S \sm N$.
By Lemma \ref{g=0}, $~Q$ is a holed sphere.
Since $Q$ is  bounded by essential loops on $S$, $~Q$ has at least two boundary components. 
By Lemma \ref{pi(Q)}, the holonomy of $C|Q$ is  trivial.  
Thus we can isomorphically lift $Q$ to a component $\td{Q}$ of $\Ct \sm \Nt$ (unique up  to an element of $\Gm$).
Then the developing map of $C|Q$  is the restriction of $f$ to $\td{Q}$. 

For the first assertion of $(iii)$, it suffices to show that there is a component $R$ of $\Og \sm \Nt'$ such that,  via $f$, each boundary component $\ell$ of $\Qt$ covers a boundary component of $R$.
Clearly $f(\ell)$ is a loop of $\Nt'$.
Since $f$ is a local homeomorphism, there is  a small neighborhood of $\ell $ in $\td{Q}$ so that its $f$-image is contained in  a component $R$ of $\Og \sm \Nt'$.
We show that, if $m$ is another boundary component of $\td{Q}$, then the loop $f(m)$ is also a boundary component of $R$.
Let $\ap\colon [0,1] \to \td{Q}$ be a property embedded arc connecting $\ell $ to $m$, so that $\ap(0) \in \ell$ and $\ap(1) \in m$.
We can assume that $f \cc \ap$ is transversal to $\Nt$, if necessary by a small isotopy of $\ap$.
Suppose that $f(m)$ is  {\it not} a boundary component of $ R$.
Then $f(m)$ is contained in the interior of a component $D$ of $\rs \sm R$.
Then $D$ is a topological disk bounded by a boundary component of $R$. 

 {\it Case 1.} Suppose that $\pt D =: n'$ is different from $f(\ell)$.
Then  $f(\ell)$ and $f(m)$ are contained in the distinct components of $\rs \sm  n'$.
Then the curve $f \circ \ap$ is transversal to $n'$. 
Since $f \circ \ap$ is a curve on $\rs$ connecting $f(\ell)$ to $f(m)$,
it intersects $n'$ an odd number of times. 
Therefore $\ap$ transversally intersects $f\iv(n')$ an odd number of times. 
Then, similarly to the proof of Lemma \ref{pi(Q)}, we can deduce that $\ap$ transversally  intersects $\Nt$.
This contradicts the assumption that $\ap$ is in $\td{Q}$.

{\it Case 2.} Suppose that $f(\ell)$ bounds $D$.
Since $\ap$ is property embedded in $\Qt$,  the curve $f \cc \ap\cn [0,1] \to \rs$ takes  $(0, \e]$ into $R$ for sufficiently small $\e > 0$.
Then the point $f \cc\ap(\e)$ and the loop $f(m)$ are contained in the distinct components of  $\rs \sm f(\ell)$.
Then the restriction of $f \circ \ap$ to $[\e, 1]$ transversally intersects $f(\ell)$ an odd number of times.
Then we can  similarly deduce a contradiction that $\ap$ must transversally intersects $\Nt$.  

For the second assertion of $(iii)$, we show that $f(\pt \td{Q})$ consists of at least two boundary components of $R$.
Suppose that $f(\pt \td{Q})$ is a single boundary component $\ell '$ of $R$.
Then $f|_{\td{Q}}\colon \td{Q} \to \rs$ has the path lifting property on $\rs \sm \ell'$.
Since $f\iv(\ell')$ is a union of disjoint loops on $\St$, $~f\iv(\ell') \cap \td{Q}$ is a union of disjoint loops on $\td{Q}$.
Let $P$ be a component of $\td{Q} \sm f\iv(\ell')$ such that $P$ and $\td{Q}$ share a boundary component.
Then this common boundary component is an essential loop on $\St$.
On the other hand $f|_P$ is a covering map from $P$ onto a component of $\rs \sm \ell'$, which is a disk.
Thus $P$ is also a disk, and therefore the boundary component of $P$ is inessential on $\St$, which is a contradiction
\end{proof}

\section{Schottky structures and  handlebodies with cellular structures.}\label{CellularStructures}

(The techniques used in this this section has great influence from a paper by Oertel \cite{Oertel-02}.)
We carry over our notation from \S \ref{IntoAlmostGood}. 
Let $\Og_0$ be a standard fundamental domain of the action of the real Schottky group $\Gm$  on the domain of discontinuity $\Og$, i.e.\  $\Og_0$ is the complement of   disjoint $2g$ round disks.
We in addition assume that $\Og_0$ is a closed subset of $\rs$.  
The boundary components of $\Og_0$ are paired up by generators $\gm_1, \gm_2, \dt, \gm_g$ of $\Gm$.
In addition $\pt \Og_0$ is a multiloop in $\Og$, and $\Gm (\pt \Og_0) =: \Lt'$  is a $\Gm$-invariant multiloop splitting $\Og$ into the fundamental domains  $\gm \Og_0$ with $\gm \in \Gm$.

Recall that  the projective structure $C' = \Og/\Gm$ is the boundary surface of the genus-$g$ hyperbolic handlebody $H' =  (\h^3 \cup \Og) / \Gm$.
Let $L' = \Lt' / \Gm$.
 Then $L'$ is a multiloop on $C'$ that is the union of $g$ disjoint non-parallel loops bounding meridian disks in $H'$, so that $C' \sm L'$ is a $2g$-holed sphere. 
Let $L$  be the pullback of $L'$ via $f$, which is a multiloop on $S$ (see \S \ref{pullback}).
Thus $L$ satisfies the assertions of Proposition \ref{N}.   

Let $\Lt := \lf f^{-1}(\Lt') \rf$.
Then $\Lt$ is the total lift of $L$  to $\St$,  by the definition of $L$,

\subsection{Cellular structures on handlebodies.}\label{cellular}

For a subset $X$ of $\h^3 \cup \Og$, let $Conv(X)$ denote the convex hull of $X$ in $\h^3 \cup \Og$.
Then, for each loop $\ell $ of $\Lt'$, $~Conv(\ell)$ is a copy of $\overline{\h^2}$, the union of $\h^2$ with its ideal boundary  $\pt_\infty \h^2$. 
Let $\Dlt'$ denote the union of  $Conv(\ell)$ over all  loops $\ell $ of $\Lt'$.
Then $\Dlt'$ is a multidisk properly embedded in $\h^3 \cup \Og$, and $\Dlt'$ divides $\h^3 \cup \Og$ into connected fundamental domains of the action $\Gm \curvearrowright \h^3 \cup \Og$.
Let $\Dl' = \Dlt' / \Gm$.
Then $\Dl'$ is the disjoint union of $g$ copies of  $\ol{\h^2}$ in $H'$ bounded by $L'$, and thus  $\Dl'$ cuts $H'$ into a topological 3-disk.
Thus, we can regard the pair $(H',\Dl')$ as a handlebody with a cellular structure consisting of $g$ 2-cells, the disks of $\Dl'$, and one 3-cell, the closure of $H' \sm \Dl'$.  

By Proposition \ref{N} (iii),  each component of $S \sm L$ is a sphere with at least 2 holes.
Thus we let $H$ be a  handlebody of genus $g$ and identify $S$ and $\pt H$ by a homeomorphism so that each loop of $L$ on $S$ is a meridian loop of $H$. 
Let $\Dl$ be the multidisk bounded by $L$ and embedded in $H$ properly.
Then $\Dl$ divides $H$ into finitely many $3$-disks.
Thus, similarly, we regard $(H,\Dl)$ as the handlebody $H$ with the natural cellular structure: its 2-cells are the disks of $\Dl$ and 3-cells are the closures of components of $H \sm \Dl$.
Let $\Ht$ denote the universal cover of $H$, so that $\pt \Ht = \St$.
Let $\Dlt$ denote the total lift of $\Dl$ to $\Ht$, which is a $\Gm$-invariant multidisk bounded by $\Lt$.

In \S \ref{CellularStructures}, we prove:

\begin{proposition}\label{e}
There exists an embedding $\e\colon (H , \Dl)\to (H', \Dl')$ such that
\begin{itemize}
\item[(i)]  $~\e$ embeds each 2-cell  of $(H,\Dl)$ into a 2-cell of $(H', \Dl')$ and each 3-cell  of $(H,\Dl)$ into  a 3-cell of $(H', \Dl')$,
\item[(ii)] $H'$ minus the interior of ${\rm Im}(\e)$ is homeomorphic to the product $S \times [0,1]$, and
\item[(iii)] if $\ell $ is a loop of $\Lt$, then $\td{\e}(\ell)$ is contained in   $Conv(f(\ell)) \cong \ol{\h^2}$, where $\et\cn(\Ht, \Dlt) \to (\h^3 \cup \Og, \Dlt')$ is a lift of $\e$.
\end{itemize}
\end{proposition}
{\it Remarks:}
In (ii),  $S \times \{0, 1\}$ corresponds to $\pt H' \sq \e(\pt H)$.
In (iii),  the lift $\et$ exists since $\e$ is $\po$-injective by (ii).
It is straightforward to see that (iii) can  be equivalently replaced by: (iii') $\rho$ is equal to  the induced homomorphism $(\e|_S)_\ast\colon \po(S)  \to \po(H') = \Gm$.

By Proposition \ref{e} (i), (iii), and Proposition \ref{N} (iii), we immediately obtain
\begin{corollary}\label{eCorollary}
If $R$ is the closure of a component of $\St \sm \Lt$, then $\et$ properly embeds $R$ into the convex hull of  ${\it Supp}_f(R)$ in $\h^3$.
\end{corollary}  
\subsection{Dual graphs of cellular handlebodies.}\label{DualGraph}

Let $(M, \Dl_M)$ be a pair of a 3-manifold $M$ with boundary and a locally finite multidisk $\Dl_M = \sq_{i \in I} D_i$ properly embedded in $M$. 
Then there is a graph dual to this pair:
Pick  pairwise disjoint regular neighborhoods $N_i$ of the disks $D _i$ ($i \in I$) so that $N_i$ are homeomorphic to $\D^2 \times [-1,1]$ and $N_i \cap \pt M$ are homeomorphic to $\s^1 \times [-1, 1]$.
Collapse each $\D^2 \times \{x\}$ to a single point for each $i \in I$ and $x \in [-1,1]$ and each component of $M \sm \sq_i N_i$ also to a single point.
Then the resulting quotient space is a graph whose edges bijectively correspond to the regular neighborhoods $N_i~(i \in I)$ and vertices to  the components of $M \sm \sq_i N_i$.
This graph is called the {\it dual graph} of $(M, \Dl_M)$ and denoted by $(M, \Dl_M)\at$.
Accordingly, if $X$ is an edge or a vertex of the dual graph, let $X\at$ denote its corresponding disk $D_i$ or a component of $M \sm  \Dl_M$, and if $X$ is a disk of $\Dl_M$ or a component of $M \sm \Dl_M$, then let $X\at$ denote its corresponding edge or  vertex of the dual graph.
  
The dual graph $(M, \Dl_M)\at =: G_M$ can be embedded in $(M, \Dl_M)$, realizing the duality: 
Each vertex of $G_M$ is in the corresponding  component of $M \sm \sq_i N_i$ and each edge of $G_M$ transversally intersects $\Dl_M$ in a single point contained in the disk $D_i$ dual to it (see Figure \ref{dual}).

We in addition assume that $M$ is a handlebody of genus $g$ and $\Dl_M$ is a union of finitely many disjoint meridian disks in $M$ such that $\Dl_M$ splits $M$ into $3$-disks.
Particularly,  we have contracted such pairs $(H, \Dl)$ and $(H', \Dl')$ in \S \ref{cellular}.
Under this assumption,  $G_M$ is a finite connected graph and every vertex of $G_M$ has degree at least two.
Since $\Dl_M$ splits  $M$ into $3$-disks, we can choose the dual embedding of $G_M$ into $(M, \Dl_M)$ so that $M$ is a regular neighborhood of $G_M$.
Then $\po(G_M)$ is isomorphic to the  free group of rank $g$. 

Let $\td{M}$ be the universal cover of the handlebody $M$, and let $\Dlt_M$ be the total lift of the multidisk $\Dl_M$ to $\td{M}$.
Let $\Gt_M$ be the universal cover of the dual graph $G_M$.
Then $\Gt_M$ is the dual graph of $(\td{M}, \Dlt_M)$.
The action of $\po(G_M) = \po(M)$ commutes with the dual embedding $\Gt_M$ into $(\td{M}, \Dlt_M)$.
Thus,
for every $\gm \in \po(M)$ and every cell $X$ of $\Gt_M$ and $(\td{M}, \Dlt_M)$, we have $(\gm  X)\at = \gm  \cdot X\at$. 

Conversely, given a finite graph $K$, we can easily construct a pair of a handlebody $H_K$ and a union $\Dl_K$  of disjoint meridian disks in $H_K$ splitting $H_K$ into $3$-disks such that  $K$ is the dual graph of  $(H_K, \Dl_K)$.

\begin{figure}[htbp]

\includegraphics[width = 2in]{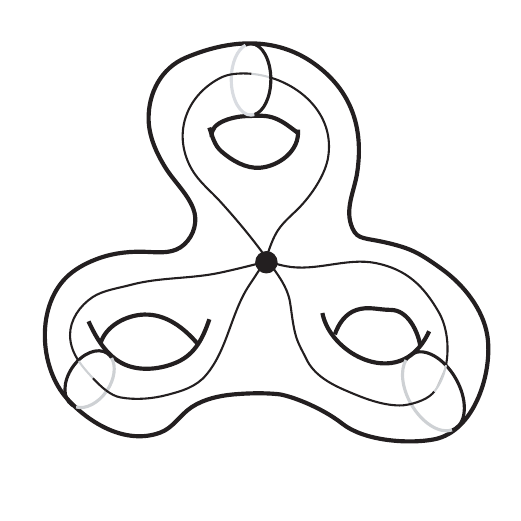}
\caption{The dual graph  embedded in $(H', D')$ in the case of $g = 3$.}\label{dual}
\end{figure}

\subsection{Graph map}\label{GraphHomomorphism}

A {\it graph map} is a simplicial map between graphs that maps each edge onto an edge (and each vertex to a vertex).
Let  $G$ and  $G'$ be the dual graphs of the cellular handlebodies $(H, \Dl)$ and $(H',\Dl')$, respectively, constructed  in \S \ref{cellular}.
Then $G'$ is  a bouquet of $g$ circles consisting of $g$ edges and one vertex (e.g. Figure \ref{dual}).

We will see that the developing map $f$ induces a graph map $\kp\colon G \to G'$.
Observe that the identification  $S = \pt H$ induces an isomorphism between  
$\po(H)$ and $\po(S)/\ker \rho$. 
Recall that the Schottky representation $\rho\cn \po(S) \to \Gm$ induces an isomorphism  $\rt\colon \po(S)/\ker(\rho) \to \Gm$.
Then this isomorphism can also be  regarded as an isomorphism from $\po(H)$  to $\po(H')$ and, by the duality, from $\po(G)$ to $\po(G')$.
 
Let $\Gt$ and $\Gt'$ be the universal covers of $G$ and $G'$, respectively.
Then we construct a $\rt$-equivariant  $\kpt\colon \Gt \to \Gt'$. 
For each vertex $v$ of $\Gt$, its dual $v\at$ is a component of $\Ht \sm \Dlt$, and $ v\at \cap \St$ is a  component of $\St \sm \Lt$.
Thus, by Corollary \ref{SuppAreAdjacent} (i),  the restriction of the projective structure $\Ct$ to $v\at \cap \St$ is an almost good holed-sphere supported on the fundamental domain $\gm \Og_0$ for a unique $\gm \in \Gm$.
Then $Conv(\gm \Og_0)$ is a component of $(\h^3 \cup \Og) \sm \Dlt'$.
Thus we define $\kpt(v)$ to be the vertex of $\Gt'$ dual to  $Conv(\gm \Og_0)$.
Then
\begin{lemma}\label{kpt}
(i) $\kpt$ is $\rho$-equivariant. 
(ii) Let $v_1$ and $v_2$ be the adjacent vertices of $\Gt$.
Then $\kpt(v_1)$ and $\kpt(v_2)$ are also adjacent vertices of $\Gt'$.
\end{lemma}

\begin{proof}
(i). 
Let $v$ be a vertex of $\Gt$.
Then, as above,  ${\it Supp}_f(v\at \cap \St) = \gm \Og_0$ for a unique $\gm \in \Gm$.
 For all $\og \in \Gm$, we have  $(\og \cdot v)\at = \og \cdot v\at$.
Since $f$ is $\rho$-equivariant, $${\it Supp}_f((\og \cdot v)\at \cap \St) = {\it Supp}_f(\og \cdot (v\at \cap \St)) = \og \cdot {\it Supp}_f(v\at \cap \St) = \og \cdot \gm \Og_0.$$
Also
$$\og \cdot Conv(\gm \Og_0) = Conv (\og \cdot \gm \Og_0)\ {\rm and}$$  $$\og \cdot (Conv(\gm \Og_0))\at = (Conv (\og \cdot  \gm \Og_0))\at.$$
Thus $$\og \cdot \kpt(v) = \og \cdot (Conv( \gm \Og_0))\at = (\og \cdot Conv( \gm \Og_0))\at = \kpt(\og \cdot v).$$

(ii). For all adjacent vertices $v_1$ and $v_2$ of $\Gt$,   there is an edge $e$ of $\Gt$ connecting $v_1$ and $v_2$.
Since $e\at$ is a disk of $\Dlt$, then  $e\at \cap \St$ is a loop of $\Lt$.
 Thus $v_1\at \cap \St$ and $v_2\at \cap \St$ are adjacent components of $\St \sm \Lt$ along the loop $e\at \cap \St$. 
 By Corollary \ref{SuppAreAdjacent} (ii), ${\it Supp}_f(v_1\at \cap \St)$  and ${\it Supp}_f(v_2\at \cap \St)$ are adjacent components of $\Og \sm \Lt'$, having the common boundary component   $f(e\at \cap \St)$.
Therefore $\kpt(v_1)$ and $\kpt(v_2)$ are adjacent vertices of $\Gt'$.
\end{proof}

By Lemma \ref{kpt} (ii), $~\kpt$ from the vertices of $\Gt$ to of $\Gt'$ uniquely extends to the edges of $\Gt$:
 For each edge $[v_1, v_2]$ of $\Gt$ connecting a vertex $v_1$ to its adjacent vertex $v_2$, its image $\kpt([v_1, v_2])$ is the edge of $\Gt'$ connecting $\kpt(v_1)$ to  $\kpt(v_2)$.
Then,  by Lemma \ref{kpt} (i), $~\kpt\cn \Gt \to \Gt'$ is $\rho$-equivariant.
Therefore, quotienting $\kpt$ by $\Gm$,  we obtain a graph map  $\kp\colon G \to G'$. 

\subsection{Labeling}\label{labeling}
Recall that the free groups $\pi_1(H'), \pi_1(G')$ and $\Gm$ are canonically identified.
Let $e$ be an oriented edge of the graph $G'$.
Since $G$ is a bouquet of $g$ circles,  we can regard $e$ as a simple closed curve on $G'$ and its homotopy class as a unique element of $\{\gm_1^{\pm}, \gm_2^\pm, \dt, \gm_g^\pm\}$. 
We call this homotopy class  the {\it label} of $e$ and denote it by ${\it label}(e)$.

Let $\pg$ denote the covering map, induced by the $\Gm$-action, from $\St$ to $S$,  from $\Og$ to $\pt H'$, from $\Gt$ to $G$ and from $\Gt'$ to $G'$. 
We will see that the labels on the oriented edges of  $G'$ induce the unique labels on the oriented edges of $G, \Gt, \Gt'$ by  $\{\gm_1^{\pm}, \gm_2^\pm, \dt, \gm_g^\pm\}$  so that the labels are preserved under the graph maps $\kp\cn G \to G'$,  \,$\kpt\cn \Gt \to \Gt'$, \, $\pg\colon \Gt \to G$ and $\pg\colon \Gt' \to G'$. 
First, for each oriented edge $e$ of $\Gt'$, define ${\it label}(e)$ to be ${\it label}(\pg(e))$.
Then the labels on the oriented edges of $\Gt'$ are $\Gm$-invariant.
Next, for each oriented edge $e$ of $\Gt$, define ${\it label}(e)$ to be ${\it label}(\kpt(e))$.
Since $\kpt$ is $\rho$-equivariant, the labels on the oriented edges of $\Gt$ are also $\Gm$-invariant.
Finally, for each  oriented edge $e$ of $G$, define ${\it label}(e)$ to be ${\it label}(\td{e})$, where $\td{e}$ is a lift of $e$ to $\Gt$.
Since the labeling on the edges of $\Gt$ is $\Gm$-invariant, ${\it label}(e)$ does {\it not} depend on the choice of the lift $\td{e}$.
Then ${\it label}(e) = {\it label}(\td{e}) = {\it label}(\kpt(\td{e})) = {\it label}(\pg \cc \kpt(\td{e})) = {\it label}(\kp(e))$.
Therefore $\kp$ also preserves the labels. 

\subsection{Folding maps.}(See \cite{Stallings-83}.)
Let $K$ be a graph.  
Label each  oriented edge of $K$  with an element of $\{\gm_1^{\pm}, \gm_2^\pm, \dt, \gm_g^\pm\}$.
Suppose that there are two different oriented edges $e_1 = [u, v_1]$ and $e_2 = [u, v_2]$ of $K$ with a command vertex $u$ and
distinct vertices  $v_1, v_2$. 
In addition  assume that 
 ${\it label}(e_1) = {\it label}(e_2)$.
Then we can naturally identify the edges $e_1$ and $e_2$, yielding a new labeled graph $K'$ (see Figure \ref{folding}).
This operation is called a $\it{folding}$ and the graph map $\mu\colon K \to K'$ realizing this folding  is called a $\it{folding ~ map} $.
Note that a folding  decreases the number of edges in $K$ by $1$, if $K$ is a finite graph.
Clearly $K$ and $K'$ are homotopy equivalent, and in particular $\mu$ induces an isomorphism $\mu^\star\colon \po(K) \to \po(K')$.
Using the covering theory for graphs (\cite[\S 3.3]{Stallings-83}), we see that  $\mu$ lifts to a $\mu^\star$-equivariant  graph map $\mt$ from the universal cover of $K$ to that of $K'$. 
 
 \begin{figure}[htbp]
\includegraphics[width = 3in]{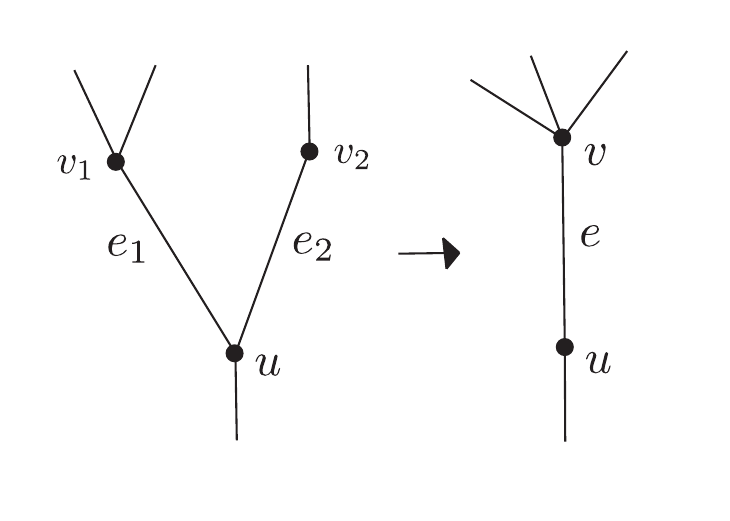}
\caption{Folding.}
\label{folding}
\end{figure}

In  \S \ref{GraphHomomorphism}, we constructed the  $\kp\cn G \to G'$ and, in \S \ref{labeling}, labeled the oriented edges of  $G$ and $G'$ so that $\kp$ preserves the labels. 
Then
\begin{lemma}\label{decomposition} 
There is a sequence of folding maps $$G = G_0 \xrightarrow{\mu_1} G_1 \xrightarrow{\mu_2} \dt \xrightarrow{\mu_n} G_n = G'$$ such that $\kp = \mu_n \cc \mu_{n-1} \cc \dt \cc \mu_1$.
\end{lemma}

\begin{proof} 
By \cite[\S 3.3]{Stallings-83}, $\kp\colon G \to G'$ is the composition of a sequence of folding maps 
$$G_0 \xrightarrow{\mu_1} G_1  \xrightarrow{\mu_2} G_2  \xrightarrow{\mu_3} \dt  \xrightarrow{\mu_n} G_n$$ and an {\it immersion} $\iota\colon G_n \to G'$ (i.e. locally injective graph map). 
Since $\kp, \mu_1, \mu_2, \dt, \mu_n$ are homotopy equivalences, $\iota\at\cn \po(G_n) \to \po(G')$ is an isomorphism; therefore $\iota$ is also a homotopy equivalence. 
Thus the lift of $\iota$ to a graph map between the universal covers of $G_n$ and $G'$ is an isomorphism that commutes with the action of $\Gm$; hence $\iota$ is the identity map. 
\end{proof}

\subsection{Proof of Proposition \ref{e}} 
Set  $\kp = \mu_n \cc \mu_{n-1} \cc \dt \cc \mu_1\cn G = G_0 \to G_n = G'$ as in Lemma \ref{decomposition}.
Then, for each $i \in \{1,2, \dt, n\}$, the folding map $\mu_i\colon G_{i-1} \to G_i$ lifts to a graph map  $\mt_i\colon \Gt_{i-1} \to \Gt_i$ between their universal covers, so  that $\mt_i$ is  equivariant under the isomorphism $\mu_i\sr\cn \po(G_{i-1}) \to \po(G_i)$.
Let $\kp_0 = \kp\cn G \to G'$ and $\kpt_0 = \kpt\cn \Gt \to \Gt'$.
For each $i \in \{1,2, \dt, n -1\}$, let $$\kp_i = \mu_n \cc \mu_{n-1} \cc \dt \cc \mu_{i+1}\cn G_i \to G'$$ and let $\kp_n\cn G_n \to G'$ be the identity map.
Then  let  $$\kpt_i =  \mt_n \cc \mt_{n-1} \cc \dt \cc \mt_{i+1}\cn \Gt_i \to \Gt',$$ which is a lift of $\kp_i$.
Then $\kpt_i$ is equivariant under $\kp_i\sr\cn \po(G_i) \to \po(G')$ and $\kpt_i \cc \mt_i = \kpt_{i-1}\cn G_{i -1} \to G'$ for each $i \in \{1, \dt, n\}$.

Let $H_0 = H$ and $\Dl_0 = \Dl$.
Similarly, for each $i \in \{1,2, \dt, n\}$,  let $(H_i, \Dl_i)$ denote the cellular handlebody  dual to the graph $G_i$, where $H_i$ is a genus $g$-handlebody and $\Dl_i$ is a union of  disjoint meridian disks in $H_i$ (see \S \ref{DualGraph}).
For each $i \in \{0, 1, \dots, n\}$, let $L_i$ be the multiloop on $\pt H_i$ bounding $\Dl_i$.
Let $\Ht_i$ denote the universal cover of $H_i$.
Let $\Dlt_i$ and  $\Lt_i$ denote the total lifts of $\Dl_i$ and $L_i$ to $\Ht_i$, respectively, so that $\pt \Dlt_i = \Lt_i$.

For a multiloop $M$, which is a subset of a surface, let $[M]$ denote  the set of all loops of $M$.
We define a $\kp_i\sr$-equivariant map $f_i\colon [\Lt_i] \to [\Lt']$ for each $i \in \{0,1,\dt,n\}$.
For each loop $\ell $ of $\Lt_i$, let $D_\ell$ be the disk of $\Dlt_i$ bounded by $\ell $.
Then $D_\ell\at$  is an edge of $\Gt_i$.
Then $\kpt_i(D_\ell\at)$ is an edge of $\Gt'$, and $(\kpt_i(D_\ell\at))\at$ is a disk of $\Dlt'$.
Define $f_i(\ell)$ to be  the loop of $\Lt'$ bounding the disk $(\kpt_i(D_\ell\at))\at$. 
Then by the definition of $\kp$ we see that  $f_0\colon  [\Lt] \to [\Lt']$ is exactly the correspondence given by the covering map $f|_\Lt\colon \Lt \to \Lt'$. 
Since $\kpt_i$ is $\kp_i\sr$-equivariant, $f_i$ is also $\kp_i\sr$-equivariant for all $i \in \{0,1,\dots, n\}$.
Since $\mt_i$ is $\mt_i\sr$-equivariant,  $\mt_i$ similarly induces a $\mu_i\sr$-equivariant map $h_i\colon [\Lt_{i-1}] \to [\Lt_i]$.
Then, since $\kpt_{i-1} = \kpt_i \cc \mt_i$, we have $f_{i-1} = f_i \cc h_i$.

The following proposition,  when $i = 0$, implies Proposition \ref{e}.
\begin{proposition}
For each $i \in \{0,1,2,\dt,n \}$, there exists an embedding $\,\e_i\colon (H_i, \Dl_i) \to (H', \Dl')$ such that\\
(i)   $\e_i$ takes each 2- and 3-cell of $(H_i,\Dl_i)$ into a 2- and 3-cell of $(H', \Dl')$, respectively,\\
(ii) the complement of ${\rm {\rm Im}}(\e_i)$ is  a product so that $H_i \sm int({\rm Im} (\e_i))$ is homeomorphic to  $S \times [0,1]$, and \\
(iii) if $\ell $ is a loop of $\Lt_i$, then $\et_i(\ell)$ is contained in $Conv(f_i(\ell)) \cong \ol{\h^2}$, where $\et_i\cn (\Ht_i, \Dlt_i) \to (\h^3 \cup \Og, \Dlt')$ is a lift of $\e_i$. 
\end{proposition}

\begin{figure}[htbp]
\includegraphics[width = 5in]{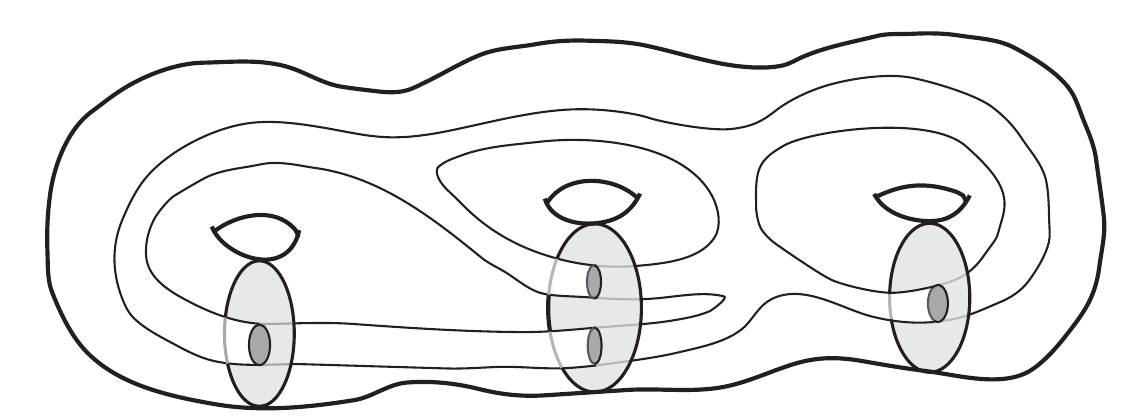}
\caption{An example of $\e_i$.}

\end{figure}

\begin{proof}
When $i = n$,  the graph map $\kp_n\cn G_n \to G'$ is an isomorphism. 
Then $\kp_n$ induces an isomorphism from $(H_n, \Dl_n)$  to $(H', \Dl')$.
The graph $G'$ is embedded in $(H', \Dl')$, realizing the duality between $G'$ and $(H', \Dl')$, so that $H'$ is a regular neighborhood of $G'$.
Thus we can easily construct an embedding $\e_n\cn(H_n, \Dl_n)  \to (H', \Dl')$, in an obvious way, onto a smaller regular neighborhood of $G'$  that satisfies (i), (ii) and (iii).

Next, assuming  that there is $\e_i$ satisfying (i) - (iii), we construct $\e_{i - 1}$ satisfying (i) - (iii).
(Figure \ref{Before&AfterSplitting} illustrates this induction.)
Let $e_1 = [u, v_1]$, $e_2 = [u, v_2]$ denote the edges of $G_{i-1}$ that the holding map $\mu_i$ identifies,
and let  $e = [u, v]$ denote the edge of $G_i$  obtained by this identification. 

Let $P$ and $Q$ be the components of $H_i \sm \Dl_i$ that are dual to the vertices $u$ and $v$, respectively.
Let $c_1, c_2, \dt, c_p$ be the edges of $G_{i-1}$, other than $e_1$, that end at $v_1$, and 
let $d_1, d_2, \dt, d_q$ be the edges of $G_{i-1}$, other than $e_2$, that end at $v_2$.

Let $D_{c_1}, D_{c_2}, \dt, D_{c_p}$, $D_{d_1}, D_{d_2}, \dt, D_{d_q}, D_e$ denote the disks of $\Dl_i$ that are dual to $c_1, c_2, \dt, c_p, \linebreak[2] d_1,d_2, \dt, d_q, e$, respectively. 
Then $Q$ is bounded by the meridian disks $D_{c_1}, D_{c_2}, \linebreak[2]  \dt, D_{c_p}$, $D_{d_1}, D_{d_2}, \dt, D_{d_q}, D_e$.  
Pick  two disjoint meridian disks $D_1$ and $D_2$ of $H_i$ parallel to $D_e$ such that $D_1$ and $D_2$ are contained in $P$ and that $D_1$ and $D_e$ bound a solid cylinder in $H_i$ containing $D_2$.

\begin{figure}[htbp]
\includegraphics[width = 5in]{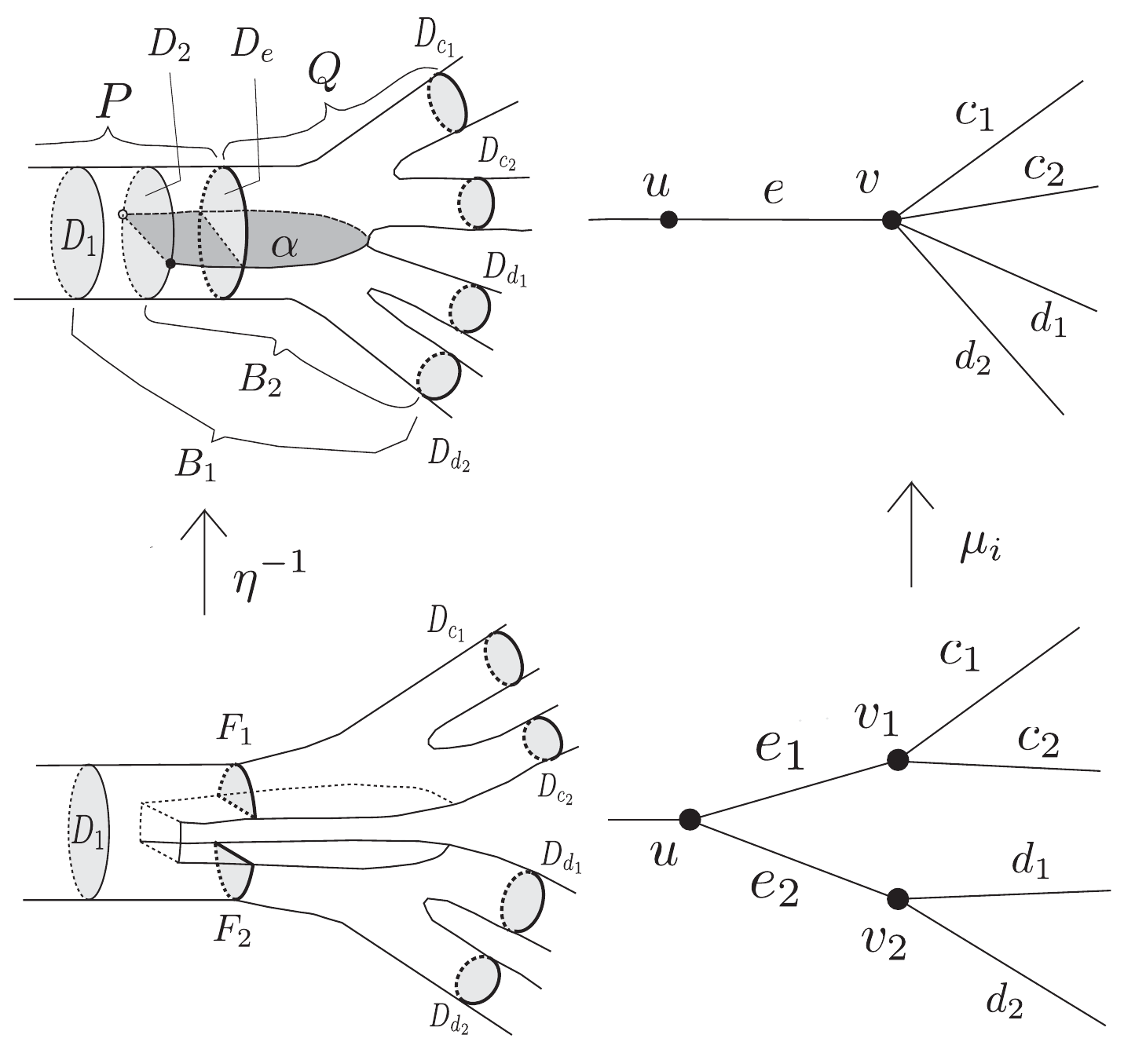}
\caption{}
\label{Before&AfterSplitting}

\end{figure}

{\it Case 1.}
Suppose that $u$ and $v$ are different vertices of $G_i$.
Then $u$, $v_1$, $v_2$ are different vertices of $G_{i -1}$, and $P$ and $Q$ are different components of $H_i \sm \Dl_i$. 
For each $j \in \{1,2\}$, let $B_j$ be the union of $Q$ with the solid cylinder, in $P$, bounded by $D_j$ and $D_e$.
Then $B_j$ is topologically the  $3$-disk and it is bounded by the meridian disks $D_j, D_{c_1}, D_{c_2}, \dt, D_{c_p}$, $D_{d_1}, D_{d_2}, \dt, D_{d_q}$.
Then $\pt B_2$ is a 2-sphere and those meridian disks are disjointly embedded in $\pt B_2$, so that the complement is a $(1 + p + q)$-holed sphere. 
Choose an arc $\ap$ properly embedded in this $(1 + p + q)$-holed sphere, such that:
(I) the end points of $\ap$ are distant points on $\pt D_2$,
 (II) $\ap$ separates $D_{c_1}, D_{c_2}, \dt, D_{c_p}$ and $D_{d_1}, D_{d_2}, \dt, D_{d_q}$ in the 2-disk $\pt B_2 \sm D_2$, and
(III) $\ap$ transversally intersects the circle $\pt D_e$ in exactly two points. 
Pick  an arc $\bt$ properly embedded in the disk $D_2$  connecting the end points of $\ap$. 
Then $\ap \cup \bt$ is a loop on $\pt B_2$ and it bounds a $2$-disk $E$  properly embedded in $B_2$.
In addition, we can assume that  $E$ transversally intersects $D_e$ in a single arc.

We will see that the compression along disk $E$ naturally transforms  $(H_i, \Dl_i )$  to $(H_{i - 1}, \Dl_{i - 1} )$ and also  $\e_i$  to $\e_{i -1}$  (Figure \ref{Before&AfterSplitting}).
Choose a small regular neighborhood $N$ of $E$ in $B_2$, so that $D_e \sm N$ is a union of two disjoint 2-disks, which we denote by $F_1$ and $F_2$. 
Then $N$ is a 3-disk such that $\pt N \cap \pt H_i$ is a 2-disk contained in the sphere $\pt B_1$. 
Thus there is an isotopy $\eta$ from $H_i$ to $H_i \sm N$ supported on $B_1$.
Then $Q \sm N$ has two components; 
if necessarily by swapping the names $F_1$ and $F_2$ of the disks,  a component of $Q \sm N$ is bounded by $F_1, D_{c_1}, D_{c_2}, \dt, D_{c_p}$ and the other by $F_2, D_{d_1}, D_{d_2}, \dt, D_{d_q}$. 
Then $H_i \sm N$ is still a genus-$g$ handlebody.
In addition $\Dl_i \cap (H_i \sm N)$ is obtained from $\Dl_i$ by replacing  the disk $D_e$ by the disks $F_1$ and $F_2$; thus  $\Dl_i \cap (H_i \sm N)$  is a union of meridian disks in the handlebody $H_i \sm N$.
Thus we see that the new cellular handlebody $(H_i \sm N,  \Dl_i \cap (H_i \sm N))$ is isomorphic to $(H_{i-1}, \Dl_{i-1})$.

We see moreover that the inverse isotopy $\eta\iv$ of $H_i \sm N$ to $H_i$ induces the folding map $\mu_i\cn G_{i -1} \to G_i$ via the duality: 
The dual of $v_1$ is the component of $H_{i-1} \sm \Dl_{i-1}$ bounded by $F_1, D_{c_1}, D_{c_2}, \dt, D_{c_p}$, \linebreak[4] and the dual of $v_2$ is the component of  $H_{i-1} \sm \Dl_{i-1}$ bounded by \linebreak[4] $F_2, D_{d_1}, D_{d_2}, \dt, D_{d_q}$; 
the isotopy $\eta\iv$ combines $F_1$ and $F_2$ into $D_e$, and accordingly $\mu_i$ folds the edges $F_1\at = e_1$ and $F_2\at = e_2$ into $D_e\at = e$; 
the isotopy $\eta\iv$ is supported on $B_i$ and in particular it preserves all the other disks of $\Dl_{i-1}$, and accordingly $\mu_i$ is an isomorphism onto its image in the complement of  $e_1$ and $e_2$.

We define $\e_{i -1}\colon (H_{i-1}, \Dl_{i-1}) \to (H',\Dl')$  to be the restriction of  $\e_i\colon (H_i, \Dl_i) \linebreak[2] \to (H', \Dl')$ to $(H_i \sm N,  \Dl_i \cap (H_i \sm N))$.
We show that $\e_{i-1}$ satisfies (i), (ii) and (iii).
Each 2- and 3-cell of $(H_{i-1}, D_{i-1})$ is contained in a 2- and 3-cell of $(H_i, \Dl_i)$, respectively.
Since $\e_i$ satisfies (i),  thus $\e_{i-1}$ also satisfies (i).
Recall that  $\eta$ isotopes ${\rm Im}(\e_i)$ to ${\rm Im}(\e_{i -1})$.    
Since $\e_i$ satisfies (ii), so does $\e_{i-1}$.

We show (iii).
Since the isotopy $\eta $ is supported on the 3-disk $B_1$ embedded in $H_i$, it lifts to a ($\Gm$-equivariant) isotopy $\td{\eta }$ from $\Ht_i$ to $\Ht_{i-1}$ supported on the total lift  $\td{B}_1$ of $B_1$ to $\Ht_{i}$. 
Since each component $R$ of $\td{B}_1$ is homeomorphic to $B_1$, we can canonically identify $\ett|_R$ with $\eta|_{B_1}$.
For each loop $\ell $ of $\Lt_{i-1}$, let $D_\ell$ denote the disk of $\Dlt_{i -1}$ bounded by $\ell $.
Let $m = h_i(\ell)$,  which is a loop of $\Lt_i$, and let $D_m$ be the disk of $\Dlt_i$ bounded by $m$.
Since $f_{i-1} = f_i \cc h_i$, we have $f_{i-1}(\ell) = f_i(h_i(\ell)) = f_i(m)$. 

First, suppose that $\ell $ does $not$ bound a lift of $F_1$ or $F_2$ in $\Ht_{i-1}$.
Since $\eta$ is the identity on $\Dl_i \sm D_\ell$, the isotopy $\tilde{\eta}\iv$ fixes $D_m$.
Therefore $D_\ell = D_m$ by the inclusion $(\Ht_{i-1}, \Dlt_{i-1}) \st (\Ht_i,\Dlt_i)$.
Then we have $\et_{i-1}(D_\ell) = \et_i(D_m)$. 
Since $\et_i$ satisfies (iii),  we have $\et_i(D_m) \st Conv(f_i(m)) = Conv(f_{i-1}(\ell))$.
Therefore $\et_{i-1}(D_\ell) \st Conv(f_{i-1}(\ell))$.
 
Next, suppose that $\ell $ bounds a lift of $F_1$ or $F_2$.
Then, accordingly, $D_\ell$ is a lift of $F_1$ or $F_2$ to $\Ht_{i-1}$.
Therefore $D_m$ is a lift of $D_e$ to $\Ht_i$, and it is contained in a component $R$ of $\td{B}_1$ by the inclusion $(\Ht_{i-1}, \Dlt_{i-1}) \st (\Ht_i,\Dlt_i)$.
Since $\ett|_R = \eta|_{B_1}$, we have $D_\ell \st D_m$ and thus $\et_{i-1}(D_\ell) \st \et_i(D_m)$.
Similarly to the previous case, we have $\et_i(D_m) \st Conv(f_i(m)) = Conv(f_{i-1}(\ell))$.
Thus $\et_{i-1}(D_\ell) \st Conv(f_{i-1}(\ell))$, and therefore $\e_{i-1}$ satisfies (iii).

\nin {\it Case 2.} (For the following discussion, see Figure \ref{case2}.)
Suppose that $u$ and $v$ are the same vertices of $G_i$.
Then, without loss of generality, we can assume that $u = v_1$ and $u \neq  v_2$ (if $u = v_1 = v_2$, there is a contradiction to the definition of folding maps). 
Then the components $P$ and $Q$ of $H_i \sm \Dl_i$ are the same component.
Since $u = v_1$, we can assume that  $e = c_1$.
Then $D_e = D_{c_1}$.

We define $B_1$ and $B_2$ analogously:
Let $B_1$ be the 3-disk in $H_i$ bounded by $D_1, D_{c_2}, \dt, D_{c_p}$,  $D_{d_1}, D_{d_2}, \dt, \linebreak[2]D_{d_q}$.
Let $B_2$ be the 3-disk in $H_i$ bounded by $D_1, D_2, D_{c_2}, \dt, D_{c_p}$, $D_{d_1}, D_{d_2}, \dt, D_{d_q}$.
Similarly let $\ap$ be an arc properly embedded in the $(1 + p+q)$-holed sphere 
$$\pt B_2 \sm (D_1 \sq D_2 \sq D_{c_2} \sq \dt \sq D_{c_p} \sq D_{d_1} \sq D_{d_2} \sq \dt \sq D_{d_q}),$$ such that:
(I) the end points of $\ap$ are contained in $\pt D_2$;
(II) $\ap$ separates $D_1, D_{c_2}, D_{c_3}, \dt, D_{c_p}$ and $D_{d_1}, D_{d_2}, \dt, D_{d_q}$ in the 2-disk $\pt B_2 \sm D_2$, and
(III) $\ap$ transversally intersects $\pt D_e = \pt D_{c_1}$  in exactly two points. 
The rest of the proof is also similar to  ${\it Case 1}$. 
\end{proof}

\begin{figure}[htbp]
\includegraphics[width = 5in]{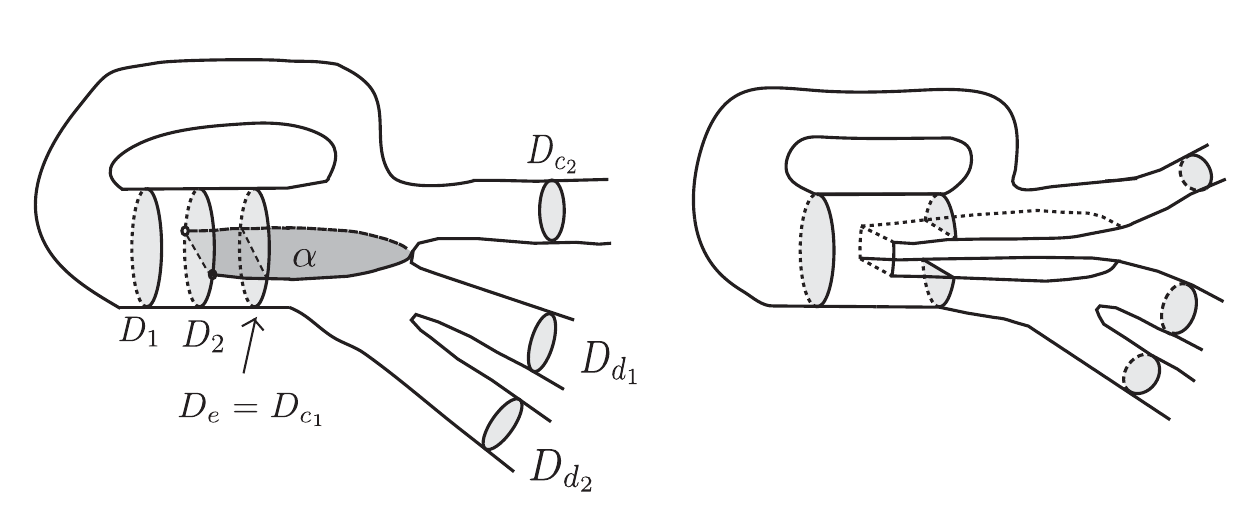}
\caption{}
\label{case2}
\end{figure}


\section{Decomposition of a Schottky structure into good holed spheres}\label{HoledSphere} 

We started with a projective structure $C$ on $S$  with real Schottky holonomy $\rho\cn \po(S) \to \Gm$.
Then, we obtained a multiloop $L$ on $S$ by Proposition \ref{N} (see \S \ref{CellularStructures}), and regarded $S$ as the boundary of the handlebody $H$ so that $L$ bounds a multidisk $\Dl$ in  $H$ (\S \ref{cellular}).
There we also constructed a pair $(H', \Dl')$ of the hyperbolic handlebody $H'$ and the multidisk $\Dl'$ consisting of copies of $\ol{\h^2}$ such that $\pt H$ is  the uniformizable structure $C' = \Og/\Gm$.
Let $\e\colon (H, \Dl) \to (H', \Dl')$ be the embedding obtained by Proposition \ref{e}.
Then by Proposition \ref{e} (ii)  there is a homeomorphism $\eta\colon S \times [0,1] \to H' \sm int({\rm Im}(\e))$.
We set $\eta_t(s) = \eta(s, t)$ so that $\eta_0$ is a homeomorphism from $S$ onto $\e(\pt H)$ and  $\eta_1$  onto $\pt H'$.
By Proposition \ref{e} (i), (iii),   the multiloop $\eta_0(L)$ embedded in $H'$ is the intersection of the boundary of  the embedded handlebody $\e(H)$ and the multidisk $\Dl'$.
Let $M' = \eta_1 (L)$.
Then $M'$ is a multiloop on $\pt H' = C'$.
 (This multiloop $M'$ depends on the choice of $\eta$, however it is  unique it is up to an isotopy.)

We show that $M'$ satisfies Assumptions (I), (II) of Proposition \ref{N}.
Since each component of $S \sm L$ is a holed sphere by Proposition \ref{N} (iii) and $\eta_1$ is a homeomorphism, each component of $C' \sm M'$ is also  a holed sphere; thus $M'$ satisfies (II). 
For an arbitrary loop $m'$ of $M'$, let $\ell$ be the corresponding loop of $L$, i.e. $\eta_1(\ell)  = m'$.
Then $\eta(\ell  \times [0,1])$ is an annulus properly embedded in the product region $H' \sm int({\rm Im}(\e))$ bounded by $\eta_1(\ell) = m'$ and $\eta_0(\ell) = \e(\ell)$.
Since $\ell $ bounds a meridian disk in $H$, then $\e(\ell)$ also bounds a meridian disk in the embedded handlebody ${\rm Im}(\e)$.
Thus the union of this meridian disk in ${\rm Im}(\e)$ and the annulus $\eta(\ell \times [0,1])$  is a meridian disk in  $H'$ bounded by $m'$.
Hence $M'$ satisfies (I).

Let $M$ be the pullback of $M'$ via the developing map $f\cn \St \to \rs$ of $C$. 
Then $M$ is a multiloop on $S$ (\S \ref{pullback}),  and it decomposes $(S,C)$ into almost good holed spheres (Proposition \ref{N} (iii)).
In this section, we show that this decomposition is moreover into good holed spheres.  
 Let $\Ct$ be the projective structure on $\St$ obtained by lifting $C$, and let $\Mt$ and  $\Mt'$ denote the total lifts of $M$ and $M'$ to $\St$ and $\Og$, respectively. 
Then 
\begin{theorem}\Label{M} 
There exists a $\td{\rho}$-equivariant homeomorphism $\zeta: \St \to \Og$ such that, if $\Pt$ is a component of $\St \sm \Mt$, then
 the restriction $\Ct | \Pt$ is a good holed sphere whose full support is $\zt(\Pt)$ and in particular each boundary component $\ell$ of $\Pt$ covers the loop $\zt(\ell)$ via $f$.  

\end{theorem}

Thus, if $P$ is a component of $S \sm M$, then the restriction of $C$ to $P$ is a good holed sphere.
Theorem \ref{M} follows from:
\begin{proposition}\label{single}
If $\mu'$ is a loop of $\Mt'$, then $\lf f\iv(\mu') \rf$ is a single loop on $\St$. 
\end{proposition}

\begin{proof}[Proof of Theorem \ref{M}, with  Proposition \ref{single} assumed] 
~The developing map $f$, by Proposition \ref{single},   yields a $\td{\rho}$-equivariant bijection $f_\ast\cn[\Mt] \to [\Mt']$ from the set of the loops of $\Mt$ onto that of $\Mt'$.  
Therefore, we can choose a $\td{\rho}$-equivariant homeomorphism $\zeta\colon \Mt \to \Mt'$ such that, if $m$ is a loop of $\Mt$, then  $m$ covers $ \zeta(m)$ via $f$.
We have seen that, if $P$ is a component of $\St \sm \Mt$, then $\Ct|P$ is an almost good holed sphere whose support is a unique component $R$ of $\Og \sm \Mt'$ via $f$.
Since $f_\ast$ is bijection, if different loops $a, b$ of $\Mt$ are boundary components of  $P$, then $f(a), f(b)$ are also  different boundary components of $R$. 
Since $f_\ast$ is a bijection,  Corollary \ref{SuppAreAdjacent} (ii) implies that $\zeta|_{\pt P}$ is  a homeomorphism onto $\pt R$.
Therefore $\Ct|P$ is a good holed sphere fully supported on $R$.
Thus $\zeta\cn \Mt \to \Mt'$ continuously extends to $P$ so that $\zeta|P$ is a homeomorphism onto $R$. 
Therefore we can extend $\zeta$ to a $\td{\rho}$-equivariant homeomorphism from $\St$ to $\Og$ with the desired property. 
\end{proof}

\subsection{An outline of the proof of Proposition \ref{single}}
Proposition \ref{single} is the main proposition of this paper, and we here outline its proof.
Let $\mu'$ be a loop of $\Mt'$ and let $\ld$ be the loop of $\Lt$ with $\ett_1(\ld) = \mu'$.
Let $\ld' = f(\ld)$, which is a loop of $\Lt'$.

{\it Step 1.}
We first reduce the proposition to a similar assertion about a good holed sphere obtained from $(\St, \Ct)$ as follows.   
Find a compact subsurface $F$ of $\St$ with boundary,  such  that the interior of $F$ contains the loop $\ld$  and the multiloop  $\lf f\iv(\mu') \rf$ and that the restriction $\Ct|F$ is an almost good holed sphere (Proposition \ref{F} and Corollary \ref{Fcontains}).
Let $\Fc$ be the puncture sphere obtained by attaching a once-punctured desk long every boundary component of $F$\,;
then $F$ is a subsurface of  $\Fc$.  
We naturally extend  $\Ct|F$ to a good structure $C_\Fc = (f_{\Fc}, \rho_\id)$ on $\Fc$, so that  $\lf f\iv(\mu') \rf = \lf f_{\Fc}\iv(\mu') \rf \st F$.
Thus it suffices to show that $\lf f_{\Fc}\iv(\mu') \rf $ is a single loop (Proposition \ref{reduced}).
We also see, in particular, that  $\lf f_\Fc\iv(\ld') \rf$ is a multiloop on $\Fc$ that decomposes $C_\Fc$ into almost good holed spheres.  

{\it Step 2.}
Proposition \ref{e} yields  an embedding $\et\cn \Ht \to \hb$  associated with the projective structure $C$.
We construct an analogous embedding $\e_\Fc$ of a 3-disk into $\hb$ associated with $C_\Fc$, where the punctured sphere $\Fc$ sits in  the boundary of  the 3-disk.
(See the left picture in Figure \ref{ExamplePhi}.)
In this analogy, the punctures of $\Fc$ correspond to the limit set of the Schottky group.

{\it Step 3.}
An isotopy of  $\rs$ induces a deformation of the good projective structure $C_\Fc$, by postcomposing its developing map $f_\Fc$ with the isotopy.
Under this deformation the projective structure remains a good structure on $\Fc$.
In addition  the loop $\mu'$ on $\rs$ is also isotoped, and accordingly the pullback of $\mu'$, a multilloop, is also isotoped on $\Fc$.
We then isotope $\rs$ so that, yet with the loop $\ldp$ fixed on $\rs$, the pullback of $\ldp$ is a single loop homotopic to $\ld$ on $\Fc$.
This isotopy is realized as a composition of isotopies of $\rs$ that inductively reduce the number of the components of the pullback of $\ld'$.

We then extend this isotopy of $\rs$ to $\hb$  (Figure \ref{ExamplePhi}) so that the induced good structure on $\Fc$ still satisfies all properties analogues to  Proposition \ref{e} (i) -(iii) at each induction step. 
Then a certain product structure analogues to Proposition \ref{e} (iii) implies that, after the isotopy, the loop $\mu'$ becomes isotopic to the fixed loop $\ld'$ on the punctured sphere  fully supporting the induced structure on  $\Fc$.
We thus conclude that the pullback of $\mu'$ is  also a single loop isotopic to $\ld$.

\subsection{The proof of Proposition \ref{single}}

 {\it Step 1.} Recall that  $\Og_0$ is the compact fundamental domain for the $\Gm$-action on $\Og$ bounded by $2g$ round loops of $\Lt'$.
Accordingly $Conv(\Og_0)$ is a fundamental domain for the $\Gm$-action on $\h^3 \cup \Og$.
Then $Conv(\Og_0)$ is a compact subset of $\h^3 \cup \Og$  bounded by $2g$ disks of $\Dlt'$.
Also $\et\cn (\Ht, \Dlt) \to (\h^3 \cup \Og, \Dlt')$ is the $\rt$-equivariant lift of $\e$. 
Similarly the homeomorphism $\eta\cn S \times [0,1] \to H' \sm int({\rm Im}(\e))$ lifts to a unique $\rt$-equivariant homeomorphism $\ett: \St \times [0,1] \to (\h^3 \cup \Og) \sm int({\rm Im}(\et))$.
Set $\ett_t(s) = \ett(s, t)$.

Let $\mu'$ be a loop of $\Mt'$ and $\ld$ be its corresponding loop of $\Lt$, i.e.\ $\ett_1(\ld) = \mu'$.
Then $f(\ld)$ is the loop of $\Lt'$ with $Conv(f(\ld)) \supset \et(\ld)$ ~(Proposition \ref{e} (iii)).
Then $\ett(\, \ld \times [0,1] )$ is a compact annulus embedded properly in the product region $(\h^3 \cup \Og) \sm int({\rm Im}(\et))$ and  bounded by $\ett(\ld \times \{0\}) = \et(\ld)$ and $\ett( \ld \times \{1\}) = \mu'$. 
Let $$F_0' = \bigcup \,\{\, \gm \Og_0 \,|\, \gm \in \Gm,\, Conv(\gm \Og_0) \cap \ett(\ld \times [0,1]) \neq \emptyset\, \}.$$
Then
\begin{lemma}\label{F_0'}
$F'_0$ is a compact connected subsurface of $\Og$ bounded by (finitely many) loops of $\Lt'$, and the interior of $F'_0$ contains the loops $\mu'$ and $f(\ld)$.
\end{lemma}

\begin{proof}
The multidisk $\Dlt'$ decomposes $\h^3 \cup \Og$ into the compact fundamental domains $Conv(\gm \Og_0)$ with $\gm \in \Gm$.
Since $\ett(\ld \times [0,1])$ is compact,   it intersects $Conv(\gm \Og_0)$ for only finitely many $\gm \in \Gm$.
Therefore, letting
 $$E = \cup \{\, Conv (\gm \Og_0) \,|\, \gm \in \Gm,\, Conv(\gm \Og_0) \cap \ett(\ld \times [0,1]) \neq \emptyset\, \}, $$
$E$ is compact. 
Besides, since $\ett(\ld \times [0,1])$ is connected,  $E$ is also connected. 
Thus  $E$ is a connected compact convex subset of $\h^3 \cup \Og$ bounded by (finitely many) disks of $\Dlt'$. 
Since $Conv (F'_0) = E$, thus $F_0'$ is a connected compact subsurface of $\Og$ bounded by finitely many loops of $\Lt'$.
Since $\mu' = \ett(\ld \times \{1\}) \st \ett(\ld \times [0,1])$, the interior of $F_0'$ contains $\mu'$ by the definition  of $F_0'$.  
Since the loop $\e(\ld)$ is contained in the annulus $\ett(\ld \times [0, 1])$ and in the disk $Conv (f(\ld))$ of $\Dlt'$, therefore $F'_0$ contains the adjacent components of $\Og \sm \Lt'$ along the loop  $f(\ld)$.
Therefore the interior of $F'_0$  contains $f(\ld)$. 
\end{proof}

\begin{proposition}\label{F}
There exist a compact connected subsurface $F$ of $\St$ bounded by finitely many loops of $\Lt$ and a compact connected subsurface $F'$ of $\Og$ bounded by finitely many loops of $\Lt'$ such that:\\ 
(i) $F'$ contains $F'_0$, \\
(ii) $\Ct|F$ is an almost good holed sphere supported on $F'$, and\\
(iii) $F$ contains $\et\iv(Conv(F'_0)) \cap \St$.
\end{proposition}

\begin{proof}
For a component $R$ of $\St \sm \Lt$, we have either  $Supp(R) \st F'_0$ or $Supp(R) \st \St \sm F'_0$ (here we mean $\Ct|R$ by $R$, abusing notation).
By Corollary \ref{eCorollary}, if $Supp(R) \st F'_0$, then $\et(R) \st Conv(F'_0)$ and, if $R \st \St \sm F'_0$, then $\et(R) \cap Conv(F'_0) = \emptyset$.
Therefore
$\et^{\,-1}(Conv(F'_0)) \cap \St\  (=: X_0)$ is equal to 
 $$\bigcup\, \{cl(R) ~|~R  ~is~ a~ component~ of~  \St \sm \Lt,~ Supp_f(R) \st F'_0\},$$
where $cl(R)$ denote the closure of $R$.

Since $f$ is $\rt$-equivariant,  for each $\gm \in \Gm$, there is at least one but at most finitely many components of $\St \sm \Lt$ supported on $\gm \Og_0$ (by Corollary \ref{SuppAreAdjacent} (ii)).
Therefore $X_0$ is a compact subsurface of $\St$ bounded by finitely many loops of $\Lt$, but $X_0$ is $not$ necessarily connected. 
Thus we choose a compact connected subsurface $F_0$ of $\St$ bounded by finitely many loops of $\Lt$ such that $F_0 \supset X_0$.
Then each component $Q$ of $F_0 \sm \Lt$ is a component of $\St \sm \Lt$, and $Q$ is supported on a unique component of $\Og \sm \Lt'$.
Let 
$$F' =  \bigcup\, cl ( {\it Supp}(Q \hspace{.2mm})),$$
 where $Q$ varies over all components of $F_0 \sm \Lt$.
By the definition of $F_0$ and $F'$,  we have $F'_0 \st  F'$; thus (i) holds.

Since $F_0$ is compact, by Corollary \ref{SuppAreAdjacent}, we see that $F'$ is also a compact connected subsurface of $\Og$ bounded by finitely many loops of $\Lt'$.
We see that  $\et\iv(Conv(F')) \cap \St$ is the union of the closures of finitely many components $R$ of $\St \sm \Lt$ such that ${\it Supp}(R) \st F'$.
Then $\et\iv(Conv(F')) \cap \St$ is a compact subsurface of $\St$ bounded by finitely many loops of $\Lt$, but again it is $not$ necessarily connected. 
Since $F_0$ is connected, $\et\iv(Conv(F')) \cap \St$ contains a connected component containing $F_0$ by the definition of $F'$.
Let $F$ be this component of $\et\iv(Conv(F')) \cap \St$. 
Since $F_0$ contains $X_0$, then $F$ also contains $X_0$; thus (iii) holds.

Since $F$ is a compact subsurface of $\St$,  its  genus is zero and it has at least two boundary components. 
By  Corollary \ref{SuppAreAdjacent},  $\pt F$ covers $\pt F'$ via $f$. 
We see that  $f(\pt F)$ has at least two components, similarly to Proposition \ref{N} (iii).
Therefore $\Ct|F$ is an almost good holed sphere supported on $F'$; thus (ii) holds.
\end{proof}

\begin{corollary}\label{Fcontains}
The multiloop $\lf f\iv(\mu') \rf$ is contained in the interior of the compact subsurface $F$ of $\St$ given by in Proposition \ref{F}. 
\end{corollary}
\begin{proof}
For an arbitrary component of $\St \sm \Lt$ that is disjoint from $F$, let $R$ denote its closure. 
Then it suffices to show that $R\, \cap \, \lf f\iv(\mu') \rf = \emptyset$.
By Proposition \ref{F_0'} (iii),  ${\it Supp}_f(R)$ disjoint from $int(F'_0)$.
Therefore, by Lemma \ref{F_0'}, $~\mu'$ and ${\it Supp}_f(R)$ are disjoint.
Hence, by Lemma \ref{support}, $R \cap \lf f\iv(\mu') \rf = \emptyset$.
\end{proof}

{\it Step2.} Set $\Ct|F = (f_F, \rho_{id})$, where  $f_F\cn F \to \rs$ is its developing map and $\rho_{id}\cn \po(S) \to \psl$ is the trivial representation.
Then, in order to prove Proposition \ref{single}, it suffices to show that, by Corollary \ref{Fcontains}, 
$\lf f_F\iv(\mu') \rf = \lf f\iv(\mu') \rf \cap F$ is a single loop on $F$.

Every boundary components of $F$ bound a disk of $\Dlt$.
The union of such disks bound a 3-disk $H_F$ in $\Ht$.
Then $H_F \cap \St = F$. 
Let $\e_F\cn H_F \to \h^3$ be the restriction of $\et\colon \Ht \to \h^3$ to $H_F$.
Accordingly, restricting  the embedding  $\ett\colon \St \times [0,1] \to \hb$ to $F \times [0,1]$, we obtain an embedding $\eta_F: F \times [0,1] \to H'$.
Let $\check{F}$ be the punctured sphere obtained by attaching a once-punctured disk along each boundary component of $F$, and
let $p_1, p_2, \dt, p_n$ be the punctures of $\Fc$, where $n$ is the number of the boundary components of $F$.
Then $\Fc \cup p_1 \cup p_2 \dt \cup p_n =: \hat{F}$ is a 2-sphere.
Let $H_{\Fh}$ be a closed 3-disk and identify its boundary  with  $\Fh$.
Then similarly the multiloop $\pt F\, (\st \Fh)$ bounds a multidisk property embedded in $H_{\Fh}$.
Then the multidisk bounds a 3-disk in $H_{\Fh}$\,; we naturally identify this 3-disk with $H_F$, so that $\pt H_F \cap \Fh = F$.

With respect to the inclusion $H_F \st H_\Fh$, we will extend $\e_F\colon H_F \to \hb$ to an embedding $\e_\Fh\colon H_\Fh \to \hb$ such that
\begin{itemize}
\item $\e_\Fh$ takes $p_1, p_2, \dt, p_n$ to $\rs$, and the rest to $\h^3$, so that $\rs \sm {\rm Im}(\e_\Fh)$ is homeomorphic to $\Fc$, and  
\item $\hb \sm \im(\e_\Fh)$ is homeomorphic to $\Fc \times (0,1]$, and the product extends to $\Fh \times \{0\}$, inducing a homeomorphism from $ \e_\Fh(\Fc)$ to $\Fh\times \{0\}$.
\end{itemize}

Since $F$ is embedded in both $\St$ and $\Fh$,   each boundary component $\ell $ of $F$ bounds a disk of $\Dlt$  in $\Ht$ and also a disk in $H_\Fh$ that is a component of   $\pt H_F \sm F$\,; 
we denote both disks by $D_\ell$. 
Let $F_\ell$ be the component of $\St \sm F$ bounded by $\ell$.
Then $F_\ell \sm \Lt$ is a union of infinitely many components of $\St \sm \Lt$.
Since $H_F \st H_\Fh$, we let $H_\ell$ be the component of $H_{\Fh} \sm H_F$ bounded by $D_\ell$.
Then $H_\ell$ is (topologically) a 3-disk whose boundary sphere is the union of the disk $D_\ell$ and the component of $\Fh \sm F$ bounded by $\ell$, whose interior contains a unique puncture point $p(\ell) \in \{p_1, p_2, \dt, p_n\}$. 

In $\hb$, $~Conv (f(\ell)) \cong \ol{\h^2}$ is a boundary component of $Conv (F')$. 
Then, let $X_\ell$ be the component of $\hb \sm Conv(F')$ bounded by $Conv(f(\ell))$.
Let $Y_\ell$ be the component of $\rs \sm F'$ bounded by $f(\ell)$, so that  $X_\ell \cap \rs = Y_\ell$.

\begin{proposition}\label{R_i}
There is a sequence  $(R_i)_{i=1}^\infty$ of distinct connected components of $F_\ell \sm \Lt$ such that\\
(i) $\ell $ is a boundary component of $R_1$,\\ 
(ii)  $R_i$ and $R_{i+1}$ are adjacent subsurfaces of $F_\ell$ for all $i = 1,2, 3, \dt$,\\
(iii) $\et(R_i) \st X_\ell$ for all $i = 1,2, 3,\dt$ , and\\
(iv) $(\et(R_i))_{i=1}^\infty$ converges to a limit point of $\Gm$ (contained in $Y_\ell$). 
(See Figure \ref{G_l}.)
\end{proposition}

\begin{proof} 
Let $R_1$ be the component of $F_\ell \sm \Lt$  bounded by $\ell  =: \ell_ 0$ (\,(i)\,).
Then $\Ct|R_1$ is a good holed sphere supported on a unique component $\Og_1$ of  $\Og \sm \Lt'$ (via $f$).
By Corollary \ref{SuppAreAdjacent}, $~\Og_1$ and $F'$ are adjacent subsurfaces of $\Og$ sharing  the boundary component  $f(\ell)$.
Thus $\Og_1 \st Y_\ell$, and by Corollary \ref{eCorollary}, $~\et(R_1) \st Conv(\Og_1) \st X_\ell$.

We inductively define $R_i$ for $i  \geq 2$.  
For $i \geq 1$, suppose that we have components $R_1, R_2, \dt, R_i$ satisfying (ii) with $R_1$ as above.
Then, let $\Og_i$ be the component of $\Og \sm \Lt'$ that is a support of  $\Ct|R_i$.
Let $\ell _{i -1} (\st \Lt)$ denote the common boundary component of $R_{i-1}$ and $R_i$. 
Then, by the definition of an almost good holed sphere, $f(\pt R_i)$ is a union of at least $2$ boundary components of $\Og_i$.
Thus we can pick a boundary component $\ell _i$ of $R_i$ so that $f(\ell_ i)$ and $f(\ell_ {i-1})$ are different boundary components of $\Og_i$.
Let $R_{i+1}$ be the component of $F_\ell \sm \Lt$ adjacent to $R_i$ along the boundary component  $\ell _i$. 

\begin{lemma}\label{Omega}
For each $k \geq 1$, $~\Og_1, \Og_2, \dt, \Og_k$ are distinct components of $Y_\ell \sm \Lt'$, and $cl(\sq_{i =1}^k \Og_i)$ is a holed sphere in $Y_\ell$ bounded by finitely many loops of $\Lt'$ such that there is a 2-disk component of  $Y_\ell \sm cl(\sq_{i = 1}^k \Og_i)$  bounded by $f(\ell_ {k})$.
\end{lemma}

\begin{proof}
(See Figure \ref{Nesting}.)
Clearly this claim holds for $k =1$.
Suppose that this claim holds for some $k \geq 1$.
Let $B_k$ be the 2-disk component of $Y_\ell \sm cl(\sq_{i =1}^k \Og_i)$ bounded by $f(\ell_ k)$.
From the construction of $(R_i)_{i =1}^\infty$, we see that $\Og_k$ and $\Og_{k+1}$ are adjacent along the loop $f(\ell_ k)$.
Then $\Og_{k+1}$ is a  sphere with $2g$ holes contained in $B_k$.
Therefore  $\Og_1, \Og_2, \dt, \Og_k, \Og_{k+1}$ are distinct components of $Y_\ell \sm \Lt'$, and $cl(\sq_{i = 1}^{k+1} \Og_i)$ is again a holed sphere in $Y_\ell$ bounded by finitely many loops of $\Lt'$.
Since $f(\ell_ k)$ and $f(\ell_ {k+1})$ are different boundary components of $\Og_{k+1}$,  there is  a component of $B_k \sm \Og_{k+1}$ bounded by $f(\ell _{k+1})$, which is a disk; then this component is also a component of  $Y_\ell \sm cl(\sq_{i = 1}^{k+1} \Og_i)$ bounded by $f(\ell_k)$. 
\end{proof}

By Lemma \ref{Omega}, $~\Og_i \st Y_\ell$ for all $i \geq 1$.
Therefore $\et(R_i) \st Conv(\Og_i) \st X_\ell$ (\,(iii)\,).
Lemma \ref{Omega} also implies that $(f(\ell_ i))_{i=1}^\infty$  is a sequence of $nested$ loops of $\Lt'$ contained in $Y_\ell$,
 that is, $Y_\ell \sm \sq_{i=1}^\infty f(\ell_ i)$ is a union of disjoint cylinders bounded by $f(\ell_ {i-1})$ and $f(\ell_ i)$ with $i = 1,2, \dt$ (see Figure \ref{nested}) and a point to which $f(\ell_i)$ converges as $i \to \infty$. 
Since $\et(R_i) \st Conv(\Og_i)$ and  $f(\ell_ {i-1})$ and $f(\ell_ {i})$ are boundary components of $\Og_i$, thus $\et(R_i)$ converges to the same limit point as $i \to \infty$ (\,(iv)\,).
\end{proof}

\begin{figure}[h]
\begin{overpic}[scale= .5
]{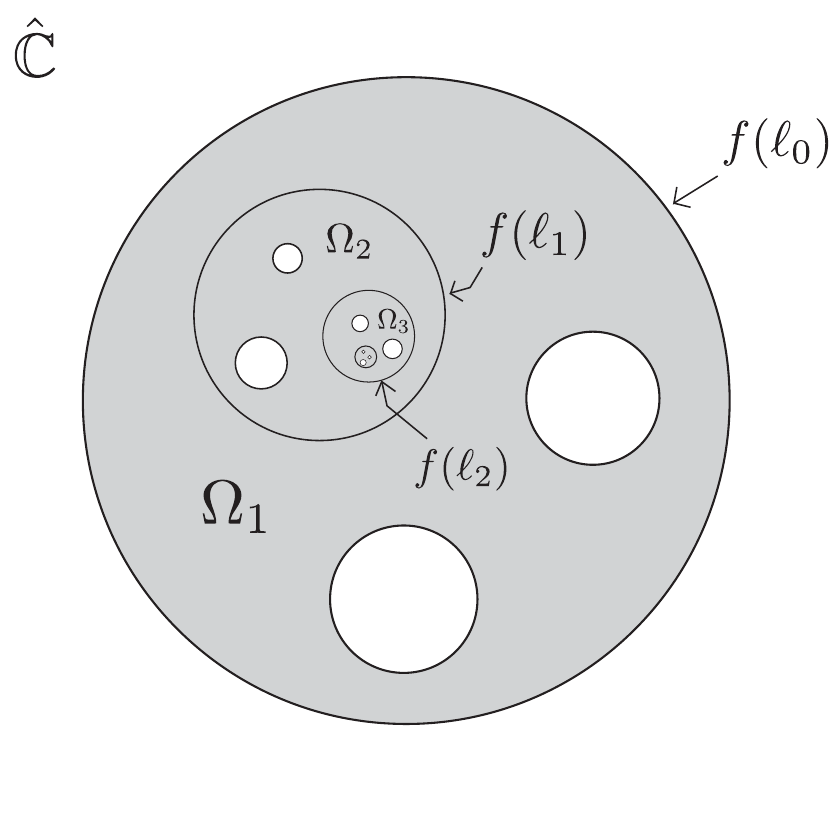}
        \put(78 ,10 ){\Large $F'$}  
      \end{overpic}
      \caption{The shaded region is $\sq_{i=1}^k \Og_i$.}
\label{Nesting}
 \end{figure}

Let $G_\ell$ be  the union of $cl(R_i)$ with $i =1,2, \dt$, obtained by Proposition \ref{R_i}.
Then $G_\ell$ is  an unbounded  connected  subsurface of $\St$ contained in $F_\ell$ and  bounded by infinitely many loops of $\Lt$.
Then $\pt G_\ell$ is a multiloop on $\St$ bounding  infinitely many disks of $\Dlt$ in $\Ht$.
Let $H_{G_\ell}$ be the closed (noncompact) subset of $\Ht$ bounded by these disks of $\Dlt$, so that $G_\ell = H_{G_\ell} \cap \St$. 
Then  $H_{G_\ell}$ is homeomorphic to the closed 3-disk $\D^3$ minus a point in $\pt \D^3$,  which corresponds to the limit point in Proposition \ref{R_i} (iv). 
Thus we can naturally identify $H_{G_\ell}$ with $H_\ell$ minus the puncture point $p(\ell)$.
\begin{figure}[h]
\begin{overpic}[scale=.7
]{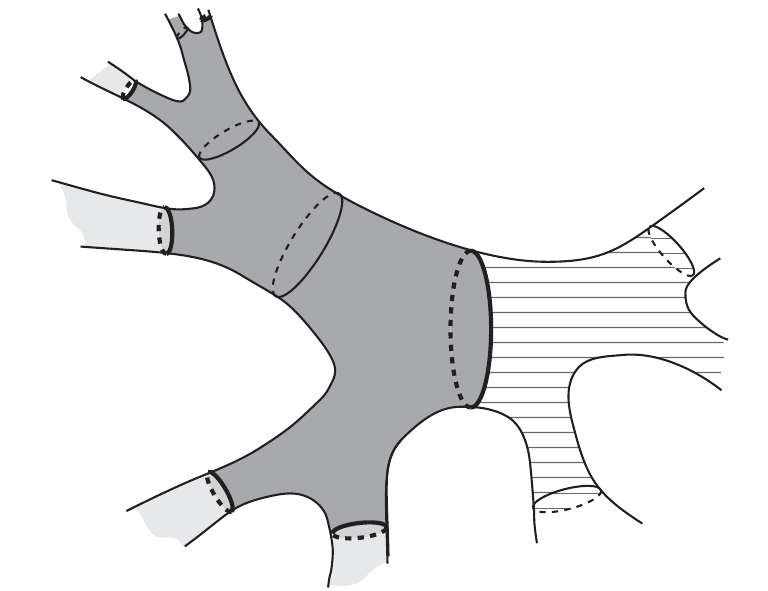}
 \put(75,35){\large $F$} 
 \put(48,30){\large $R_1$}  
\put(30,48){\large $R_2$}  
 \put(24,62){ $R_3$}  
 \put(42,58){$G_\ell = \cup_i R_i$}  
 \put(20, 30){\large $F_\ell$}
      \end{overpic}
\caption{}\label{G_l}
\end{figure}

The set of the ends of $\Ht$ is homeomorphically identified with the limit set of $\Gm$, which is a cantor set. 
Then, $\et\colon \Ht \to \hb$ continuously extends to the ends of $\Ht$, realizing the homeomorphism,  since $\et$ is $\td{\rho}$-equivariant.
In particular,  $\et$ takes the limit of $(R_i)_i^\infty$  to the limit point of $\Gm$ in Proposition \ref{R_i} (vi).
Therefore, by identifying  $H_\ell$ with the union of $H_{G_\ell}$ and its endpoint,  the embedding $\e_F\colon H_F \to \hb$ extends to an embedding of  $H_F \cup H_\ell$ into $\hb$, taking $p(\ell)$ to the limit point. 

By  Proposition \ref{R_i} (iii), for different boundary components $\ell $ of $F$, corresponding $H_{G_\ell}$ are contained in  different components of $\hb \sm Conv(F')$.
Therefore,  by applying such extension for all boundary components $\ell$ of $F$, we obtain an embedding $\e_{\Fh}\colon H_{\Fh} \to \hb$.
Then $\e_\Fh$ takes $H_{\Fh} \sm \{p_1, p_2, \dt, p_n\}$ to $\h^3$ and $p_1, p_2, \dt, p_n$ to different limit points of $\Gm$  on $\rs$.

Next we show that $\hb \sm {\rm Im}(\e_{\Fh})$ has a natural product structure. 
Let $\Fb = F \cup (\cup_\ell G_\ell)$, where the second union runs over all boundary components $\ell $ of $F$. 
Then $\Fb$ is a connected subsurface of $\St$ bounded by infinitely many loops of $\Lt$.
By the identification of $H_{G_\ell}$ and $H_\ell \sm p(\ell)$,  the inclusion $F \st \Fc$ extends to the inclusion $\Fb \st \Fc$.
Then  $\Fc \sm \Fb$ is a union of infinitely many disjoint 2-disks bounded by boundary components of $\Fb$.

On the other hand, when $\Fb$ is regarded as a subsurface of $\St$, its boundary bounds infinitely many disks of $\Dlt$ in $\Ht$.
For each boundary component $m$ of $\Fb$, let $D_m$ denote the disk of $\Dlt$ bounded by $m$.
(For the following discussion, see Figure \ref{product}.)
Then $\et(D_m)$ is a 2-disk and $\ett(m \times  [0,1])$ is an annulus embedded in $\hb$.
Since $\et(D_m)$ and $\ett(m \times  [0,1])$ share a boundary component,  their union, denoted by $E'_m$\,, is a 2-disk properly embedded in $\hb$.
Then the multidisk $\sq_m E'_m$ is property embedded in $\hb$, where the union runs over all boundary components $m$ of $\Fb$.  
We can see that  $\sq E'_m$ bounds the union of $\ett(\Fb \times [0,1])$ and ${\rm Im}(\e_\Fh)$, which is homeomorphic to a closed 3-disk.

For each boundary component $m$ of $\Fb$, let $D'_m$ be the (disk) component of $\rs \sm \ett(\Fb \times \{1\})$, bounded by $\ett(m \times \{1\})$.
Then the 2-disks $D_m'$ and $E_m'$ has disjoint interior and share the boundary component $\ett(m \times \{1\})$. 
Then $D'_m \cup E_m'$   is a 2-sphere, and it bounds a 3-disk $Q_m'$ in $\hb$.
Then $Q_m'$ is the component of $\hb \sm (\ett(\Fb \times [0,1]) \cup {\rm Im}(\e_\Fh))$ bounded by $E_m'$.
To give a natural product structure on $Q_m'$, 
choose a homeomorphism $\eta_m\colon D_m \times [0,1] \to Q_m'$ such that $\eta_m(D_m \times \{0\}) = \et(D_m)$,  $~\eta_m(D_m \times \{1\}) = D_m'$ and $\eta_m|_{\pt D_m \times [0,1]} = \ett|_{m \times [0,1]}$.
Let $\eta_\Fb: \Fb \times [0,1] \to \hb$ denote the restriction of $\ett: \St \times [0,1] \to \hb$ to $\Fb \times [0,1]$.
Then, since $\Fc \sm \Fb$ is the union of the disks bounded by boundary components $m$ of $\Fb$,
the product structures given by
$\eta_{\Fb}$ and $\eta_m$ match up along $\pt \Fb \times [0,1]$ and they yield  a homeomorphism $$\eta_\Fc\colon \Fc \times [0,1] \to \hb \,\sm \,[\,int({\rm Im}(\e_\Fh))\, \cup \, (\sq_{i=1}^n \e_\Fh(p_i)) \,]$$ such that $\eta_\Fc(\Fc \times \{0\}) = \e_\Fh(\Fc)$ and $\eta_\Fc(\Fc \times  \{1\}) = \rs \sm  \sq_{i=1}^n \e_\Fh(p_i)$. 

\begin{figure}[h]
\begin{overpic}[width=5in
]{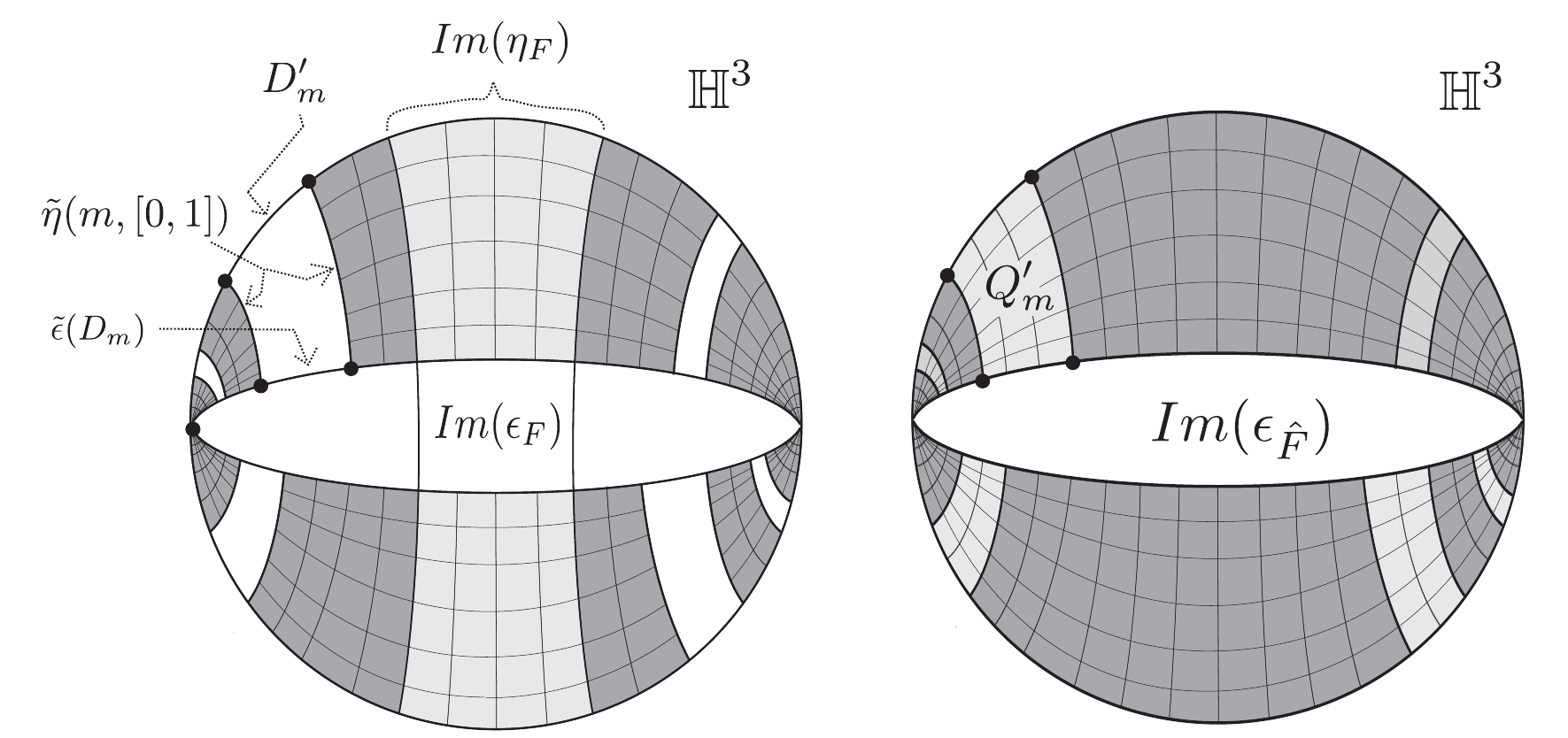}
\put(2,20){\tiny $\e_\Fh(p(\ell))$}
\put(16, 19){\small $\e_\Fh(H_\ell)$}
      \end{overpic}
\caption{A schematic for the product structure given by $\eta_\Fc$. On the left, the darkest region corresponds to $\ett(G_\ell \times [0,1])$'s, and, On the right, to $\ett(\Fb \times [0,1])$.} 
\label{product}
\end{figure}

Next, since $F$ is a subsurface of $\Fc$, we extend the almost good structure $\Ct|F = (f|_F, \rho_\id)$ on the holed sphere $F$ supported on $F'$ to a good structure on the punctured sphere $\Fc$.
Each boundary component $\ell$ of $F$ bounds a component of $\Fc \sm F$, which is  a 2-disk with  the puncture point $p(\ell)$.
Accordingly $\e_\Fh(p(\ell))$ is contained in the component of $\rs \sm F'$ bounded by the loop $f(\ell)$ that is also homeomorphic to a disk.  
Since $\e_{\Fh}$ takes different punctures points of $\Fc$ to different points on $\rs$. 
Therefore, as in \S \ref{AlmostGood}, we can uniquely extend the  almost good projective structure $\Ct|F$ on $F$ to a good projective structure $C_{\Fc} = (f_{\Fc}, \rho_{id})$ on $\Fc$ such that $f_\Fc(p(\ell)) = \e_\Fh(p(\ell))$ for all boundary components $\ell $ of $F$.
Then ${\it Supp}(C_\Fc)$ is $\rs \sm  \sq_{i=1}^n \e_\Fh(p_i)$.

We have seen that, in order to prove Proposition \ref{single}, it suffices to show that $\lf f_F\iv(\mu') \rf = \lf f\iv(\mu') \rf \cap F$ is a single loop on $F$.
If $\ell$ is a boundary component of $F$, let  $R_\ell$ be a component of $\Fc \sm F$ bounded by $\ell$.
 Then $R_\ell$ covers, via $f_\Fc$, a component of $\rs \sm F'$ minus the image of the corresponding puncture point. 
Since $\mu'$ is contained in the interior of $F'$, we have $f_\Fc\iv(\mu') \cap R_\ell = \emptyset$.
Thus the proof of Proposition \ref{single} is reduced to show:

\begin{proposition}\label{reduced}  
$\lf f_\Fc\iv(\mu') \rf$ is a single loop on $\Fc$.
\end{proposition}

Recall that $\ld$ and $\mu'$ are the loops on $\Fc$ and $\rs$, respectively, such that $\mu'  = \ett(\ld \times \{1\})$, which is equal to $\eta_{\Fc}(\ld\times \{1\})$.
Also $\lambda' = f(\ld)$ is the loop of $\Lt$ satisfying $\e_{\Fh}(\ld) \st  Conv(\ld') \cong \ol{\h^2}$.
Let $D'_\ldp =  Conv(\ld')$.

\begin{lemma}\label{phi=id}
$\e_\Fh\iv(D'_{\ld'})$ is a multidisk properly embedded in $H_\Fh$ bounded by the multiloop $\lf f_{\Fc}\iv(\ld') \rf$.
\end{lemma}
\begin{proof}
Recall that $H_F$ and all $H_\ell$ have disjoint interiors, where $\ell$ varies over all boundary components of $F$ and that $H_\Fh$ is the union of $H_F$ and all $H_\ell$.
By Lemma \ref{F_0'}, $~\ld'$ is contained in $int(F'_0)$ and thus in $int(F')$.
Then $D'_{\ld'}$ is contained in $int(Conv(F'))$.
If $\ell$ is a boundary component of $F$, then $X_\ell$ is a component of $\hb \sm Conv(F')$ and its closure contains $\e_{\Fh}(H_\ell)$, therefore $\e_{\Fh}\iv(D'_{\ld'}) \cap H_\ell = \emptyset$.
Hence $\e_{\Fh}\iv(D'_{\ld'}) \st H_F$.

For each boundary component $\ell$ of $F$, the once-punctured disk $R_\ell$ (bounded by $\ell$) covers the once-punctured disk $Y_\ell \sm \e_\Fh(p(\ell))$ in $\rs$ via $f_\Fc$.
Since $Y_\ell \cap int (F') = \emptyset$, thus $R_\ell \cap  f_\Fc\iv(\ld') = \emptyset$.
Therefore $\lf f_\Fc\iv(\ld') \rf \st F$.

Therefore it suffices to show that $\e_F\iv(D'_{\ld'})$ is a multidisk properly embedded in $H_F$ bounded by $\lf f_F\iv(\ld') \rf$.
By Proposition \ref{e} (i),  $~\e_F\iv(D'_{\ld'})$ is a union of finitely many (disjoint) disks of $\Dlt \cap H_F$.
By Proposition \ref{e} (iii), a disk $D$ of $\Dlt \cap H_F$ embeds into $D'_\ldp$ if and only if  $f_\Fc(\pt D) = f(\pt D) = \ld'$.
This completes the proof.
\end{proof}

For a surface $\Sigma$, we let $P_\Sigma$ denote the set of all punctures of $\Sigma$.
Let $L_{\ld'} = \lf f_\Fc\iv(\ld') \rf$, which is a multiloop on $\Fc$.

\begin{lemma}\label{almostgood}
Let $X$ be a component of $\Fc \sm L_{\ld'}$.
Then $C_\Fc|X$ is an almost good genus-zero surface fully supported on a 2-disk with finitely many punctures $f_\Fc(P_X)$, where the 2-disk is the component of $\rs \sm \ld'$ containing $f_\Fc(P_X)$.
\end{lemma}

\begin{proof}
Since, by Lemma \ref{phi=id}, $L_{\ld'}$ is the boundary of the multidisk $\e_\Fh\iv(D'_\ldp)$ in $H_
\Fh$, $~\e_\Fh\cn H_\Fh \to \hb$ embeds $X$ into a single component $H$ of $\hb \sm D'_\ldp$.
Then $H \cap \rs$ is a component of $\rs \sm \ld'$, which is a round 2-disk.  
If  $p \in P_\Fc$ (in particular if $p \in P_X$), then  $f_\Fc(p) = \e_\Fh(p)$. 
Since $\e_\Fh$ is an embedding, different points of $P_X$ map to different points  in the interior of $H \cap \rs$.
In addition, all boundary components of $X$ cover $\ld'$ via $f_\Fc$. 
Therefore $C_\Fc|X$ is an almost good genus-zero surface fully supported on the punctured disk $(H \cap \rs) \sm f_\Fc(P_X)$.
\end{proof}

{\it Step 3.}
We have constructed a good projective structure $C_{\Fc} = (f_\Fc, \rho_{id})$ on a puncture sphere $\Fc$; 
an embedding  $\e_{\Fh}\cn  H_\Fh \to \hb$, where the sphere $\Fh$ that is the union of $\Fc$ with its punctures and $H_\Fh$ is  a 3-disk bounded by $\Fh$; 
an embedding $\eta_{\Fc}\cn \Fc \times [0,1] \to \hb$ so that $\eta_\Fc| (\Fc \times \{0\})$ is the restriction of $\e_\Fh$ to $\Fc \st \pt H_\Fh$ and $\eta| (\Fc \times \{1\})$ is a homeomorphism onto the support of $C_{\Fc}$. 
 
Let  $\phi\colon \ol{\h^3} \to \ol{\h^3}$ be a homeomorphism.
Then we can transform $C_{\Fc}, \e_{\Fh}, \eta_{\Fc}$ by postcomposing with $\phi$, and, since $\phi$ is a homeomorphism,  this transformation preserves essential topological properties  of $C_{\Fc}, \e_{\Fh}, \eta_{\Fc}$ and correspondences between them:
Namely, we let 
\begin{eqnarray*}
f_\phi &=& \phi \cc f_{\Fc}\colon \Fc \to \rs\\
C_\phi &=& (f_\phi, \rho_\id) \\
  \e_\phi &=& \phi \cc \e_{\Fh}\colon H_{\Fh} \to \hb \\
 \eta_\phi &=& \phi \cc \eta_{\Fc}\colon \Fc \times [0,1] \to \hb. 
\end{eqnarray*}

Then
\begin{proposition}\label{phi}
There exists a homeomorphism $\phi\colon \ol{\h^3} \to \ol{\h^3}$ such that 
\begin{itemize}
\item[(i)] $\lf f_\phi \iv(\ld') \rf$ is a single loop isotopic to $\ld$ on $\Fc$, and 
\item[(ii)]  $\ld'$ is isotopic to $\e_\phi (\ld)$ in the product region ${\rm Im} (\eta_\phi)$.
\end{itemize}
\end{proposition}
(An example of such a homeomorphism $\phi$ is illustrated in Figure \ref{ExamplePhi}.)
First we show that it suffices to construct $\phi$ satisfying (i) and, instead of $(ii)$,
\begin{itemize}
\item[(II)] $\e_{\phi}\iv(D'_\ldp)$ is a multidisk properly  embedded in $H_\Fh \sm P_\Fc$ and bounded by $\lf f_\phi\iv(\ld') \rf$.
\end{itemize}
Assume that there is a homeomorphism $\phi: \hb \to \hb$ satisfying (i) and (II).
Then, by (i), $~\lf f_\phi \iv(\ld') \rf$ is a loop on $\Fc$ isotopic to $\ld$.
Thus the loop $\e_\phi (\ld)$ is isotopic to the loop $\e_\phi (\lf f_\phi \iv(\ld') \rf)$ on the surface $\e_\phi (\Fc)$.
By (II), $~D'_\ldp \cap {\rm Im} (\e_\phi)$ is a single disk bounded by $\e_\phi (\lf f_\phi \iv(\ld') \rf)$.
Therefore $\e_\phi (\lf f_\phi \iv(\ld') \rf)$ and $\ld'$  bound an annulus properly embedded in ${\rm Im} (\eta_\phi)$\,;
 in particular, they are isotopic in ${\rm Im}( \eta_\phi)$.
 Thus (ii) holds.

\begin{figure}[htbp]
\begin{center}
\includegraphics[width=5in]{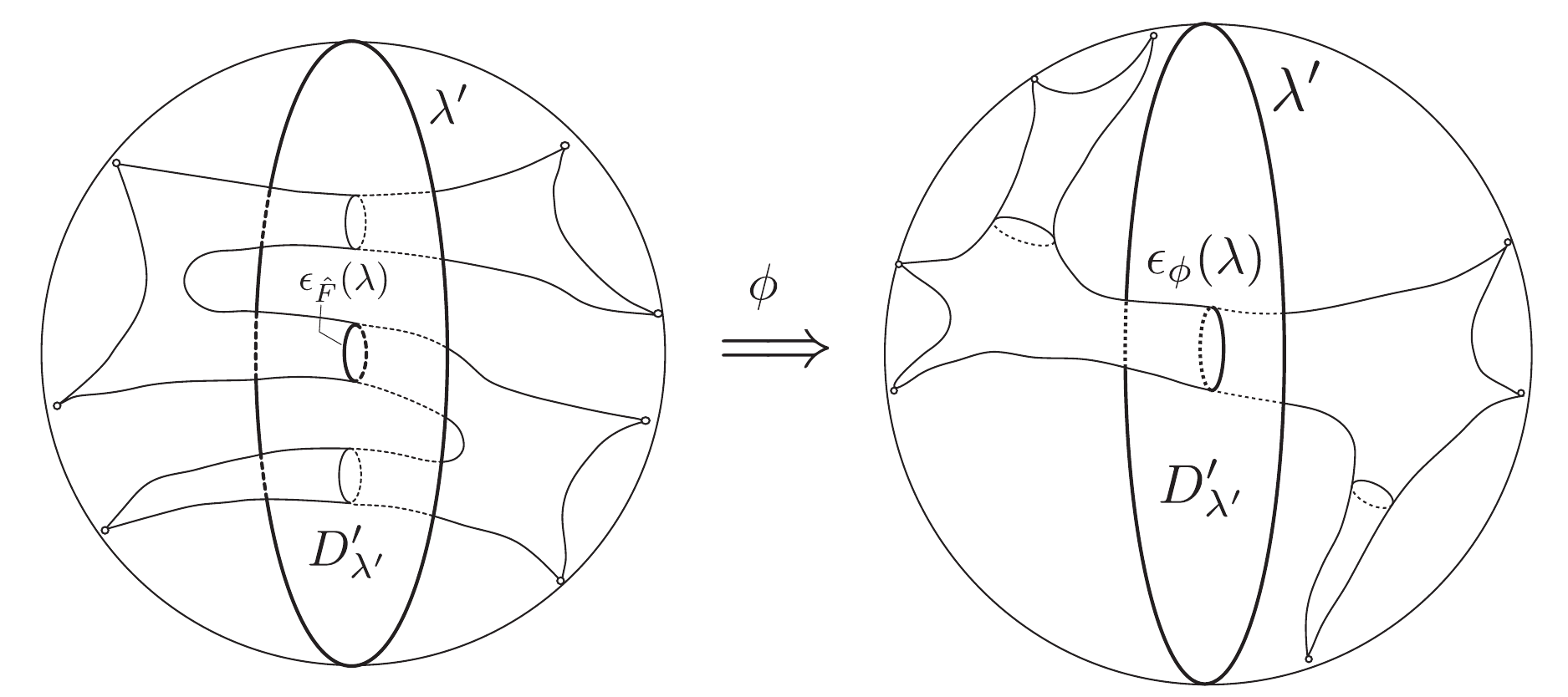}
\caption{An example of $\phi$ in Proposition \ref{phi}.}
\label{ExamplePhi}
\end{center}
\end{figure}

We will further reduce the proof of Proposition \ref{phi} to an induction.
For this, suppose that there is a homeomorphism ${\phi_1}\colon \ol{\h^3} \to \ol{\h^3}$ satisfying:\\
(I)~  $\lf f_{\phi_1}\iv(\ld') \rf =: L_{\phi_1}$ is a multiloop on ${\Fc}$ containing a loop $\ld_{\phi_1}$ isotopic to $\ld$, \\
(II) ~$\e_{\phi_1}\iv(D'_\ldp) =: \Dl_{\phi_1}$ is a multidisk properly  embedded in $H_\Fh \sm P_\Fc$ and bounded by $L_{\phi_1}$, and \\
(III)~ if $X$ is a component of $\Fc \sm L_{\phi_1}$, then $C_{\phi_1}|X$ is an almost good surface fully supported on the punctured disk $B_X \sm f_{\phi_1}(P_X)$,  where $B_X$ is the component of $\rs \sm \ld'$ containing $ f_{\phi_1}(P_X)$.
 
In particular, if $\phi_1 = id$, then, (I) holds since $f(\ld) = \ld'$,  (II)  by Lemma \ref{phi=id}, and (III) by Lemma \ref{almostgood}.

A loop $\ell $ of a multiloop $N$ on a surface $\Sigma$ is {\it outermost} if $\ell $ is a separating loop and a component of $\Sigma \sm \ell$ contains no loops of $N$; this component is called an {\it outermost component}. 
Every loop of $L_{\phi_1}$ on $\Fc$ is separating since $\Fc$ is a planar surface.

\begin{lemma}\label{InductionLemma}\label{induction} 
Let $\ell $ be an outermost loop of $L_{\phi_1}$ on $\Fc$.
Then there is a homeomorphism $\phi_2\colon \overline{\mathbb{H}^3} \to \overline{\mathbb{H}^3}$ satisfying the following properties:\\
(I')  $\lf f_{\phi_2}\iv(\ld') \rf  =: L_{\phi_2}$ is isotopic to $L_{\phi_1} \sm \ell$ on $\Fc$,\\ 
(II') $\e_{\phi_2}\iv(D'_\ldp) =: \Dl_{\phi_2}$ is a multidisk properly  embedded in $H_\Fh \sm P_\Fc$ and  bounded by $L_{\phi_2}$, and\\
(III') if $X$ is a component  of $\Fc \sm L_{\phi_2}$, then $C_{\phi_2}|X$ is an almost good surface fully supported on the punctured disk $B_X \sm f_{\phi_2}(P_X)$,  where $B_X$ is the component of $\rs \sm \ld'$ containing $f_{\phi_2}(P_X)$. \\
\end{lemma}

This lemma implies Proposition \ref{phi}: 

\begin{proof}[Proof of Proposition \ref{phi} with Lemma \ref{induction} assumed] 
Clearly Conclusions (I'), (II'),  (III') on $\phi_2$ correspond to Assumptions (I), (II), (III) on $\phi_1$, respectively. 
Therefore, starting from the base case that ${\phi_1} = id$, we repeatedly apply Lemma \ref{induction} and  inductively reduce the number of the loops of $L_{\phi_1}$.
Since $L_{\phi_1}$ always contains at least two outer most loops, we can proceed out induction, preserving the loop isotopic to $\ld_{\phi_1}$.
Thus, when there is exactly one loop isotopic to $\ld$ left, we obtain $\phi_2$ satisfying (i), (II).
Hence this $\phi_2$ realizes Proposition \ref{phi} (as discussed above).
\end{proof}

\begin{proof}[Proof (Lemma \ref{InductionLemma})]
First, we  construct a homeomorphism $\psi: \rs \to \rs$ such that $\psi \cc \phi_1|_{\rs}$ is a homeomorphism from $\rs$ to itself satisfying (I') and (III').
Later, we extend $\psi$ to a homeomorphism from $\hb$ to itself, such that $\psi \cc \phi_1\colon \hb \to \hb$  satisfies (II') (with $\psi \cc \phi_1 = \phi_2$).

Let $D_\ell$ be the disk of $\Dl_{\phi_1}$ bounded by $\ell$.
Let $Q$ be the outermost component of $\Fc \sm L_{\phi_1}$ bounded by $\ell $. 
Let $H_Q$ be the closure of the component of $H_\Fh \sm \Dlt$ bounded by $Q$.
Then $H_Q \cap \Fc = Q \cup \ell$, and,
by (II),  $~\e_{\phi_1}(H_Q)$ is contained in the closure of a component of $\hb \sm D'_{\ld'}$.
Let $R$ be the component of $\Fc \sm L_{\phi_1}$ adjacent to $Q$ along $\ell $.
Then $\pt R$ bounds a unique multidisk consisting of disks of $\Dl_{\phi_1}$.
Similarly let $H_R$ be the closure of the component of $H_\Fh \sm \Dl_{\phi_1}$ bounded by this multidisk; then $H_R \cap \Fc = R$.

Since $\e_{\phi_1}\cn H_\Fh \to \hb$ is an embedding, we regard subsets of $H_\Fh$ also as their images in $\hb$ under $\e_{\phi_1}$.
By this convention, the  disk $D_\ld$ separating $H_Q$ and $H_R$ is contained in the disk $D'_\ldp$, and the interiors of $H_Q$ and $H_R$ are contained in the different components of $\hb \sm D'_\ldp$.  
Our gold is to isotope $H_\Fh$ in $\hb$ keeping the puncture of $\Fc$ on $\rs$ so that $H_Q \cup H_R$ moves to   the single component of $\hb \sm D'_\ldp$ containing $H_R$ while  $H_\Fh \sm (H_Q \cup H_R)$ is fixed.
Such an isotopy induces a desired homeomorphism $\psi\cn \hb \to \hb$.

Let $B(\ld', +)$ and $B(\ld', -)$ denote the components of $\rs \sm \ld'$. 
Since $Q$ and $R$ are adjacent, we can assume that ${\it Supp} (C_{\phi_1}|Q) = B(\ld',+) \sm P_Q$ and ${\it Supp}(C_{\phi_1}|R) = B(\ld', -) \sm P_R$.
The full supports ${\it Supp} (C_{\phi_1}|Q)$ and ${\it Supp}( C_{\phi_1}|R )$ are punctured disks, and thus $Q$ and $R$ have at least one puncture  by the definition of almost good  projective structures. 
Let $q_1, q_2, \dt, q_h$ be the punctures of $Q$.
Let $\ld''$ be a round circle in the punctured disk $B(\ld', -) \sm P_\Fc$ parallel to $\ld' = \pt B(\ld', -)$ so that $\ld''$ and $\ld'$ bounds a round annulus $A$ in $B(\ld', -) \sm P_\Fc$. 
Similarly, let $B(\ld'', +)$ and $B(\ld'', -)$ be the components of $\rs \sm \ld''$ so that $B(\ld'', +) = B(\ld', +) \cup A'$ and $B(\ld'', -) = B(\ld', -) \sm A'$.

Let $a_1, a_2, \dt, a_h$ be disjoint paths on the punctured disk $B(\ld'', +) \sm P_\Fc$ such that $a_i$  connects the point $q_i$ to a point $r_i$ in $int(A')$ for each $i \in  \{1, 2, \dt, h\}$. 
We can in addition assume that each $a_i$ transversally intersects $\ld'$ in a single point. 
Pick a 2-disk neighborhood $U_i$ of $a_i$ in $B(\ld'', +)$ such that $U_i$ contains no points of $P_\Fc$ except $q_i$ and $U_i$ intersects $\ld'$ in a single arc (Figure \ref{V_i}).
In addition assume that $U_1, U_2, \dt, U_h$ are disjoint. 
For each $i \in \{1,2, \dt, h\}$, choose a homeomorphism  $\psi_i\colon \rs \to \rs $ {\it supported} on $U_i$ such that $\psi_i(q_i) = r_i$ (i.e.\ $\psi_i$ is the identity map on $\rs \sm U_i$).
Let $\psi = \psi_1 \cc \psi_2 \cc \dt \cc \psi_h\colon \rs \to \rs$.
Then $\psi$ is a homeomorphism supported on $\sq_{i=1}^h U_i =: U$.
Then there is an isotopy $\xi_\psi\cn [0,1] \times \rs \to \rs$ of $\rs$ supported on $U$ between $\psi$ and the identity map.

The restriction of $f_{\phi_1}$ to $\Fc \sm f_{\phi_1}\iv(P_\Fc)$ is a covering map onto $\rs \sm P_\Fc$.
Since $A'$ is disjoint from $P_\Fc$, via this covering map, $A'$ lifts to finitely many copies $A_1, A_2, \dt, A_J$ of $A$ in $\Fc$.
Therefore, letting $L_{\phi_1}^-  = \lf f_{\phi_1}\iv(\ld'') \rf$,  each $A_j\,(1< j < J)$ is bounded by a loop of $L_{\phi_1}$ and a loop of  $L^-_{\phi_1}$. 
Then  the components of $\Fc \sm L_{\phi_1}$   bijectively correspond to the components of $\Fc \sm L_{\phi_1}^-$ by adding or subtracting   $A_1, A_2, \dt, A_J$. 
By Assumption (III), if $X$ is a component of $\Fc \sm L_{\phi_1}$, then $C_{\phi_1}|X$ is an almost good surface fully supported on either $B(\ld', +) \sm P_X$ or $B(\ld', -) \sm P_X$.
Accordingly, if $X$ is a component  of $\Fc \sm L_{\phi_1}^-$, then $C_{\phi_1}|X$ is a good genus-zero surface supported on either $B(\ld'', +) \sm P_X$ (Type One) or $B(\ld'', -) \sm P_X$ (Type Two).
Since $Q$ is outermost, $Q$ has exactly one boundary component, which is $\ell  $.   
Thus there exists a unique $j \in \{1,2, \dt, J\}$ such that $A_j$ adjacent to $Q$ along $\ell$. 
Then $Q \cup A_j$ is a component of $\Fc \sm L_{\phi_1}^-$ of Type 1.

Let $X$ be a component of $\Fc \sm L^-_{\phi_1}$  of  Type Two. since $\ld'$ and ${\it Supp}(C_{\phi_1}|X)$ are disjoint, by Lemma \ref{support}, $L_{\phi_1} \cap X = \emptyset$.
Since $\psi$ is supposed on the complement of ${\it Supp}( C_{\phi_1}|X )$, we have $L_{\phi_2} \cap X = \emptyset$. 

Next let  $X$ be a component of  $\Fc \sm L^-_{\phi_1}$ of  Type One.
Then $f_\Fc\iv(A') \cap X$ is a regular neighborhood of $\pt X$, and it is a union of some $A_j$'s.
Thus $X \cap L_{\phi_1}$ is a multiloop on $X$ isotopic to $\pt X$.
Suppose that $X$ does $not$ contain $Q$.
Since $U$ is disjoint from  $P_\Fc \sm P_Q$, then $P_X$ is disjoint from $U$. 
Thus ${\it Supp}( C_{\phi_1}|X )$ contains  $U$.
Since the isotopy $\xi_\psi$ of $\rs$ is supported on $U$, lifting $\xi_\psi$  via $dev(C_{\phi_1}|X)$, we obtain an isotopy  of $X \cap L_{\phi_1}$ to $X \cap L_{\phi_2}$ on $X$. 

Last suppose that $X$ contains $Q$. 
Then $X = Q \cup A_j$ for some $j \in \{1,2, \dt, J\}$ and  $X \cap L_{\phi_1} = \ell$.
The homeomorphism  $\psi$ moves $P_X \st B(\ld',+)$ to $int(A') \st B(\ld', -)$.
Then $C_{\phi_2}|X$ is a good surface fully supported on $B(\ld'', +) \sm \psi(P_X)$.
Since the support $B(\ld'', +) \sm \psi(P_X)$ contains the disk $B(\ld', +)$ bounded by $\ld'$, therefore, by Lemma \ref{support}, $~X \cap L_{\phi_2} = \emptyset$.
We have seen that on all other components $X$ of $\Fc \sm L_{\phi_1}^-$, the multiloop $L_{\phi_1} \cap X$ is isotopic to $L_{\phi_2} \cap X$. 
Hence   $L_{\phi_2}$  is isotopic to $L_{\phi_1} \sm \ell$ on $\Fc$ (\,(I')\,).

If $Z$ is a component of $\Fc \sm L_{\phi_2}$, 
then each boundary component of $Z$ covers $\ldp$ via $dev( C_{\phi_2}|Z ) = f_{\phi_2}|_Z$.  
In addition, under the isotopy of $\Fc$ moving $L_{\phi_2}$ to $L_{\phi_1} \sm \ell$, the component $Z$ is isotoped to either the union $Q \cup R$ or  a component of $\Fc \sm L_{\phi_1}$ that is $not$  $Q$ or $R$.
First suppose that $Z$ is isotopic to $Q \cup R$.
We have $\psi(P_Q) \st A'$ and $\psi(P_R) = P_R \st B(\ld'', -)$.
Thus, since $\psi$ is a homeomorphism, $dev(C_{\phi_2}|Z)\colon Z \to \hb$ takes all punctures of $Z$ to distinct points in  $B(\ld', -)$.
Therefore, since all boundary components of $Z$ cover $\ld'$,  thus $C_{\phi_2}|Z$ is an almost good  surface whose support is $B(\ld', -) \sm \psi(P_Z)$.
Next suppose that $Z$ is isotopic to a component $Y$ of $\Fc \sm L_{\phi_1}$.
Then, since $\psi$ fixes $P_Z$, obviously $\psi(P_Z)$ is contained in either $B(\ld', +)$ or $B(\ld', -)$. 
Therefore $C_{\phi_2}|Z$ is an almost good surface whose support is, accordingly, $B(\ld', +) \sm \psi(P_Z)$ or $B(\ld', -) \sm \psi(P_Z)$.
 (Moreover it is easy to see that  $C_{\phi_2}|Z = C_{\phi_1}|Y$ since a projective structure is defined up isotopy of its base surface.)
Thus (III') holds.

We now extend the homeomorphism $\psi\colon \rs \to \rs$ to a homeomorphism $\psi\colon \hb \to \hb$.
Let $H'(\ld', +) = Conv(B(\ld', +))$ and $H'(\ld', -) = Conv(B(\ld', -))$, so that  $\hb \sm D'_\ldp = H'(\ld', +) \sq H'(\ld',-)$. 
Similarly, let $H'(\ld'', +) = Conv(B(\ld'', +))$ and $H'(\ld'', -) = Conv(B(\ld'', -))$, so that  $\hb \sm Conv(\ld'') = H'(\ld'', +) \sq H'(\ld'',-)$.

For each $i \in \{1,2,\dt, h\}$, let $V_i$ be a neighborhood of $a_i$ in $H'(\ld'', +)$ homeomorphic to a closed 3-disk that is a ``natural extension'' of $U_i$\,:
Namely we require
\begin{itemize}
\item[(i)] $V_i \cap \rs = U_i$,
\item[(ii)] $V_i \cap D'_\ldp$ is a 2-disk,
\item[(iii)]  $\pt V_i \cap H_\Fh$ is contained in $H'(\ld', +)$, and it is  the union of a closed 2-disk $K_i$  in $\h^3$ and the point $q_i$ in $\rs$,
\item[(iv)] $V_1, V_2, \dt, V_n$ are disjoint, and
\item[(v)] $\pt V_i, D_{\ldp}, \pt H_\Fh$ intersects transversally
\end{itemize}
(see Figure \ref{V_i}).
By (i) and (ii), $\pt( V_i \cap D'_\ldp)$ is a circle that is the union of  an arc properly embedded in $D'_\ldp$ and the arc $U_i \cap \ld'$.
By (iii), $V_i \cap H_\Fh$ is a closed 3-disk $T_i$.

\begin{figure}[htbp]
\includegraphics[width=3in]{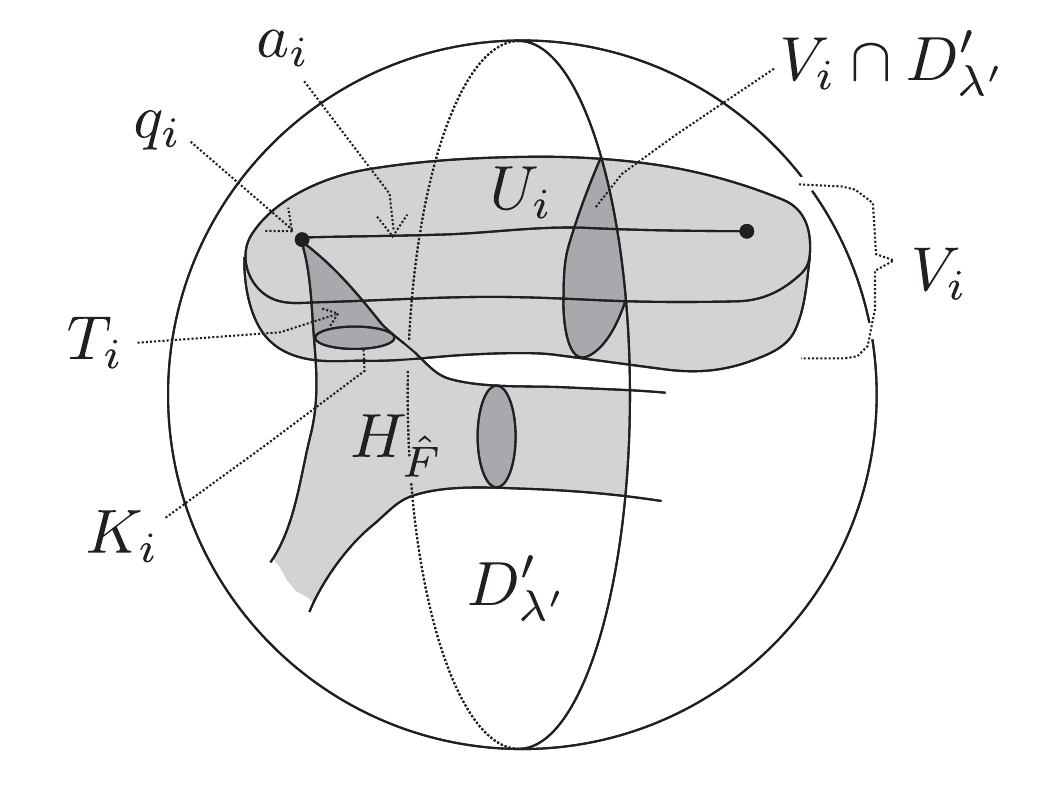}

\caption{A picture of $V_i$.}
\label{V_i}
\end{figure}

Then, for each $i = 1, 2, \dt, h$, we can extend the homeomorphism $\psi_i\colon \rs \to \rs$ supported on $U_i$ to a homeomorphism $\sigma_i\colon \hb \to \hb$ supported on $V_i$ so that the boundary of $\sigma_i (T_i)$ transversal intersects  $D'_\ldp$ in a single circle;
let $D_{T_i}$ be the 2-disk $\sigma_i (T_i) \cap D'_\ldp$ bounded by the circle
(see (i) and (ii) in Figure \ref{isotopies}).
Accordingly, $\sigma = \sigma_1 \cc \sigma_2 \cc \dt \cc \sigma_h$ has extended to a homeomorphism from $\hb$ to itself.
In addition the isotopy $\xi_\sigma$ of $\rs$ extends to an isotopy of $\hb$ supported on $V = \sq\sp h_{i =1}     V_i$ between $\sigma$ and the identity map.  

Since we have isotoped $H_\Fh (= \e_{\phi_1}(H_\Fh))$ by $\xi_\sigma$, we now regard subsets of $H_\Fh$  as their images under the embedding $\sigma \cc \e_{\phi_1}\cn H_\Fh \to \hb$.
(In partiular, the puncture $q_i$ of $\Fc$ is regarded as the point  $r_i \in \rs$.)
Since $\sigma$ is a homeomorphism,  $\eta_\sigma := \sigma \cc \eta_\Fc\colon \Fc \times [0,1] \to \sigma \cc {\rm Im}(\eta_\Fc) \st \hb$ is a homeomorphism that takes $\Fc \times \{0\}$ to $\Fc$ (= \ $\sigma \cc \e_\Fh(\Fc)$)  and $\Fc \times \{1\}$ to $\rs \sm P_\Fc~$ (= \ $\rs \sm \sigma \cc \e_\Fh(P_\Fc))$.  

Choose a puncture $r$ of $R$.
Let $\beta_1, \beta_2, \dt, \beta_h$ be disjoint paths on the punctured sphere $\Fc \, (\st \hb)$ such that $\beta_i$ connects $r_i (= q_i)$ to $r$ for each $i = 1,2, \dt, h$ (see Figure \ref{isotopies} (ii)).
In addition, we can assume that $\beta_i$ and $T_j$ are disjoint if $i \neq j$ and that $\beta_i$ transversally intersects $\pt K_i$ and $D'_\ldp$ minimally (i.e. it intersects $\pt K_i$ in a single point and $D'_{\ldp}$ in two points).
Then, since $L_{\phi_1}$ is contained in $D_\ld'$, then $\beta_i$ is contained in the connected subsurface $Q  \cup R$ of $\Fc$ and it intersects $\ell$ in a single point. 

The union of the product region ${\rm Im} (\eta_\sigma)$ and the puncture points $P_\Fc$ is naturally homeomorphic to the product $F_\Fh \times [0,1]$ with the  interval $[0,1]$ above each puncture point collapsed to a single point. 
Then for each $i \in \{1,2,\dots, h\}$, since $\beta_i$ is a path on $\Fc$ connecting the puncture points $r_i$ and $r$,
then $\eta_\sigma(\beta_i \times [0,1]) \cup \{r_i, r\} =: E_i$ is a 2-disk properly embedded in ${\rm Im} (\eta_\sigma) \cup P_\Fc$.
Then the disks $E_1, E_2, \dt, E_h$ share the point $r$, and $E_1 \sm \{r\}, E_2 \sm \{r\}, \dt, E_h \sm \{r\}$ are disjoint.
In addition, we can assume that each $E_i$ transversally intersects $D'_\ldp$, if necessarily, by a small isotopy of $E_i$ in ${\rm Im}(\eta_\sigma)$. 
Then there is a unique arc component $e_i$ of $E_i \cap D'_\ldp$  connecting  the points of $\beta_i \cap D'_\ldp$.
Then $e_i$ connectes the disks $D_{T_i}$ and $D_\ell$. 
Take disjoint (small) regular neighborhoods   $N(E_1), N(E_2), \dt, N(E_h)$ of $E_1 \sm \{r, r_1\}, E_2 \sm \{r, r_2\}, \dt, E_h \sm \{r, r_h\}$, respectively, in ${\rm Im}(\eta_{\sigma})$, such that $N(E_i)  \cap V_j = \emptyset$ for all $i, j \in \{1,2, \dt, h\}$ with $i \neq j$.
In addition, we can assume that  $N(E_i) \cap D'_\ldp$ is a regular neighborhood of the one-dimensional manifold $E_i \cap D'_\ldp$ properly embedded in $D'_\ldp$ and that $N(E_i) \cap \pt K_i$ is a single arc, which is a regular neighborhood of the single point $E_i \cap \pt K_i$ in the circle $\pt K_i$.
In particular, there is  a component of $N(E_i) \cap D'_\ldp$ that is  a regular neighborhood $N(e_i)$ of the  arc $e_i$ in $D'_\ldp$.
Then  $N(e_i)$ is a rectangular strip connecting $D_\ell$ and $D_{T_i}$, and $N(e_i) \cup D_{T_i}$ is a (topological) 2-disk properly embedded in the 3-disk $T_i \cup N(E_i) \cup \{r, r_i\}$.

Take a small regular neighborhood of $N(e_i) \cup D_{T_i}$ in $T_i \cup N(E_i)$.
Then,  since $N(e_i) \cup D_{T_i}$ is contained in $D'_\ldp$, we parametrize this regular neighborhood as $(N(e_i) \cup D_{T_i}) \times [-1,1]$ so that $(N(e_i) \cup D_{T_i}) \times [-1, 0)  \st H'(\ld', -)$ and $(N(e_i) \cup D_{T_i}) \times (0, 1] \st H'(\ld', +)$.
Then $N(e_i) \times \{-1\}$ is a 2-disk properly embedded in  the 3-disk $N(E_i)$, splitting $N(E_i)$ into two 3-disks.
Let $\np(E_i)$ denote the one of these two 3-disks containing $N(e_i) \times [-1,1]$(see Figure \ref{isotopies} (ii)).
Then $int(\np(E_i))$ is disjoint from $int(H_\Fh)$, and  the intersection of $\pt \np(E_i)$ and $\pt H_\Fh$ is  a 2-disk  contained in the subsurface $Q \cup R$ of $\Fc$.
Therefore we can isotope $H_\Fh$ to $H_\Fh \cup \np(E_i)$ in $\hb$, fixing $H_\Fh \sm (H_Q \cup H_R)$ and $\pt \h^3$.
Accordingly we now regard subsets of $H_\Fh$ as their images in $\hb$ under the composition of 
$\sigma \cc \e_\Fh \cn H_\Fh \to \hb$ and this isotopy.

Since $\np(E_i)$ and $T_i$ have  disjoint interiors and their boundaries intersect in a 2-disk, therefore $\np(E_i) \cup T_i$ is a 3-disk.
Then  $(N(e_i) \cup D_{T_i}) \times \{-\frac{1}{2}\}$ is a 2-disk properly embedded in the 3-disk $\np(E_i) \cup T_i$, 
and therefore it separates $\np(E_i) \cup T_i$ into two closed 3-disks.
Let $D_i^+$ and $D_i^-$ denote these two 3-disks so that $D_i^+$ contains $(N(e_i) \cup D_{T_i}) \times [-\frac{1}{2}, 1]$ and $D_i^-$ contains $(N(e_i) \cup D_{T_i}) \times [ -1, -\frac{1}{2}]$ (see Figure \ref{isotopies} (iii)).
Then $D_i^+$ is contained in $\h^3$.
On the other hand $D_i^-$ is contained in $H'(\ld', -)$ and it intersects $\rs$ at the point $r_i$.
Then $D_i^+$ and $cl(H_\Fh \sm D_i^+)$ are 3-disks with disjoint interiors, and their boundaries share a single 2-disk. 
Thus, similarly, there is an isotopy of $\hb$ supported on a small neighborhood of $D_i^+$ that moves $H_\Fh$ to $cl(H_\Fh \sm D_i^+)$; in particular it  fixes $H_\Fh \sm (H_Q \cup H_R)$ and $\rs$ (see Figure \ref{isotopies} (iv)).

For each $i \in \{1,2, \dots, h\}$, modify $\sigma_i\colon \h^3 \to \h^3$  by postcomposing it with be the homeomorphism induced by  the composition of  the isotopies of $\hb$ applied after $\sigma_i$ (as in Figure \ref{isotopies}), which transforms the initial image $\e_{\phi_1}(H_\Fh)$ to the 3-disk $cl(H_\Fh \sm D_i^+)$ in the preceding paragraph. 
Then we still have $\sigma_i|\rs = \psi_i$.
To compare the difference between them, below we regard subsets of $H_\Fh$ as subsets of $\e_\Fh(H_\Fh)$, returning to the initial identification.
Therefore $H_Q$ of $H_\Fh$  was transformed to $H_Q \sm T_i$ and $H_R$ to $H_R \cup D^-_i$; the rest of $H_\Fh$ reminded the same. 
Thus, topologically, $\sigma_i$ has just moved the puncture point $q_i$ on $\rs$ across the loop $\ld'$ and accordingly $q_i$ has moved from on the boundary of $H_Q$  to to the boundary of  $H_R$ across $\ell$, while  $\sigma_i$  fixes $H_\Fh \sm (H_Q \cup H_R)$.

For different $i, j \in \{1,2,\dots, h\}$, since $V_i$ and $N(E_j)$ are disjoint, thus $T_i$ and $N(E_j) \sm r$ are disjoint; then the homeomorphisms $\sigma_i\cn \hb \to \hb$  $(i \in \{1,2, \dt, h\})$ have disjoint support.
Then let $\sigma\cn \hb \to \hb$ be their composition $\sigma_1 \cc \sigma_2 \cc \dots \sigma_h$, which is obviously homeomorphism.
Then $\sigma$ transforms $H_Q$ to $H_Q \sm (T_1 \cup T_2 \cup \dt \cup T_h)$, which contains no puncture points and $H_R$ to $H_R \cup (D^-_1 \cup D^-_2 \cup \dt \cup D^-_h)$, which contains $P_Q \cup P_R$.
Then $H_Q \sm (T_1 \cup T_2 \cup \dt \cup T_n)$ is topologically a 3-disk in $H'(\ld', +) \cap \h^3$ and its boundary intersects $D'_\ldp$ in a single 2-disk as the boundary of  $H_Q$ does.
Therefore, there is an isotopy of $\hb$ supported in small neighborhood of the 3-disk $H_Q \sm (T_1 \cup T_2 \cup \dt \cup T_n)$ that  moves its neighborhood $H_Q \sm (T_1 \cup T_2 \cup \dt \cup T_n)$ into $H'(\ld', -)$.
By this isotopy the entire union $H_R \cup H_Q$ has moved into $H'(\ld', -)$ while its complement $H_\Fh \sm (H_R \cup H_Q)$ remains fixed. 
Modify the homeomorphism $\sigma\colon \hb \to \hb$ by postcomposing with this isotopy. 
Thus we have  $(\sigma \cc \e_{\phi_1})\iv(D'_\ldp)$ is $\Dl_{\phi_1} \sm D_\ell$.

Combining with (I'), since $\Dl_{\phi_1} \sm D_\ell$ in $H_\Fh$ is bounded by $L_{\phi_1} \sm \ell$ in $H_\Fh$,  thus the multiloop $L_{\phi_1} \sm \ell$ is  isotopic to $L_{\psi \cc \phi_1}$ on $\Fc$.
There is an isotopy of $\hb$ that fixes $\rs$, preserves $\sigma(H_\Fh)$, in particular the surface $\sigma(\Fc) (\st \h^3)$, and moves $\sigma(L_{\phi_1} \sm \ell)$ into $D'_{\ldp}$.   
Finally extend $\psi\cn \rs \to \rs$ to the homeomorphism of $\h^3$ obtained by post-composing $\sigma$ with this isotopy.
Then we have $\pt \Dl_{\psi \cc \phi_1} = L_{\psi \cc \phi_1}$ (\,(II')\,).
\end{proof}
 
 \begin{figure}[htbp]
\begin{center}

\includegraphics[width=5in]{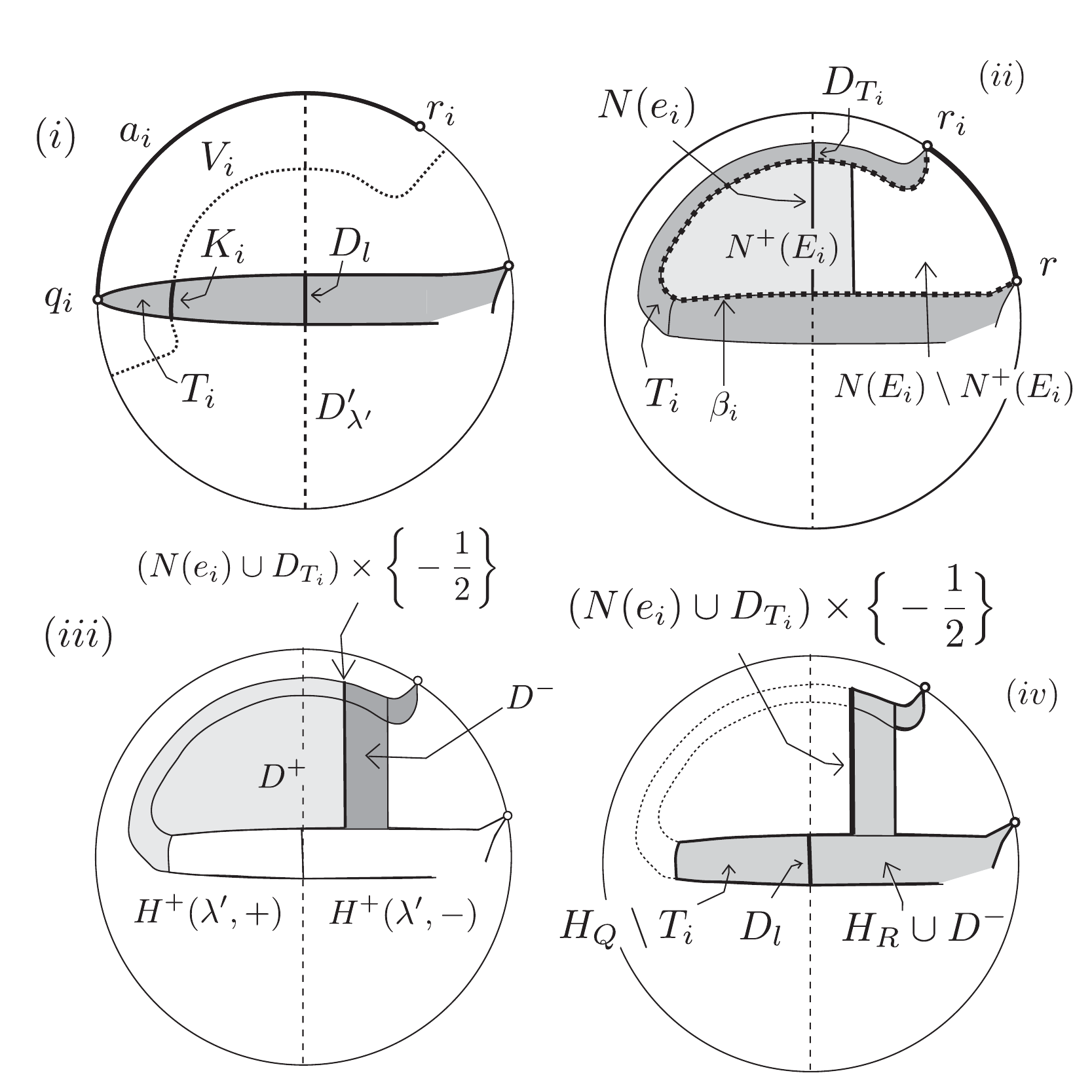}
\caption{The series of isotopes of $H_\Fh$ in $\hb$.}
\label{isotopies}
\end{center}

\end{figure}

\begin{proof}[Proof (Proposition \ref{reduced}).]
Let $\phi\colon \ol{\h^3} \to \ol{\h^3}$ be the homeomorphism obtained by Proposition \ref{phi}.
By Proposition \ref{phi} (ii), there is an isotopy  between $\ld'$ and $\e_\phi(\ld)$ in ${\rm Im}(\eta_\phi)$.
By the product structure of $\eta_\phi$, there is also an isotopy  between the loops $\eta_\phi(\ld \times \{0\}) = \e_\phi(\ld) $ and $\eta_\phi(\ld \times \{1\})$ in ${\rm Im}(\eta_\phi)$. 
Then there is an isotopy between the loops $\ld'$ and $\eta_\phi(\ld  \times \{1\})$ in ${\rm Im} (\eta_\phi)$.
By postcomposing with $\phi\iv$, we have an isotopy between the loops $\phi\iv(\ld')$ and $\phi\iv \cc \eta_\phi(\ld \times \{1\}) = \eta_\Fc(\ld \times \{1\}) = \mu'$ in $\phi\iv({\rm Im}(\eta_\phi)) = {\rm Im}(\eta_\Fc)$.
Those loops $\phi\iv(\ld')$ and $\mu'$ are on the punctured sphere $\rs \sm \e_\Fh(P_\Fc) = \eta_\Fc (\Fc \times \{1\})$.
By the obvious projection from ${\rm Im}(\eta_\Fc) \cong \Fc \times [0,1]$ to $\Fc \times \{1\}$, the isotopy between $\phi\iv(\ld')$ and $\mu'$ in ${\rm Im}(\eta_\Fc)$ induces to a homotopy between $\phi\iv(\ld')$ and $\mu'$ on $\rs \sm \e_\Fh(P_\Fc)$; 
Thus there is moreover an isotopy between them. 
Thus, via $f_\Fc$, this isotopy between  $\phi\iv(\ld')$ and $\mu'$  lifts to an isotopy between the multiloops $f_\Fc\iv(\phi\iv(\ld')) = f_\phi\iv(\ld')$ and $f_\Fc\iv(\mu')$ on $\Fc$.
In particular, their essential parts $\lf f_\phi\iv(\ld') \rf$ and $\lf f_\Fc\iv(\mu')\rf$ are also isotopic on $\Fc$.
By Proposition \ref{phi} (i),  there is an isotopy between  $\lf f_\phi\iv(\ld') \rf$ and the loop $\ld$ on $\Fc$.
Hence $\lf  f_\Fc\iv(\mu') \rf$ is a single loop.
\end{proof}


\section{A characterization of good structures by grafting}

Let $F$ be a sphere $\Fh$ with  $n$ punctures $p_1, p_2, \dt, p_n$. 
Let $C = (f, \rho_{id})$ be a good projective structure on  $F$, where $\rho_\id\colon \po(F) \to \psl$ is the trivial representation. 
Then the developing map $f\colon F \to \rs$ continuously extends to a branched covering map $f\colon \Fh \to \rs$.
Since $C$ is a good structure,  $f(p_1) =: q_1,  f(p_2) =: q_2, \dt, f(p_n) =: q_n$ are distinct points on $\rs$, and ${\it Supp}(C)$ is  the $n$-punctured sphere $\rs \sm \{q_1, q_2, \dt, q_n\} $.
Let  $f_0: F \to {\it Supp}(C) $ be  a homeomorphism such  that (its extension satisfies) $f_0(p_i) = q_i$ for all $i \in \{1,2, \dt, n\}$.
Then the pair $(f_0, \rho_{id})$ represents  an embedded projective structure on $F$ associated with $C$.
On the other hand, every embedded projective structure on $F$ associated with $C$ can be obtained in such a way. 
We  prove

\begin{proposition}\label{grafting}
Every good projective structure $C = (f, \rho_\id)$ on a punctured sphere $F$ can be obtained by grafting a basic structure associated with $C$ along a multiarc  (each arc of which connects different punctures of $F$). 
\end{proposition}

For each $i \in  \{1,2, \dt, n\}$, let $d_i$ be the ramification index of $f$ at the ramified  point $p_i$.
If $d_i > 1$, then $p_i$ is called a {\it proper ramification point}; 
if $d_i = 1$, then $p_i$ is called a {\it trivial ramification point}, i.e.\ $f$ is a local homeomorphism at $p_i$.
In the latter case, we may regard $p_i$ and $q_i$ merely as marked points. 
Let $d$ be the degree of $f$, i.e.\  the cardinality of $f\iv(x)$ for  $x \in  {\it Supp}(C)$. 
Let $\dl = d -1$ and $\dl_i = d_i -1$ for each $i \in \{1,2, \dt, n\}$.
Clearly we have $\dl \geq \dl_i \, (\geq 0)$ and, by the Riemann-Hurwitz formula, $2 \dl  = \Sigma_{i=1}^n \dl_i ~(\in 2 \N)$. 
Therefore we have

\begin{equation}
\displaystyle 2 \max_{1 \leq i \leq n} \dl_i  \leq \sum_{i=1}^n \dl_i. \label{RH}
\end{equation}

\begin{lemma}\label{multiarc}
There is a multiarc $A$ on $F$ such that\\
(i) each arc of $A$ connects distinct punctures of $F$, and \\
(ii) for each $i  \in \{ 1,2, \dt, n \}$, there are exactly $\dl_i $ arcs of $A$ ending at $p_i$. 
\end{lemma}

The following claim implies Lemma \ref{multiarc} (c.f. \cite{Kuperberg-96}):
\begin{claim}\label{n-gon}
Let $X$ be an $n$-gon,  and let $e_1, e_2, \dt, e_n$  denote its edges in a cyclic order. 
For each edge $e_i$, choose $\dl_i$ distinct marked points in its interior. 
Then, there exists a multiarc $A'$ properly embedded in $X$ such that  each arc of $A'$ connects marked points on different edges of $X$ and each marked point is an end of exactly  one arc of $A'$. 
\end{claim}

\begin{figure}[htbp]

\label{nested}
\begin{center}
\includegraphics[width=3in]{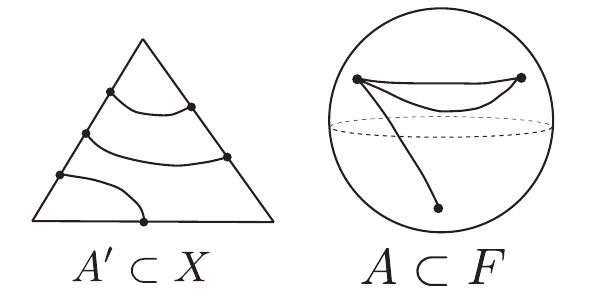}
\caption{The multiarcs $A$ and $A'$ for $(\dl_1, \dl_2, \dl_3) = (1,2,3)$.}
\label{multiarc.pdf}
\end{center}
\end{figure}

\begin{proof}[Proof of Lemma \ref{multiarc} with Claim \ref{n-gon} assumed] (See Figure \ref{multiarc.pdf}.)
For each $i \in \{1,2, \dt, n\}$, let $e'_i$ be the closed arc contained in the interior of  the edge $e_i$ so that $e_i'$ contains all marked points on $e_i$.
Consider the quotient space $X/\!\!\sim$ obtained by collapsing each $e'_i$ to a single point $r_i$  so that the end points of  the multiarc $A'$ on each $e_i$ are identified. 
Embed $X/\!\!\sim(\cong \D^2)$ into the sphere $\Fh$ so that $r_i$ maps to $p_i$ for each $i \in \{1,2, \dt n\}$.
Via this embedding, $A'/\!\!\sim$ realizes the desired multiarc $A$ on $F$. 
\end{proof} 

\begin{proof}[Proof of Claim \ref{multiarc}]
We prove this claim by induction on  the even non-negative integer  $\Sigma \dl_i$.
Suppose that the lemma holds if $(\dl_1, \dl_2, \dt, \dl_n)$ satisfies $\Sigma \dl_i = 2(k -1)$ for some  $k \in \Z_{> 0}$.
If $(\dl_1, \dl_2, \dt, \dl_n)$ satisfies $\Sigma_{i=1}^{n} \dl_i = 2k$, then 
without loss of generality, we can in addition assume that $\dl_1 = \max_{1 \leq i \leq n} \dl_i$.
Let $m$ be the minimal integer in $\{2, 3, \dots, n\}$ with $\dl_m \neq 0$. 
Then there are adjacent marked points on $e_1$ and $e_m$ (along the circle $\pt X$).
Thus pick an arc $\ap$ properly embedded in $X$ connecting these marked points; then that a component of $X \sm \ap$ contains no marked points.
Therefore it suffices to find a multiarc for the reduced $n$-tuple $$(\dl_1 -1, 0, 0,\dt, 0,\dl_m -1, \dl_{m+1}, \dl_{m +2}, \dt, \dl_n).$$
For this reduces $n$-tuple, we have $$(\dl_1 -1) + 0 + \dt + 0 + (\dl_m -1) + \dl_{m+1} + \dl_{m +2} + \dt + \dl_n = 2(k-1).$$
In addition, using Inequality (1) and also from the way the inequality is obtained, we see that 
$$2 \,max \{\dl_1 -1, 0, 0,\dt, 0,\dl_m -1, \dl_{m+1}, \dl_{m +2}, \dt, \dl_n \} \leq 2(k-1)$$
(if more than one $i \in \{1,2, \dots, n\}$ realizes $\max_{1 \leq i \leq n} \dl_i$, then consider if there are more than two non-zero $\dl_i$ or not).
Therefore, by the induction hypothesis, there is a multiarc $A$ on $X$ for the reduces $n$-tuple.
Then we can in addition assume that $A$ is contained in $X \sm \ap$.
Then $\ap \sq A$ is indeed the desired multiarc on $X$ for the original n-tuple $(\dl_1, \dl_2,  \dt, \dl_n)$.
\end{proof}

\begin{proposition}\label{TopEquiv}
Let $f_1, f_2\colon \Fh \to \rs$ be  branched covering maps, such that, for each $i \in \{1,2, \dt n\}$,  $~p_i$ is a ramification point of both $f_1$ and $f_2$ over $q_i$ with the ramification index $d_i$ and that $f_1$ and $f_2$ have no other proper ramification points. 
Then $f_1$ and $f_2$ are topologically equivalent, i.e.\  there are homeomorphisms $\phi\colon  \Fh \to \Fh$ and $\phi'\colon \rs \to \rs$ such that\\
(i) $\phi(p_i) = p_i$ and $\phi'(q_i) = q_i$ for all $i \in \{1,2, \dt, n\}$ and \\
(ii) $\phi' \cc f_1 = f_2 \cc \phi$.  
\end{proposition}

\begin{proof}
Without loss of generality, we can assume that $d_1, d_2, \dt, d_k > 1$ and $d_{k+1} = d_{k+2} = \dt = d_n =1$ for some integer $k \in \{1,2,\dt, n\}$. 
For each $j =1,2$, let $\mC_j$ be the complex structure on $\Fh \cong \s^2$ obtained by pulling back the complex structure on $\rs$ via $f_j$.
Then $f_j\colon (\Fh, \mC_j) \to \rs$ is a meromorphic function.
By the uniformization theorem, $f_j$ is conformally equivalent  to a rational function, i.e.\  there exist a rational function $\tau_j\colon \rs \to \rs$ and a conformal map $\psi_j\colon (\Fh, \mC_j) \to \rs$ such that $f_j = \tau_j \cc \psi_j$. 
Then,  for each $i \in \{1,2, \dt, k\}$,  $~\psi_j(p_i)$ is the ramification point of $\tau_j$ over $q_i$ with the ramification index $d_i$, and for each $i \in \{k+1, k+2, \dt, n\}$, $\psi_j(p_i)$ is the trivial ramification point of $\tau_i$.

By Theorem \ref{connected}, there is a path $\tau_t ~(t\in [1,2])$ connecting $\tau_1$ to $\tau_2$  in  the space $\mR(d_1,d_2, \dt, d_k)$ of rational functions defined in \S \ref{Hurwitz}.
Along $\tau_t$, the (proper) ramification points of $\tau_t\cn \rs \to \rs$ continuously move on the source sphere $\rs$ without hitting each other; 
similarly the branched points of $\tau_t$ continuously move on the target sphere $\rs$ without hitting each other. 
Thus, for each $i \in \{1,2,\dt  k\}$, there is a unique closed curve $Q_i(t) ~(t \in [1,2])$ on the target sphere  such that $Q_i(1) = Q_i(2)= q_i$ and $Q_i(t)$ is a branched point of $\tau_t$ for all $t \in [1,2]$.
Then $Q_1(t), Q_2(t), \dt, Q_n(t)$ are the branched points of $\tau_t$  for all $t \in [1,2]$.
Similarly, for each $i \in \{1,2,\dt, k\}$, we have a unique ($not$ necessarily closed) curve $P_i(t)~ (t \in [1,2])$ on the source sphere, such that $P_i(t)$ is the ramification point of $\tau_t$ over $Q_i(t)$ with the ramification index $d_i$ for each $t \in [1,2]$.
Then $P_1(t), P_2(t), \dt, P_k(t)$ are the branched points of $\tau_t$ for each $t \in [1,2]$.
In particular $P_i(1) = \psi_1(p_i)$ and $P_i(2) = \psi_2(p_i)$ for each $i \in \{1,2, \dt, k\}$.

For each $i  \in \{k+1, k+2, \dt, n\}$, pick a path $P_i(t)$ on the source sphere $\rs$ connecting $\psi_1(p_i)$ to $\psi_2(p_i)$ such that, for each $t \in [1,2]$, $~P_1(t), P_2(t), \dt, P_n(t)$ are different points on the source sphere.
For each $i \in \{k+1, k+2, \dt, n\}$, let $Q_i(t) = \tau_t(P_i(t))~(t \in [1,2])$,  a path  on the target sphere $\rs$.
Then $Q_i(t)$ is a closed path based at $q_i$.
We can in addition assume that $Q_1(t), Q_2(t),\linebreak[2] \dt, Q_n(t)$ are different points on $\rs$, by perturbing the paths $P_i(t)$ with $k+1 \leq i \leq n$ if necessarily. 

Thus the paths $Q_i(t)$ obviously induce an isotopy between the branched points of $\tau_1$ and of $\tau_2$ on $\rs$. 
This isotopy of the branched points extends to an isotopy of the target sphere, denoted by  $\xi'_t\colon \rs \to \rs ~ (t \in [1,2])$. 
Since $\tau_t\cn \rs \to \rs$ is  continuous in $t$, the isotopy $\xi'_t$ on the target sphere uniquely lifts to an isotopy of the source sphere, $\xi_t\colon \rs \to \rs ~(t \in [1,2])$,  via $\tau_t$. 
Then $\tau_t \cc \xi_t = \xi'_t \cc \tau_1$ for all $t \in [1,2]$, and $\xi_t(p_i) = P_i(t)$ for each $i \in \{1,2, \dots, n\}$.
In particular  $\tau_2 \cc \xi_2 = \xi'_2 \cc \tau_1$.
Therefore we have  
$$f_2 \cc (\psi_2\iv \cc \xi_2 \cc \psi_1) = \tau_2 \cc \xi_2 \cc \psi_1 = \xi_2' \cc \tau_1 \cc \psi_1 =\xi'_2 \cc f_1.$$
Here $\xi'_2$ is a homeomorphism of $\rs$ fixing $q_i$, and  $\psi_2\iv \cc \xi_2 \cc \psi_1\colon \Fh \to \Fh$ is a homeomorphism fixing $p_i$  for all $i \in \{1,2, \dt, n\}$.
Thus $f_1$ and $f_2$ are topologically equivalent. 
\end{proof}
\begin{proof}[Proof (Proposition \ref{grafting})]
Let $C_0 = (f_0, \rho_{id})$ be an embedded projective structure on $F$ associated with $C$.
Let $A$ be the multiarc on $F$ obtained by Lemma \ref{multiarc}.
By Lemma \ref{multiarc} (i), we can graft $C_0$ along $A$; set $C_1 = (f_1, \phi_{id})$ to denote  $Gr_A(C_0)$. 
Then,  for all $i \in \{1,2, \dt, n\}$, we have $f(p_i) = f_1(p_i) = q_i$ and, by Lemma \ref{multiarc} (ii), the ramification index of $f$ and $f_1$ are both $d_i$ at $p_i$.
Therefore, by Proposition \ref{TopEquiv}, there are homeomorphisms $\phi\colon \Fh \to \Fh$ fixing all $p_i$ and a homeomorphism $\phi'\colon \rs \to \rs$ fixing all $q_i$, such that $\phi' \cc f = f_1 \cc \phi$.
Therefore $f = {\phi'}\iv \cc f_1 \cc \phi$.
For a homeomorphism $e: F \to  {\it Supp}(C) $ and an appropriate multiarc $N$ on $F$, let $Gr_N(\psi)$ denote the developing map of $Gr_N((e, \rho_\id))$, where $(e, \rho_\id)$ represents an embedded projective structure on $F$.  
In particular $f_1 = Gr_A(f_0)$.
Then
\begin{eqnarray*}
f = {\phi'}\iv \cc Gr_A(f_0) \cc \phi &=&  {\phi'}\iv \cc Gr_A (f_0 \cc \phi) \\
&=& Gr_{{\phi'}\iv(A)}({\phi'}\iv \cc f_0 \cc \phi).
\end{eqnarray*} 
Thus $({\phi'}\iv \cc f_0 \cc \phi,~ \rho_{id})$ is an embedded projective structure on $F$, and $C = (f, \rho_{id})$ is obtained by grafting this embedded projective  structure along the multiarc ${\phi'}\iv(A)$.
\end{proof}


As an immediate corollary of Proposition \ref{grafting}, we obtain:
\begin{proposition}\label{GoodHoledSphere}
Let $C$ be a good projective structure on a holed sphere $F$.
Then $C$ can be obtained by grafting an embedded structure  associated with $C$ along a multiarc  (each arc of which connects different boundary components of $F$).
\end{proposition}

\section{The proof of the main theorem}
Recall that $S$ is a closed orientable surface of genus $g$, $~\Gm$ is a real Schottky group of rank $g > 1$, and $\rho\colon \po(S) \to \Gm \st \psl$ is an epimorphism. 
\begin{theorem}\label{main}
Every Schottky structure $C = (f,\rho)$ on $S$  can be obtained by grafting a uniformizable Schottky structure with holonomy $\rho$ (along a multiloop on $S$). 
\end{theorem}
\nin{\it Remark:} By Lemma \ref{basicSchottky}, a uniformizable projective structure with holonomy $\rho$ is $\Og/\Gm$ with some marking.
\begin{proof}
In $\S \ref{HoledSphere}$, we constructed the multiloops $M$ and $M'$  on $(S, C)$ and $\Og/\Gm$, respectively;
the multiloops $\Mt$ and $\Mt'$ are the total lifts of $M$ and $M'$ to $\St$ and $\Og$, respectively. 
Then $M$ decomposes $(S, C)$ into projective structures $(F_i, C_i)_{i=1}^n$, where $F_i$ are components of $S \sm M$ and $C_i = C|F_i$. 
Then, for each $i \in \{1,2, \dt, n\}$,  by Theorem \ref{M},  $(F_i, C_i)$ is a good holed sphere fully supported on a component of $\Og \sm \Mt'$ that corresponds to $C_i$ via the developing map $f$.
By Proposition \ref{GoodHoledSphere}, $C_i = Gr_{A_i}(C_{0,i})$, where $C_{0,i}$ is an embedded structure on the holed sphere $F_i$ with ${\it Supp}(C_{0,i}) = {\it Supp}(C_i)$ and $A_i$ is a multiarc on $F_i$ such that each arc of $A_i$ connects distinct boundary components of $F_i$.
For each loop $\ell $ of $M$, every lift $\lt$ of $\ell$ to $\St$ is a loop covering a loop of $\Mt'$ via $f$.
Let $d_\ell$ be the degree of this covering map $f|_{\lt}$.
The loop $\ell$ corresponds to exactly two boundary components of  $F_1 \sq F_2 \sq \dt \sq F_n$. 
Then, on each of these two boundary components, there are exactly $d_\ell - 1$ arcs of $A_1 \sq A_2 \sq \dt \sq A_n$ ending. 
Therefore we can isotope $A_i$ on $F_i$ for all $i \in \{1,2, \dt, n\}$, keeping the endpoints of $A_i$ of $\pt F_i$, so that the endpoints of $A_1, A_2, \dt, A_n$ pair up  and  $\cup A_i$ is a multiloop $A$ on $S$.
(There are infinitely many non-isotopic choices for $A$: We can Dehn twist $A$ along a boundary component of $F_i$.)

\begin{lemma}\label{lemma}
(i) With respect to $S = \cup_{i =1}^n F_i$\,, the union of the embedded projective structures $C_{0,i}$  on $F_i$ $(i = 1,2,\dots, n)$ is a uniformizable Schottky structure on $S$ with holonomy $\rho$.\\
(ii) If $\ap$ is a component of the multiloop $A$, then $\rho(\ap)$ is loxodromic,  i.e.\ $\rho(a) \neq 1$.  
\end{lemma}

\begin{proof}
(i).
Assume that $C_i$ and $C_j$ are adjacent components of $C \sm M$, sharing a boundary component $\ell $.
Then ${\it Supp}(C_i)$ and ${\it Supp}(C_j)$ are adjacent components of $\Og \sm \Mt$ (up to an element of $\Gm$),  sharing a boundary component $f(\lt)$, where $\lt$ is a lift of $\ell $ to $\St$.
Since $C_{0,i}$ and $C_{0,j}$ are  the canonical projective structures embedded in $\rs$, respectively, we can identify the boundary components of  $C_{0,i}$ and $C_{0,j}$ corresponding to $\ell $. 
Similarly identify  all corresponding boundary components of  $C_{0,i} ~(i=1,2, \dt, n)$ and obtain a projective structure $C_0$ on $S$.
Let $\Ct_0 = (f_0, \rho_{id})$ be the projective structure on $\St$ obtained by lifting $C_0$ to $\St$, where $f_0$ is a $\rt$-equivariant immersion from $\St$ to $\rs$.
Since $f_0$  embeds each components of $\St \sm \Mt$ onto a component of $\Og \sm \Mt'$,  thus it
 is a $\rt$-equivariant embedding onto $\Og$.
 Therefore $(S, C_0)$ is a uniformizable Schottky structure with holonomy  $\rho$.

(ii). Let $\ap$ be  a loop of $A$.
Then $M$ decomposes $\ap$ into subarcs $a_1, a_2, \dt, a_m$,  so that  $\ap = a_1 \cup a_2 \cup \dt \cup a_m$.
Let $\td{\ap}$ be a lift of $\ap$ to $\St$.
Then each $a_j ~(j =1,2, \dt, m)$ is an arc properly embedded in $F_i$ for some $i \in \{1,2, \dt, n\}$,  connecting different boundary components of $F_i$. 
Therefore, for each component $P$ of $\St \sm \Mt$, either $\td{\ap}$ is disjoint from $P$ or $\td{\ap}$ intersects $P$ in a single arc connecting different boundary components of $P$.
Thus $\td{\ap}$ is a biinfinite simple curve properly embedded in $\St$, and then $\ap$, as an element in $\po(S)$, translates $\St$  along $\td{\ap}$. 
Therefore $~\rho(\ap)$ is loxodromic.
\end{proof} 

Let $C_0$ be the uniformizable projective structure  $\cup_{i = 1}^n C_{0, i}$ on $S$ given by Lemma \ref{lemma} (i).
We will show that $C$ is obtained by grafting $C_0$ along $A$.
To complete the proof, we will see that the grafting $Gr_A$ on $C_0$ ``commutes'' with the decomposition $C_0 = \cup_{i=1}^n G_{0,i}$ as 
  $$\cup_{i=1}^n Gr_{A_i}(C_{0,i}) = Gr_A(C_0).$$ 
For each $j \in \{1,2, \dt, m\}$, the arc $a_j$ is property embedded in $C_{0,i}$ with some $i \in \{1,2, \dt, n\}$;
let $b_j, c_j$ denote the boundary components of $C_{0,i} = (F_i, C_{0,i})$ connected by  $a_j$.

Clearly the developing map $dev(C_{0,i})$ embeds $C_{0,i}$ isomorphically onto ${\it Supp}(C_{0,i})$, which is a component of $\Og \sm \Mt'$.
Via this isomorphism, (the images of) $b_j$ and $c_j$ bound a projective  cylinder $Y_j$ in $\rs$, and $Y_j$ contains the arc $a_i$ connecting connecting $b_j$ and $c_j$.
By the definition of graftings,  $Gr_{a_j}(C_{0,i})$ is obtained by  appropriately identifying the boundary arcs of  $Y_j \sm a_j$ and $C_{0,i} \sm a_j$ corresponding to $a_j$.

Suppose that $a_{j_1}$ and $a_{j_2}$ ($j_1, j_2 \in \{1,2, \dt, m\}$) are adjacent arcs in $\ap$, sharing an endpoint $v$.
Similarly, for each $k =1,2$, we have $a_{j_k} \st C_{0, i_k}$ for some $i_k \in  \{1,2, \dt, n\}$;  let $Y_{j_k}$ be its corresponding projective cylinder, as above, containing $a_{j_k}$. 
Then $C_{0,i_1}$ and $C_{0,i_2}$  are isomorphic to some adjacent components of $C_0 \sm M$, sharing the boundary component containing $v$.
Accordingly, $Y_{j_1}$ and $Y_{j_2}$ are also adjacent cylinders embedded in $\rs$ (up to an element of $\Gm$).
Thus we can identify the corresponding boundary components of  $Y_{j_1}$ and $Y_{j_2}$.
Similarly identify all corresponding boundary components of  $Y_j$  ($j=1,2, \dt, m$) and obtain a projective torus $T = \cup_j Y_j$.
Then the loop $\ap = \cup_j \,a_j$ is naturally embedded in $T$.

\begin{figure}[htbp]\label{AddCylinders}
\begin{center}
\includegraphics[width=4in]{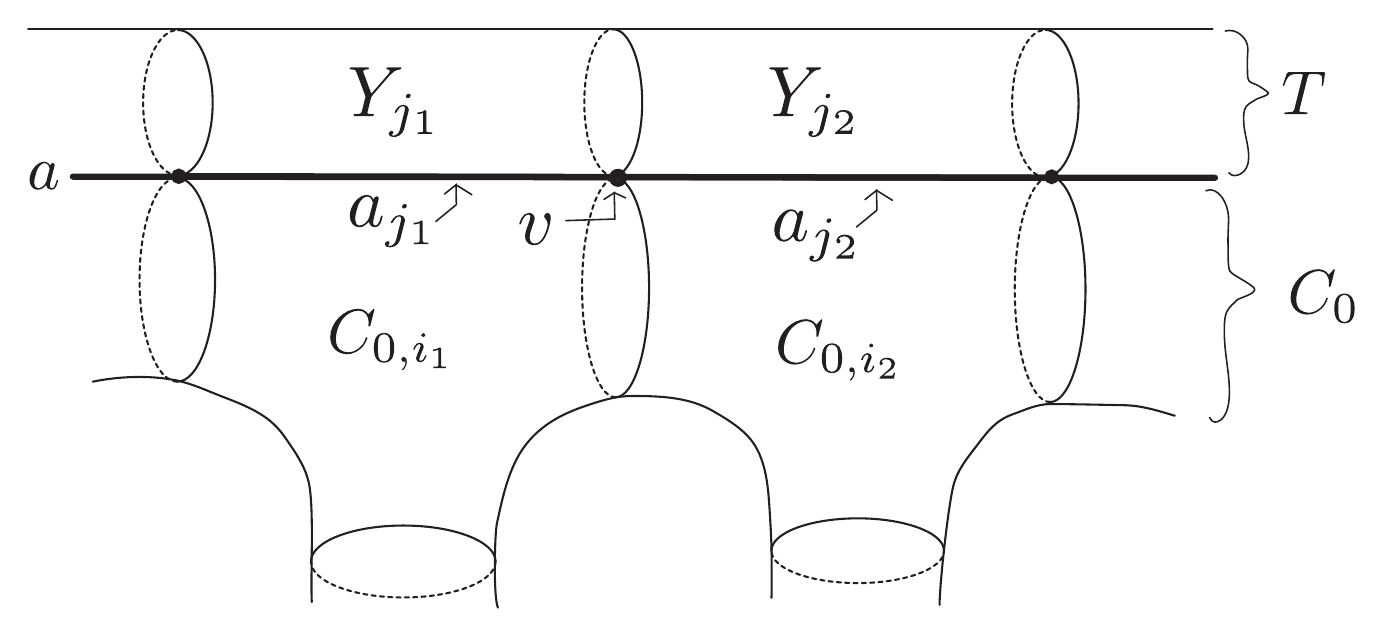}
\caption{}
\end{center}
\end{figure}

We  show that $T$ is a hopf torus. 
Let $N$ be the union of boundary components of $Y_j$ in $T$. 
Then $N$ is the multiloop splitting $T$ into $Y_1, Y_2, \dots, Y_m$.
The homotopy class of $\ap$ generates an infinite cyclic subgroup $\langle \ap \rangle$ of $\po(S)$. 
Let $\td{T}$ be the projective  cylinder coving $T$ with the covering transformation group $\langle \ap \rangle$. 
Then $\td{\ap} \st \St$ is also embedded in $\Tt$.
Let $\td{N}$ denote the total lift of $N$ to $\Tt$. 
Then $\td{\ap}$ transversally intersects each loop of  $\Nt$ in a single point.
Then $\td{\ap}$ is embedded into $\Og$ under the developing map of $C_0$ and, therefore, of $T$.
Thus $dev(T)$ embeds $\Nt$ onto the multiloop on $\Og$ consisting  of the loops of $\Mt'$ interesting $\apt$, and therefore
 $dev(T)$ is an embedding onto $\rs$ minus the fixed points of $\rho(\ap)$.
Hence $T$ is a Hopf torus.  

We adapt this observation for a single loop to the the entire multiloop $A$.
Set $A = \ap_1 \sq \ap_2 \sq \dt \sq \ap_r$,
where $\ap_1, \ap_2, \dt, \ap_r$ are the loops of $A$.
Further decomposing those loops into the components of $A_1,  A_2, \dt,  A_n$ similarly to the proof Lemma \ref{lemma} (ii), we set $A = a_1 \cup a_2 \cup \dt \cup a_m$.
Then, for each $j \in \{1,2, \dt, m\}$, $~a_j$ is an arc properly embedded in $F_{i(j)}$ with some $i(j) \in \{1,2, \dt, n\}$. 
Similarly, let $Y_j$ be the projective cylinder for the grafting of $C_{0, i(j)}$ along the arc $a_j$, so that  $Gr_{a_j}(C_{0,i(j)}) = (C_{0, i(j)} \sm a_j) \cup (Y_j \sm a_j)$.
Then, for each $i \in \{1,2, \dt, n\}$, we have
$$Gr_{A_i}(C_{0,i}) = (C_{0,i} \sm A_i) \cup (\,\sq_j \,\{\,Y_j \sm a_j \}\,),$$
where the second union runs over all $ j \in  \{1,2, \dt, m\}$ with  $a_j \st A_i$.
For each $k \in \{1,2, \dt r\}$, let $T_k$ denote the Hopf torus associated with $\ap_k$\,;
then  $T_k = \cup\, Y_j $, where the union runs over all $j \in \{1,2,\dots, m\}$ with $a_j \st \ap_k$.
Therefore   
\begin{eqnarray*}
C  &=& \cup_{i =1}^n C_i\\
&=& \cup_{i=1}^n Gr_{A_i}(C_{0,i})\\
&=& \cup_{i=1}^n [~(C_{0,i} \sm A_i) \cup (\,\sq\,\{\, Y_j \sm a_j~ | ~a_j \st A_i\}\,)~ ]\\
&=& [\,\cup_{i=1}^n \,(C_{0,i} \sm A_i)\,] \,\sq \,[\,\cup_{j=1}^m (Y_j \sm a_j)\,]\\
&=& (C_0 \sm A) \cup  [\,(T_1 \sm \ap_1) \cup (T_2  \sm \ap_2) \cup \dt \cup (T_r \sm \ap_r)\,]\\
&=& Gr_A(C_0)~.
\end{eqnarray*}
\end{proof}

\bibliographystyle{plain}
\bibliography{Reference}

\begin{thebibliography}{10}

\bibitem{Baba-10}
Shinpei Baba.
\newblock A {S}chottky decomposition theorem for complex projective structures.
\newblock {\em Geom. Topol.}, 14(1):117--151, 2010.

\bibitem{Fulton-69}
William Fulton.
\newblock Hurwitz schemes and irreducibility of moduli of algebraic curves.
\newblock {\em Ann. of Math. (2)}, 90:542--575, 1969.

\bibitem{Gallo-Kapovich-Marden}
Daniel Gallo, Michael Kapovich, and Albert Marden.
\newblock The monodromy groups of {S}chwarzian equations on closed {R}iemann
  surfaces.
\newblock {\em Ann. of Math. (2)}, 151(2):625--704, 2000.

\bibitem{Goldman-87}
William~M. Goldman.
\newblock Projective structures with {F}uchsian holonomy.
\newblock {\em J. Differential Geom.}, 25(3):297--326, 1987.

\bibitem{Hejhal-75}
Dennis~A. Hejhal.
\newblock Monodromy groups and linearly polymorphic functions.
\newblock {\em Acta Math.}, 135(1):1--55, 1975.

\bibitem{Hubbard-81}
John~H. Hubbard.
\newblock The monodromy of projective structures.
\newblock In {\em Riemann surfaces and related topics: {P}roceedings of the
  1978 {S}tony {B}rook {C}onference ({S}tate {U}niv. {N}ew {Y}ork, {S}tony
  {B}rook, {N}.{Y}., 1978)}, volume~97 of {\em Ann. of Math. Stud.}, pages
  257--275. Princeton Univ. Press, Princeton, N.J., 1981.

\bibitem{Kamishima-Tan-92}
Yoshinobu Kamishima and Ser~P. Tan.
\newblock Deformation spaces on geometric structures.
\newblock In {\em Aspects of low-dimensional manifolds}, volume~20 of {\em Adv.
  Stud. Pure Math.}, pages 263--299. Kinokuniya, Tokyo, 1992.

\bibitem{Kapovich-01}
Michael Kapovich.
\newblock {\em Hyperbolic manifolds and discrete groups}, volume 183 of {\em
  Progress in Mathematics}.
\newblock Birkh\"auser Boston Inc., Boston, MA, 2001.

\bibitem{Kuiper-50}
N.~H. Kuiper.
\newblock On compact conformally {E}uclidean spaces of dimension {$>2$}.
\newblock {\em Ann. of Math. (2)}, 52:478--490, 1950.

\bibitem{Kullkani-Pinkall-98}
Ravi~S. Kulkarni and Ulrich Pinkall.
\newblock A canonical metric for {M}\"obius structures and its applications.
\newblock {\em Math. Z.}, 216(1):89--129, 1994.

\bibitem{Kuperberg-96}
Greg Kuperberg.
\newblock Spiders for rank {$2$} {L}ie algebras.
\newblock {\em Comm. Math. Phys.}, 180(1):109--151, 1996.

\bibitem{Liu-Osserman}
Fu~Liu and Brian Osserman.
\newblock The irreducibility of certain pure-cycle {H}urwitz spaces.
\newblock {\em Amer. J. Math.}, 130(6):1687--1708, 2008.

\bibitem{Luft-78}
E.~Luft.
\newblock Actions of the homeotopy group of an orientable {$3$}-dimensional
  handlebody.
\newblock {\em Math. Ann.}, 234(3):279--292, 1978.

\bibitem{Maskit-69}
Bernard Maskit.
\newblock On a class of {K}leinian groups.
\newblock {\em Ann. Acad. Sci. Fenn. Ser. A I No.}, 442:8, 1969.

\bibitem{Matsuzaki-Taniguchi-98}
Katsuhiko Matsuzaki and Masahiko Taniguchi.
\newblock {\em Hyperbolic manifolds and {K}leinian groups}.
\newblock Oxford Mathematical Monographs. The Clarendon Press Oxford University
  Press, New York, 1998.
\newblock Oxford Science Publications.

\bibitem{Oertel-02}
Ulrich Oertel.
\newblock Automorphisms of three-dimensional handlebodies.
\newblock {\em Topology}, 41(2):363--410, 2002.

\bibitem{Stallings-83}
John~R. Stallings.
\newblock Topology of finite graphs.
\newblock {\em Invent. Math.}, 71(3):551--565, 1983.

\bibitem{Sullivan-Thurston-83}
Dennis Sullivan and William Thurston.
\newblock Manifolds with canonical coordinate charts: some examples.
\newblock {\em Enseign. Math. (2)}, 29(1-2):15--25, 1983.

\bibitem{Tan-88}
Ser~Peow Tan.
\newblock {\em Representations of surface groups into ${\rm PSL}(2,
  \mathbb{R})$ and geometric structures}.
\newblock PhD thesis, 1988.

\bibitem{Thompson_10}
Joshua Thompson.
\newblock Grafting real complex projective structures with fuchsian holonomy.
\newblock {\it preprint, arXiv:arXiv:1012.2194}.

\bibitem{Thurston-97}
William~P. Thurston.
\newblock {\em Three-dimensional geometry and topology. {V}ol. 1}, volume~35 of
  {\em Princeton Mathematical Series}.
\newblock Princeton University Press, Princeton, NJ, 1997.
\newblock Edited by Silvio Levy.

\bibitem{Volklein-96}
Helmut V{\"o}lklein.
\newblock {\em Groups as {G}alois groups}, volume~53 of {\em Cambridge Studies
  in Advanced Mathematics}.
\newblock Cambridge University Press, Cambridge, 1996.
\newblock An introduction.

\end{thebibliography}

\end{document}